\documentclass[a4paper,reqno]{amsart}

\usepackage{amsmath,amsthm,amssymb}
\usepackage{mathrsfs}
\usepackage[shortlabels]{enumitem}
\usepackage{graphicx}
\usepackage[font=small]{caption}
\usepackage{xcolor}
\usepackage{url}

\newcommand{\N}{\mathbb{N}}
\newcommand{\R}{{\mathbb{R}}}
\newcommand{\C}{{\mathbb{C}}}

\newcommand{\dd}{{{\rm d}}}
\newcommand{\ii}{{\rm i}}
\newcommand{\e}{{\rm e}}

\newcommand{\ov}{\overline}
\newcommand\wt{\widetilde}
\newcommand\wh{\widehat}

\newcommand{\la}{\lambda}

\newcommand{\eps}{\varepsilon}

\newcommand{\spd}{\sigma_{\rm disc}}

\newcommand{\se}[1]{\sigma_{\rm e#1}}
\newcommand{\spp}{\sigma_{\rm p}}

\newcommand{\essinf}{\operatorname*{ess \,inf}}

\newcommand{\Dom}{{\operatorname{Dom}}}

\newcommand{\Ran}{{\operatorname{Ran}}}

\renewcommand{\Re}{\operatorname{Re}}
\renewcommand{\Im}{\operatorname{Im}}

\newcommand{\sgn}{\operatorname{sgn}}
\newcommand{\supp}{\operatorname{supp}}

\newcommand{\loc}{\mathrm{loc}}
\newcommand{\BigO}{\mathcal{O}}
\newcommand{\diag}{\operatorname{diag}}

\newcommand{\Rd}{\mathbb{R}^d}

\newcommand{\LtlocOm}{{L^2_{\loc}(\Omega)}}

\newcommand{\LOm}{{L^2(\Omega)}}

\newcommand{\LolocOm}{{L^1_{\loc}(\Omega)}}

\newcommand{\CcOm}{{C_0^{\infty}(\Omega)}}

\newcommand{\ls}{\lesssim}
\newcommand{\gs}{\gtrsim}

\theoremstyle{plain}

\newtheorem{theorem}{Theorem}[section]
\newtheorem{lemma}[theorem]{Lemma}

\newtheorem{proposition}[theorem]{Proposition}
\newtheorem{corollary}[theorem]{Corollary}

\theoremstyle{definition}
\newtheorem{example}[theorem]{Example}
\newtheorem{remark}[theorem]{Remark}
\newtheorem{asm-sec}[theorem]{Assumption}

\newcommand\cA{\mathcal A}
\newcommand\cB{\mathcal B}

\newcommand\cD{\mathcal D}

\newcommand\cH{\mathcal H}

\newcommand\cS{\mathcal S}

\newcommand\cV{\mathcal V}
\newcommand\cW{\mathcal W}

\usepackage{mathtools}
\usepackage{hyperref}

\newcommand{\Loneloc}{L^1_{\operatorname{loc}}}

\newcommand{\iprod}[2]{\langle #1,#2 \rangle}
\newcommand{\Coo}{C_0^\infty}

\newcommand{\norm}[1]{\lVert#1\rVert}

\newcommand{\abs}[1]{\lvert#1\rvert}

\newcommand{\Aa}{\mathcal{A}}

\newcommand{\Dd}{\mathcal{D}}

\newcommand{\Bb}{\mathcal{B}}

\newcommand{\Hh}{\mathcal{H}}
\newcommand{\Ww}{\mathcal{W}}

\newcommand{\Ss}{\mathcal{S}}

\newcommand{\Vv}{\mathcal{V}}

\newcommand{\Ltw}{L^2_{w}(\Omega)}

\newcommand{\re}{\operatorname{Re}}
\newcommand{\im}{\operatorname{Im}}

\newcommand\ec{\eps_{\rm crit}}

\mathtoolsset{showonlyrefs}

\title[Schr\"odinger operators in weighted spaces]{Schr\"odinger operators with accretive potentials in weighted spaces}

\begin{document}
	
\numberwithin{equation}{section}
	
\graphicspath{{../Figures/}}

\author{Borbala Gerhat}
\address[Borbala Gerhat]{Institute of Science and Technology Austria, Am Campus 1, 3400 Klosterneuburg, Austria}
\email{borbala.gerhat@ista.ac.at}

\author{Petr Siegl}
\address[Petr Siegl]{Institute of Applied Mathematics, Graz University of Technology, Steyrergasse 30, 8010 Graz, Austria}
\email{siegl@tugraz.at}



\thanks{B.~Gerhat has received funding from the European Union’s Horizon 2020 research and innovation
		programme under the Marie Sk\l odowska--Curie Grant Agreement No.~101034413, the Czech Science Foundation EXPRO Grant No.~0-17749X and the Swiss National Science Foundation Grant No.~169104. She further expresses her gratitude to the SEMP student exchange programme for their support and the Queen's University Belfast for their hospitality.}

\begin{abstract}
	We analyse Schr\"odinger operators with accretive potentials in weighted spaces. We find conditions on potentials and weights for which the Dirichlet realisation, introduced by generalised form methods, has non-empty resolvent set. We establish a domain and graph norm separation property, as well as sufficient conditions for the compactness and Schatten class of the resolvent. Moreover, we investigate the relation between discrete spectra and eigenfunctions of operators in standard and weighted spaces. As applications we extend results on the completeness of eigensystems of operators with accretive potentials from standard to weighted spaces and analyse operator matrices exhibiting a Schur dominance property, in particular, related to a wave equation with strong accretive damping.
\end{abstract}

\date{September 21, 2026}
\maketitle

\section{Introduction}
\label{sec:intro}

The spectral properties of Schr\"odinger operators
\begin{equation}
	T = -\Delta + V
\end{equation}
in the space $L^2(\Omega)$ with (possibly unbounded) complex potentials $V$ on an open set $\Omega \subset \Rd$ have been studied extensively, see e.g.~\cite{Caliceti-1980-75,Davies-2007,EE,Exner-1983-24,Helffer-2013-book,Kato-1966,Trefethen-2005}. These operators arise in several applications, ranging from superconductivity~\cite{Almog-2008-40,Almog-2013-365,Rubinstein-2010-195}, Bloch--Torrey equations \cite{Almog-2022-4,Grebenkov-2018-50}, hydrodynamics \cite{Almog-2021-241,Almog-2016-48,Gallagher-2009-2009},  optics with gains and losses~\cite{Dohnal-2016-57,Zezyulin-2012-85} to damped wave equations \cite{Arifoski-2020-52,Freitas-2018-264,Ikehata-2020-63,Arnal-Gerhat-Royer-Siegl-2026-arxiv} and many more.

In particular for accretive potentials $V: \Omega \to \C$ with $\Re V \geq 0$ an m-accretive Dirichlet realisation in $L^2(\Omega)$ was found in \cite{Brezis-1979-58}, see also~\cite{Kato-1978-5}, \cite[Sec.~VII.2]{EE} and \cite{Grinshpun-1994-124}. For more regular potentials which satisfy $V \in W^{1,\infty}_{\rm loc}(\ov \Omega)$, cf.~Assumption~\ref{asm:main.gn}, and
\begin{equation}\label{Q.asm.sep}
	\exists~\eps_\nabla \in \big[0,2-\sqrt2\big), \qquad  \exists~M_{\nabla} \geq 0, \qquad  
	|\nabla V| \leq \eps_\nabla |V|^\frac 32 + M_{\nabla} \quad \text{a.e.~in~} \Omega,
\end{equation}
the operator domain exhibits the separation property and is thus substantially more explicit. More precisely,
\begin{equation}
	\Dom(T) = \Dom(-\Delta_{\rm D}) \cap \Dom(V),
\end{equation}
where $\Dom(-\Delta_{\rm D})$ is the domain of the self-adjoint Dirichlet Laplacian in $L^2(\Omega)$, and the corresponding separation of the graph norm holds
\begin{equation}\label{T.norm.est}
\|T f\| + \|f\| \geq a_\nabla (\|\Delta f\| + \|V f\|+\|f\|), \qquad f \in \Dom(T),
\end{equation}
with $a_\nabla=a_\nabla(\eps_\nabla,M_\nabla)>0$; for details see~\cite{Evans-1978-80,Everitt-1978-79,Boegli-2017-42,Helffer-2019,Krejcirik-2017-221,Semoradova-2022-54} where also further extensions, e.g.~for an additional magnetic field, can be found.

The separation \eqref{T.norm.est} allows for reducing questions on the compactness and Schatten class of the resolvent of $T$ to the properties of the self-adjoint operator $-\Delta + |V|$, which can be further employed to investigate the completeness of the eigensystem of $T$, see~\cite{Almog-2015-40,Siegl-2012-86,Tumanov-2021-280}. Moreover, \eqref{T.norm.est} is also a key step in the analysis of the spectral convergence of domain truncations, see~\cite{Boegli-2017-42,Semoradova-2022-54}.

In this paper we analyse $T_w = -\Delta + V$ with accretive potential $V$ in a weighted space $L_{w}^2(\Omega)$. The potential satisfies a more general version of \eqref{Q.asm.sep} and we impose an admissibility condition on $w: \Omega \to (0,\infty)$ depending on the strength of $V$, see Assumption~\ref{asm:main.LM} for details, which allows for (super)-exponential weights if the potential is unbounded at infinity. The study of Schr\"odinger operators in weighted spaces dates back to Whittaker \cite{Whittaker-1914-33} and Sommerfeld \cite{Sommerfeld-1929}. Among many other instances, it appears in the spectral analysis of the bi-stable potential in quantum mechanics \cite{Razavy-1980-48}, the hypoelliptic Laplacian studied by Bismut and Lebeau \cite[Chap.~16]{Bismut-2008-167}, see also \cite{Mityagin-2021-22}, the Hill operator with a two-term potential \cite{Djakov-2007-242}, the Ornstein--Uhlenbeck operator \cite{Metafune-2002-196} or the Black--Scholes operator~\cite{Baaquie-2020-book}; a recent analysis of Schr\"odinger operators with non-negative potentials in $L^p$ spaces with exponentially decaying and growing weights can be found in \cite{Bailey-2021}.

It is crucial to observe that, due to the possible unboundedness of $w$ and/or $w^{-1}$, the connection between $T_w$ in $L_{w}^2(\Omega)$ and $T:=T_1$ in $L^2(\Omega)$ can be quite loose in general. Equivalently, this applies to the relation between $T$ and the unitary transformation $S := w^\frac12 T_w w^{-\frac12}$ of $T_w$ to $L^2 (\Omega)$. On the level of formal differential expressions, on the other hand, $-\Delta + V$ and the conjugation $w^\frac12(-\Delta+V)w^{-\frac12}$ seem to have a closer link since they can be related via an ``unbounded similarity transform'' or adding a complex rotation-free magnetic potential
\begin{equation}
	w^\frac12(-\Delta+V)w^{-\frac12} = \left(-\ii \nabla - A \right)^2 + V, \qquad A = - \frac \ii 2 \frac{\nabla w}{w},
\end{equation}
see also~\cite{Krejcirik-2019-51, Krejcirik-Nguyen-Raymond-2026-constantMF, Krejcirik-Nguyen-Raymond-2026}.

In order to illustrate non-trivial effects in the relation between $T$ and $T_w$, consider first the self-adjoint realisation of the one-dimensional
\begin{equation}
	T =-\partial_x^2 + |x|^\alpha, \qquad \alpha \geq 2,
\end{equation}
in $L^2(\R)$ and the weight $w (x)= \exp(x)$, $x \in \R$; see \cite{Mityagin-2017-272} for a thorough study of $S = w^\frac12 T_w w^{-\frac12}$ in this case. The spectra of $T$ and $T_w$ coincide, consist of discrete, simple and real eigenvalues and the corresponding eigenfunctions are complete both in $L^2(\R)$ and $L^2_{w}(\R)$.  Nevertheless, the eigenfunctions do not form a basis in $L^2_{w}(\R)$ and the norms of the corresponding spectral projections $P_k$ diverge at the rate
\begin{equation}
	\lim_{k \to \infty} \frac{\log \|P_k\|}{k^{\frac2{2+\alpha}}} = C_\alpha>0;
\end{equation}
more details, e.g.~the explicit constant $C_\alpha$ and analogous results for various potentials and weights (even of very slow growth) can be found in \cite{Mityagin-2017-272}.

For a second ill-behaved example related to an advection-diffusion operator, see \cite{Davies-2002-34,Krejcirik-2015-56,Reddy-1994-54}, consider the one-dimensional self-adjoint $T = - \partial_x^2 + V$ in $L^2(\R)$ with compactly supported real-valued $V \in L^\infty(\R;\R)$ and the weights
\begin{equation}
	w_\alpha(x) = \exp(\alpha x), \qquad \alpha \in \R, \qquad x \in \R.
\end{equation}
Not only do the essential spectra of $T$ and $T_{w_\alpha}$ differ, but more importantly, their point spectra might not coincide either, see Example~\ref{ex:spe} below for details.

Our analysis is focused on fundamental properties of weighted second order operators in a divergence form
	\begin{equation}
		T_w = -\nabla \cdot P \nabla + V
	\end{equation}
with accretive potentials $V$ and a sectorial coefficient matrix $P$. We identify conditions on $V$, $P$ and the weight $w$ allowing for a densely defined Dirichlet realisation in $L^2_{w}(\Omega)$ with non-empty resolvent set, see Theorem~\ref{thm:t}, as well as the graph norm separation, see Theorem~\ref{thm:T.graph.sep}. As the key technical ingredient we employ the notion of generalised coercivity of the associated sesquilinear form introduced in \cite{Almog-2015-40}, cf.~Section~\ref{ssec:gen.coer}. In Theorem \ref{thm:comp}, the boundedness of compositions of the type
\begin{equation}
	B_1(- \nabla \cdot (P\nabla) + V + 1)^{-1} B_2
\end{equation}
is addressed, where $B_1$ and $B_2$ are first order operators. Furthermore, in Theorem~\ref{thm:inv.Sp}, we study the compactness and Schatten class of the resolvent of $T_w$. The invariance of discrete spectra and eigenfunctions is treated in Theorem~\ref{thm:spd.inv}. As a corollary, a (super-)exponential decay of eigenfunctions is established and the obtained exponential rates are optimal in known cases, see Example~\ref{ex:optimality.EF}; this improves in particular the results in \cite{Krejcirik-2017-221}. Moreover, a natural set of eigenvalues and eigenfunctions can be covered due to a different proof strategy employed here (based on O'Connors Lemma, see \cite{OConnor-1973-32}, \cite[p.~196]{Reed4} and the proof of Theorem~\ref{thm:spd.inv}).

As examples for applications of our results, we first investigate the completeness of eigensystems of Schr\"odinger operators with unbounded accretive potentials, see Section~\ref{ssec:complete}, in particular extending the results in \cite{Almog-2015-40} and \cite{Siegl-2012-86}. Next, we show how our results enter the analysis of operator matrices with non-symmetric (in the sense of their strength) differential entries which exhibit Schur complement dominance, see Sections~\ref{ssec:ex.1}--\ref{ssec:dwe}. In particular, the last example in Section~\ref{ssec:dwe} deals with a wave equation in a weighted space subject to strong accretive damping, for which we prove the generation of a $C_0$-semigroup (and thereby generalise results in \cite{Freitas-2018-264,Gerhat-2024-286,Ikehata-2020-63}).

The paper is organised as follows. The main results are presented in Section~\ref{sec:results} and their proofs are given in Section~\ref{sec:proofs}. Section~\ref{sec:appl} contains examples of applications. Some known used results are collected in Appendix~\ref{app:prelim}.

\subsection{Notation}
\label{ssec:notation}
	Throughout the paper, $\Omega \subset \R^d$ is an open and non-empty set.
	We write
	\begin{equation}
		\langle x \rangle := (1+|x|^2)^\frac 12, \qquad x \in \Rd.
	\end{equation}

	If not specified otherwise, all inequalities between equivalence classes of measurable functions $\Omega \to \R$ are understood almost everywhere in $\Omega$. We write $a \lesssim b$ for  $a \le Cb$ with a constant $C>0$ which is independent of any relevant variable or parameter. The convention for $a \gtrsim b$ is analogous and we write $a \approx b$ if both inequalities hold.
	
	The essential support of a  measurable function $m: \Omega \to \C$ is
	\begin{equation}
		\operatorname{ess} \supp m = \left\{x \in \Omega \, : \, \forall~r>0, \, \, \lambda^d \big(\{ y \, : \, f(y) \neq 0\} \cap B_r (x)\big) > 0 \right\}.
	\end{equation}

	The norm and inner product on $L^2(\Omega)$ are denoted by $\|\cdot\|$ and  $\langle \cdot, \cdot \rangle$.
	Given a weight $0 < w \in C(\Omega)$, we introduce the Hilbert space
	\begin{equation}
		\begin{aligned}
			L^2_{w}(\Omega) & : = \left\{ f: \Omega \to \C \, : \, f \text{ is measurable and } \|f\|_{w} <\infty\right\}, \\
			\|f\|_{w} & : = \left(\int_{\Omega} |f|^2 w \, \dd x \right)^\frac12 = \|w^\frac12 f\|,
		\end{aligned}
	\end{equation}
	with the corresponding inner product denoted by
	\begin{equation}
		\langle f,g \rangle_{w} := \int_{\Omega} f \ov{g} w \, \dd x = \langle w^\frac12  f , w^\frac12 g \rangle;
	\end{equation}
		notice that $ w \in W^{1,\infty}_{\rm loc}(\Omega)$ in Assumption~\ref{asm:main.LM} implies that $w$ is locally Lipschitz continuous, see e.g.~\cite[p.~280]{Evans-1998}.
	
	The inner product notation is used whenever the integral converges; e.g.~for $f, g \in \Coo(\Omega)$ and $V \in \Loneloc (\Omega)$ we write
	\begin{equation}
		\langle Vf,g \rangle_{w} := \int_{\Omega}V f \ov{g} w \, \dd x.
	\end{equation}
	
	For $f : \Omega \to \C^d$, recall that $f \in L^2(\Omega)^d$ if $|f| \in L^2(\Omega)$, and we write
	\begin{equation}
		\|f\|_{L^2(\Omega)^d} = \| |f| \|_{L^2(\Omega)}.
	\end{equation} Analogously, for matrix valued functions $M : \Omega \to \C^{d\times d}$ we write $M \in L^2(\Omega)^{d\times d}$ if the matrix norm induced by the Euclidean norm is square integrable, i.e.~if
	\begin{equation}
		|M(\cdot)| := \sup_{ 0 \neq\xi \in \C^d}  \frac{|M(\cdot) \xi|}{|\xi|} \in L^2(\Omega), \qquad \|M\|_{L^2(\Omega)^{d\times d}} := \| |M| \|_{L^2(\Omega)}.
	\end{equation}
	Whenever it is clear from the context, we use $\| \cdot\|$ and $\langle \cdot, \cdot \rangle$ also for the norm and inner product on $L^2(\Omega)^d$ and $L^2(\Omega)^{d\times d}$. 
	Moreover, the same conventions are used for spaces of (locally) $p$-integrable vector or matrix valued functions with $p \in [1, \infty]$, as well as their weighted versions.

	Distributions in $\cD'(\Omega)$, as well as elements of the dual $\mathcal X^*$ of a normed space $(\mathcal X, \|\cdot\|_{\mathcal X})$, are by convention antilinear functional.
	
	The essential spectra $\sigma_{{\rm e}j}(T)$, $j=1,\dotsc,5$, of a densely defined, closed (non-self-adjoint) operator $T$ in a Hilbert space are defined as in~\cite[Sec.~IX.1]{EE}. 
	The discrete spectrum is
	\begin{equation}
		\begin{aligned}
		\spd (T) & := \big \{ \la \in \spp (T)\, : \, \la \,\, \text{is an isolated point of} \,\, \sigma(T) , \,\, m_{\rm a} (\la,T) < \infty \big\}
		\\ 
		& = \sigma (T) \setminus \se{5} (T) ,
		\end{aligned}
	\end{equation}
	where the algebraic multiplicity of an isolated spectral point $\la_0$ is defined as
	\begin{equation}
		m_{\rm a} (\la_0,T) = \dim \Ran(P_{\la_0}), \qquad P_{\la_0} = -\frac{1}{2 \pi \ii} \int_{\partial B_\eps (\la_0) } (T-\la)^{-1}\dd \la;
	\end{equation}
	for details see~\cite[Cor.~8.4]{Gohberg-1990} and \cite[Chap.~I \S 2]{Gohberg-1969} (where the relevant results also hold for unbounded operators), and notice that if $m_{\rm a} (\la_0,T)<\infty$ then $T-\la_0 $ has an approximate left inverse and hence its range is closed, see \cite[Thm.~I.3.13]{EE} and \cite[Chap.~III Equ.~(6.34)]{Kato-1966}.

\section{Main results}
\label{sec:results}

For an open set $\Omega \subset \R^d$, we introduce a Dirichlet realisation of the second order partial differential expression
\begin{equation}\label{T.def.1}
	T_w = -\nabla\cdot (P \nabla) + V 
\end{equation}
with an accretive potential $V: \Omega \to \C$ in the space $L^2_{w}(\Omega)$ with a suitable weight $w: \Omega \to (0,\infty)$ and show that it has dense domain and non-empty resolvent set. Employing the constructed $T_w$, we discuss bounded extensions for compositions of the type
	\begin{equation}\label{comp}
		(b_1 \cdot \nabla + c_1) (-\nabla\cdot (P \nabla) + V -\la)^{-1} (\nabla \cdot b_2 + c_2).
\end{equation}
Sufficient conditions for the Schatten class of the resolvent are derived, as well as for the invariance of discrete spectra and generalised eigenfunctions when passing from $T:=T_1$ in $L^2(\Omega)$ to $T_w$ in $L^2_w(\Omega)$. Finally, we establish the domain and graph norm separation and thereby generalise the result \eqref{T.norm.est} for Schr\"odinger operators to more general second order operators and, most importantly, to weighted spaces.

\subsection{Dirichlet realisation}

The construction of the Dirichlet realisation $T_w$ and subsequent theorems on the Schatten class of its resolvent, invariance of its spectra and eigenfunctions, as well as the boundedness of the compositions in~\eqref{comp}, are based on the following main set of assumptions.

\begin{asm-sec}\label{asm:main.LM}
	Let $\emptyset \neq \Omega \subset \R^d$ be open and let the following hold.
		\begin{enumerate}[\upshape (i)]
			\item \label{item:LM.reg} \emph{Regularity of coefficients and weight}: Assume that
			\begin{equation}
				V \in \Loneloc(\Omega), \qquad P \in \Loneloc (\Omega)^{d \times d}, 
			\end{equation}
			and let the weight
			\begin{equation}
				w \in W^{1,\infty}_{\rm loc}(\Omega)
			\end{equation}
			be uniformly positive on compact subsets of $\Omega$ (such that $w^{-1}$ has the same properties). Write 
			\begin{equation}\label{eq:def.P1.P2}
				\begin{aligned}
					V_1 & := \re V, & \quad  P_1 & := \Re P = \frac{1}{2} (P + P^*), \\
					V_2 & := \im V, & \quad P_2 & := \Im P = \frac{1}{2 \ii} (P - P^*),
				\end{aligned}
			\end{equation}
			and assume that $P_1$ is positive definite a.e.~in $\Omega$ with 
			\begin{equation}
				P_1^{-1} \in \Loneloc(\Omega)^{d\times d}.
			\end{equation}
			\item \label{item:asm.P}\emph{Sectoriality of $P$}: Assume there exists $C_P \geq 0$ such that for a.e.~$x \in \Omega$ it holds that
			\begin{equation}\label{asm:P.sect}
				\forall ~\xi \in \C^d \,\, \, : \,\, \, |\iprod{P_2 (x)  \xi}{\xi}_{\C^d}| \leq C_P  \iprod{P_1 (x)  \xi}{\xi}_{\C^d}.
			\end{equation}
			\item \label{item:asm.V.accr} \emph{Accretivity of $V$}: Assume that
			\begin{equation}
				 V_1 \ge 0.
			\end{equation}  
			
			\item \label{asm:LM.nabla.Phi}
			\emph{Control of multiplier $\Phi$}: Let $\Phi : \Omega \to [-1,1]$ satisfy $\nabla \Phi \in \Loneloc(\Omega)$
			and assume that for every $\eps >0$ there exists $C_{\Phi,\eps} \ge 0$ such that
			\begin{equation}\label{asm:Phi.sep}
				|P_1^{\frac{1}{2}} \nabla \Phi| \le \eps |V|^\frac12 + C_{\Phi,\eps}.
			\end{equation}
			
		\item \label{asm:LM.Phi.V2}
		\emph{Smallness of $\sgn V_2 - \Phi$:} Assume that $0 \le R \in \Loneloc(\Omega)$ satisfies
		\begin{equation}
			\label{eq:phi.sgn.rem}
			\Phi V_2 \ge |V_2| - R
		\end{equation}
		and for every $\eps>0$ there exists $C_{R,\eps}\ge 0$ such that for all  $f \in \Coo(\Omega)$
		\begin{equation}
			\label{eq:rem.rel.bdd.eps}
			\| R^\frac12 f\|^2_w \le \eps \|P_1^\frac12 \nabla f\|^2_w +  \eps \||V|^\frac12 f\|^2_w + C_{R,\eps} \|f\|_w^2.
		\end{equation}

		\item \label{asm:LM.w} \emph{Admissibility of $w$}: Let $\kappa_w, \tau_w, C_w \ge 0$ such that
		\begin{equation}\label{asm:w.sep}
			\frac{|P_1^{\frac{1}{2}} \nabla w|}{w} 		\le 
			\kappa_w V_1^\frac12 + \tau_w |V_2|^\frac12 + C_w;
		\end{equation}
		notice that $w$ is admissible if and only if $w^{-1}$ is admissible (with the same constants $\kappa_w, \tau_w$ and $C_w$).
	\end{enumerate}
\end{asm-sec}

We remark that Assumption~{\rm\ref{asm:main.LM}}~\ref{asm:LM.nabla.Phi} and \ref{asm:LM.Phi.V2} are formulated (implicitly) in terms of the multiplier $\Phi$ to allow for choosing a suitable $\Phi$ in various situations. Explicit sufficient conditions in terms of $V$ are in Section~\ref{sssec:Phi.V} below.

In Theorem~\ref{thm:t} below, we introduce $T_w$ using the representation theorems in Section~\ref{ssec:gen.coer}. Thereby the form 
	\begin{equation}\label{t.def}
			\mathbf t_w(f,g) : = \langle P \nabla f, \nabla (gw) \rangle + \langle  V f,  g \rangle_w, \qquad \Dom(\mathbf t_w) : = \cV_{w,0},
	\end{equation}
in $\Ltw$ is employed, with its domain being the closure of $\Coo(\Omega)$ in an ambient Hilbert space (a subspace of $L^2_{w} (\Omega)$)
\begin{equation}\label{def.big.cV}
	\Vv_w   :=  \left\{ f \in L^2_{w} (\Omega) \, : \,   \nabla f \in \Loneloc(\Omega)^d, \, \, |V|^\frac12 |f|  + |P_1^\frac12 \nabla f| \in L^2_{w}(\Omega) \right\}
\end{equation}
equipped with the scalar product
\begin{equation}\label{cV.def}
	\langle f, g \rangle_{\cV_w} : = \langle P_1^\frac12 \nabla f, P_1^\frac12 \nabla g \rangle_w +  \langle |V|^\frac12  f, |V|^\frac12  g \rangle_w +  \langle f,g \rangle_w.
\end{equation}
More precisely,
\begin{equation}\label{cV.0.def}
	\cV_{w,0} := \overline{\Coo(\Omega)}^{\cV_w} \subset \cV_w;
\end{equation}
see Lemmas~\ref{lem.cV.compl} and~\ref{lem:t.bdd} for the completeness of $\cV_w$ and boundedness of $\mathbf t_w$ on $\cV_{w,0}$ (and the fact that $\mathbf t_w$ is well-defined on $\cV_{w,0}$).

\begin{theorem}[Dirichlet realisation]
\label{thm:t}
Let Assumption~{\rm\ref{asm:main.LM}} hold with $\kappa_w$, $\tau_w$, and $C_P$ small enough such that there exists $0<\beta<1/C_P$ (with the convention $1/0 := \infty$) satisfying 
\begin{equation}\label{ellipse}
	\beta \kappa_w^2  + \tau_w^2  < \frac{4 \beta }{ 1 + \beta^2} \,\frac{1 - \beta C_P}{(1+C_P)^2}.
\end{equation}
Let $\Vv_{w,0}$ be as in~\eqref{cV.0.def}. Then the operator in $\Ltw$
\begin{equation}\label{T.descr}
	\begin{aligned}
		\Dom(T_w) & := \Big\{  f \in {\cV_{w,0}} \,: \, - \nabla \cdot (P \nabla f) + V f \in L^2_{w}(\Omega) \Big\}, \\
		T_w f & := - \nabla \cdot (P \nabla f) + V f,
	\end{aligned}
\end{equation}
with its action understood in $\Dd'(\Omega)$, is closed, has non-empty resolvent set and its domain is dense both in $\cV_{w,0}$ and in $L^2_{w}(\Omega)$.

\end{theorem}

\begin{remark}
	\label{rem:gcoer}
	\begin{enumerate}[\upshape (i),wide]
		\item \label{item.rem.opt} When selecting $\tau_w=0$ in~\eqref{asm:w.sep}, condition~\eqref{ellipse} becomes
		\begin{equation}
			\kappa_w^2 < \sup_{0 < \beta < 1/C_P} \frac{1-\beta C_P}{1+\beta^2} \, \frac{4}{(1+C_P)^2} = \frac{4}{(1+C_P)^2}.
		\end{equation}
		Similarly, setting $\kappa_w=0$  leads to the condition
		\begin{equation}\label{delta.w.beta.im}
			\tau_w^2 < \max_{0< \beta < 1/C_P} \frac{\beta (1-\beta C_P)}{1+\beta^2} \, \frac{4}{(1+C_P)^2}= \frac{2}{ ( \sqrt{1+C_P^2} + C_P)(1+C_P)^2 }.
		\end{equation}
		In particular, if $C_P=0$, then the above restrictions read $\kappa_w<2$ and $\tau_w<\sqrt{2}$, which lead to optimal exponents in eigenfunction decay rates for Schr\" odinger operators, see Example~\ref{ex:optimality.EF}.
		\item \label{asm.gen.LM} In fact, we prove a more flexible result, see the proof of Proposition~\ref{prop:g.coer}. Indeed, the statements of Theorem~\ref{thm:t} (and also of Corollary~\ref{cor:res.bd}, Theorems~\ref{thm:comp} and \ref{thm:inv.Sp}, Corollary~\ref{cor:Sp} and Theorem \ref{thm:spd.inv} below) hold if one replaces~\eqref{eq:rem.rel.bdd.eps} by
		\begin{equation}
			\label{eq:rem.rel.bdd}
			\| R^\frac12 f\|^2_w \le \vartheta_R \|P_1^\frac12 \nabla f\|^2_w + \kappa_R \|V_1^\frac12 f\|^2_w + \tau_R \||V_2|^\frac12 f\|^2_w + C_R \|f\|_w^2
		\end{equation}
		with constants  $\vartheta_R, \kappa_R, C_R \ge 0$ and $0 \le \tau_R <1$, and the condition~\eqref{ellipse} by the existence of
		\begin{equation}
			\label{ellipse.gen.beta}
			0< \beta < \min \left\{ \frac{1}{C_P+\vartheta_R}, \frac{1}{\kappa_R} \right\}
		\end{equation}
		satisfying the inequality
		\begin{equation}\label{ellipse.gen}
			\beta \kappa_w^2  (1-\tau_R)+ \tau_w^2 (1-\beta\kappa_R) < \frac{4 \beta }{ (1 + \beta^2)} \, \frac{(1 - \beta (C_P+\vartheta_R)) (1-\beta \kappa_R) (1-\tau_R)}{(1+C_P)^2 }.
		\end{equation}
		(Notice that the condition \eqref{ellipse.gen} is not void in the following sense. Given $\tau_R \in[0,1)$ and $C_P,\vartheta_R,\kappa_R \ge 0$, there exist (sufficiently small) $\kappa_w,\tau_w>0$ and $\beta >0$ with~\eqref{ellipse.gen.beta} and \eqref{ellipse.gen}.)

		\item \label{item.ecrit} For the statement of Theorem~\ref{thm:t} to hold, it is sufficient to assume that~\eqref{asm:Phi.sep} is true with some $\eps \in (0, \ec)$ and $C_{\Phi,\eps} \ge 0$. Given the constants $\kappa_w$, $\tau_w$, $C_P$, $\vartheta_R$, $\kappa_R$, $\tau_R$, see point \ref{asm.gen.LM}, as well as $\beta$ in~\eqref{ellipse} or~\eqref{ellipse.gen}, the critical value $\ec$ can be obtained from a thorough analysis of the inequalities in the proof of Proposition~\ref{prop:g.coer}; see Remark~\ref{rem:ecrit} (the condition $\eps_V + \eps_w < (2- \sqrt 2)/(1+\delta_P^{-1} \|P_2\|_{L^\infty})$ in Theorem~\ref{thm:T.graph.sep} arises in a similar way, see the proof of Lemma~\ref{lem:graph.norm}).
	\end{enumerate}
\end{remark}

If the weight is trivial, $T_1 = T$ is m-accretive in $L^2(\Omega)$. While this might be lost for $T_w$ in general, the next corollary shows that one still has the expected control over the resolvent norm in the left half plane as long as the weight is suitably small. 

\begin{corollary}
	\label{cor:res.bd}
	Let Assumption~{\rm\ref{asm:main.LM}} be satisfied with $\tau_w$ and $\kappa_w$ arbitrarily small. Then there exists an arbitrarily small $\beta>0$ satisfying~\eqref{ellipse}, such that the operator $T_w$ in $L^2_w(\Omega)$ defined in~\eqref{T.descr} is densely defined in $\cV_{w,0}$ and $\Ltw$, see \eqref{cV.0.def}, and has non-empty resolvent set. Moreover, 
	\begin{equation}
		\forall~\varphi \in \left(\frac{\pi}{2}, \frac{3\pi}{2}\right) \,\,\, : \, \,\, \|(T_w-r \e^{\ii \varphi})^{-1}\| \ls \frac1r, \qquad r \to \infty.
	\end{equation}
\end{corollary}

\subsubsection{Assumption~{\rm\ref{asm:main.LM}}~\ref{asm:LM.nabla.Phi} and \ref{asm:LM.Phi.V2} in terms of $V$}
\label{sssec:Phi.V}

Assumption~\ref{asm:main.LM} is formulated in terms of the multiplier $\Phi$ to allow for choosing a suitable $\Phi$ in various situations. The main challenge in finding such $\Phi$ are sign changes of $\Im V$. In the simple case where $d=1$ and the sign changes are contained in a bounded interval, a convenient choice of $\Phi$ is possible; this is demonstrated in the following example. 
\begin{example}
	\label{ex:Phi.1d}
	Let $\Omega = \R$, $P=1$ and $V \in L^1_{\mathrm{loc}}(\R)$ with $V_1 = \re V \ge 0$ such that for $V_2 = \im V$, with some $x_0>0$,
	\begin{equation}
		V_2(x) > 0 \quad \text{if } \  x >x_0, \qquad V_2(x) < 0 \quad \text{if } \ x < -x_0.  
	\end{equation}	
	Select $\Phi \in C^\infty(\R)$ such that $- 1 \leq \Phi \leq 1$ and 
	\begin{equation}
		\Phi(x) = \begin{cases}
			1, & x > x_0, \\
			-1, & x <-x_0.
		\end{cases}
	\end{equation}	
	This $\Phi$ clearly satisfies \eqref{asm:Phi.sep}. Moreover, 
	\begin{equation}
		\Phi V_2 = |V_2| + (\Phi V_2 - |V_2|) \geq |V_2| - |\Phi V_2 - |V_2||,
	\end{equation}
	such that we may set $R:= |\Phi V_2 - |V_2|| \in L^1(\R)$ with $\operatorname{ess}\supp R \subset [-x_0,x_0]$. In Lemma~\ref{lem:R} below we verify that this $R$ satisfies~\eqref{eq:rem.rel.bdd.eps} for any admissible weight $w$.
\end{example}

Without specific information on $\sgn \, (\im V)$, a modified version of the multiplier $\Phi$ in \eqref{Phi.AH} can be used, which leads to explicit conditions on the gradient of $V$. 
\begin{proposition}\label{prop:Phi.AH-like}
	Let $\Omega$, $P$ and $V$ be as in Assumption~{\rm \ref{asm:main.LM}}. Let $V_1$ and $V_2$ be decomposed as (to real-valued functions)
	\begin{equation}
		\begin{aligned}
			V_1 & = V_{1,r} + V_{1,s}, \\
			V_2 & = V_{2,r} + V_{2,s}, 
		\end{aligned}
		\qquad V_{1,s}, V_{2,s} \in \Loneloc (\Omega), \qquad \nabla V_{1,r}, \nabla V_{2,r} \in L^1_{\rm loc}(\Omega),
	\end{equation}
	such that $V_{1,r},V_{1,s} \geq 0$ and let $\delta \in [0,1]$ be such that
	\begin{equation}\label{V1.rs.delta}
		V_{1,r} \leq \delta V_1.
	\end{equation}
	Suppose that the following conditions hold.
	\begin{enumerate}[\upshape (i)]
		\item \label{item.LM.V.1} For every $\eps >0$ there exists $C_{\nabla,\eps} \ge 0$ such that
		\begin{equation}\label{V.nabla.asm}
			\begin{aligned}
				& (1+V_{1,r}^2)  |P_1^{\frac{1}{2}} \nabla V_{2,r}| +	V_{1,r} |V_{2,r}| |P_1^{\frac{1}{2}} \nabla V_{1,r}| 
				\\
				& \qquad \qquad \qquad \qquad \qquad \quad \le (1+V_{1,r}^2 + V_{2,r}^2)^\frac32 \left(\eps |V|^\frac12 + C_{\nabla,\eps}\right). \\
			\end{aligned}
		\end{equation}
		\item \label{item.LM.V.2} There exist constants $\vartheta_s, \kappa_s, C_s \ge 0$ and $0 \le \tau_s <1/2$ such that for all $f \in C_0^\infty(\Omega)$
		\begin{equation}\label{V2s.rel.bdd}
			| |V_{2,s}|^\frac12 f\|^2_w \le  \vartheta_s \|P_1^\frac12 \nabla f\|^2_w + \kappa_s \|V_1^\frac12 f\|^2_w + \tau_s \||V_2|^\frac12 f\|^2_w + C_s \|f\|_w^2.
		\end{equation}
	\end{enumerate}
	Then 
	\begin{equation}\label{Phi.AH-like}
		\Phi := \frac{V_{2,r}}{\sqrt{1+ V_{1,r}^2 + V_{2,r}^2}}: \Omega \to [-1,1]
	\end{equation}
	satisfies \eqref{asm:Phi.sep}, \eqref{eq:phi.sgn.rem} and \eqref{eq:rem.rel.bdd} with
	\begin{equation}
		\vartheta_R = 2 \vartheta_s, \quad \kappa_R = 2 \kappa_s + \delta, \quad \tau_R = 2 \tau_s, \quad C_R = 2C_s + 1.
	\end{equation} 
\end{proposition}

\begin{remark}
	\label{eps.crit}\label{rem:im.V}\label{rem:V.sec}
	 For regular purely imaginary potentials $V = \ii V_{2,r}$ (i.e.~for $V_1 = V_{2,s} =0$), the assumptions~\ref{item.LM.V.1} and~\ref{item.LM.V.2} in Proposition~\ref{prop:Phi.AH-like} simplify to assuming that for every $ \eps>0$ there exists $C_{\nabla,\eps} \geq 0$ such that
		\begin{equation}\label{nab.Vr.im}
			|P_1^{\frac{1}{2}} \nabla V_{2,r}| \le \eps |V_{2,r}|^\frac72 + C_{\nabla,\eps};
		\end{equation}
		this improves the previously used assumption \eqref{Q.asm.sep} in case $P=I_{\C^d}$ and $w=1$.
\end{remark}

\subsection{Boundedness of compositions}

The following theorem allows us to construct bounded extensions for compositions of the resolvent of $T=T_1$ with suitable first order differential operators.

\begin{theorem}[Boundedness of compositions]\label{thm:comp}
	Let the assumptions of  Theorem~{\rm\ref{thm:t}} hold and let $T=T_1$ in $L^2(\Omega)$ be as in~\eqref{T.descr} with trivial weight. Suppose that
	\begin{equation}\label{eq:ext.reg}
		b_1, \, b_2 \in L^2_{\rm loc}(\Omega)^d, \qquad \nabla \cdot b_2 ,\,  c_1, \,  c_2 \in L^2_{\rm loc}(\Omega),
	\end{equation}
	and that there exists an admissible weight $w$ such that
	\begin{equation}\label{comp.asm}
		\frac{w^{-\frac12} c_1}{(V| + 1)^\frac12}, \,\,  \frac{w^\frac12 c_2}{(|V| + 1)^\frac12}, \, \,  w^{-\frac12} |P_1^{-\frac12} \overline b_1|, \,  \, w^{\frac12}|P_1^{-\frac12} b_2| \,\, \in L^\infty (\Omega).
	\end{equation}
	Define the composition 
	\begin{equation}
		\begin{aligned}
			F_{\la} &:= (b_1 \cdot \nabla + c_1) (T-\lambda)^{-1} (\nabla \cdot b_2 + c_2), \qquad \la \in \rho(T),
			\\
			\Dom(F_\la) &:= \CcOm.
		\end{aligned}
	\end{equation}
	Then $\Ran(F_\la) \subset \LolocOm$ and there exists $\lambda_0 \in \rho(T)$ such that $F_{\la_0}$ extends to a bounded operator on 	$L^2(\Omega)$ (cf.~Lemma {\rm\ref{lem:bdd.ext}} for the construction of the extension).
\end{theorem}

\subsection{Schatten class of resolvent}

Our next result gives a sufficient condition for the (weak) resolvent Schatten class of $T_w$. Independently of the admissible weight, it suffices that the embedding of the form domain of $T= T_1$ in $L^2(\Omega)$ is of the respective (weak) Schatten class.

\begin{theorem}[Schatten class of resolvent]\label{thm:inv.Sp}
	Let the assumptions of Theorem~{\rm\ref{thm:t}} be satisfied, let $T_w$ in $L^2_{w}(\Omega)$ be as in~\eqref{T.descr}, let $\cV_{1,0}$ and $\iota_1$ be as in \eqref{cV.0.def} and~\eqref{eq:T.hat.res} with trivial weight, let $\la \in \rho (T_w)$ and $p \in (0,\infty)$.
	\begin{enumerate}[\upshape (i)]
	\item If $\iota_1$ is compact, then $(T_w-\la)^{-1}$ is compact;
	\item if $\iota_1 \in \cS_{2p} (\cV_{1,0}, L^2 (\Omega))$, 
	then 
	$
		(T_w-\la)^{-1} \in \cS_p(L^2_{w}(\Omega));
	$
	\item if $\iota_1\in \cS_{2p,\infty} (\cV_{1,0}, L^2 (\Omega))$, 
	then
	$
		(T_w-\la)^{-1} \in \cS_{p,\infty}(L^2_{w}(\Omega)).
	$
	\end{enumerate}
\end{theorem}

Since for $P=I_{\C^d}$ the form domain $\cV_{1,0}$ equals the form domain of $-\Delta + |V|$ in $L^2(\Omega)$, combining the above theorem with existing results in the self-adjoint case, cf.~Theorems~\ref{thm:Sp} and \ref{thm:Sp.inf} in the Appendix, we obtain sufficient conditions in terms of the potential for the resolvent to belong in a (weak) Schatten class.

\begin{corollary}
	\label{cor:Sp}
	Let the assumptions of Theorem~{\rm\ref{thm:t}} hold with  $P=I_{\C^d}$, let $T_w$ in $L^2_{w}(\Rd)$ be as in~\eqref{T.descr} and let $\la \in \rho (T_w)$. Assume that $\partial \Omega \in C^{2,\alpha}$ for some $\alpha>0$.
	\begin{enumerate}[\upshape (i)]
		\item If $V \in L^2_{\rm loc}(\Omega)$ and for some $p>0$
	\begin{equation}\label{Sp.suff}
	\int_{\Omega \times \R^d} (|\xi|^2 + |V(x)| + 1)^{-p} \, \dd x \, \dd \xi < \infty,
	\end{equation}
	then 
	\begin{equation}\label{Tw.Sp.cor}
	(T_w-\la)^{-1} \in \cS_p(L^2_{w}(\Omega)).
	\end{equation}
	\item If for some $\gamma > 0$
	\begin{equation}\label{V.lb.gamma}
	|V(x)| + 1 \gs \langle x \rangle^\gamma, \quad \text{a.e.}~x \in \Omega,	
	\end{equation}
	then with
	\begin{equation}\label{p.gam.d.def}
	p_{\gamma,d} := d \frac{\gamma +2}{2 \gamma}		
	\end{equation}	
	we have for any $p>p_{\gamma,d}$
	
	\begin{equation}\label{Tw.Sp.cor.gam}
	(T_w-\la)^{-1} \in \cS_{p}(L^2_{w}(\Omega)).
	\end{equation}
	If, in addition, $\Omega = \Rd$ and $\gamma>1$, then
	\begin{equation}\label{Tw.Sp.cor.gam.inf}
	(T_w-\la)^{-1} \in \cS_{p_{\gamma,d}, \infty}(L^2_{w}(\Rd)).	
	\end{equation}
	\end{enumerate}
\end{corollary}

\subsection{Invariance of discrete spectra and decay of eigenfunctions}

In regions separated from the essential spectrum, the (discrete) spectra of $T$ and $T_w$ coincide and the corresponding (finite) algebraic multiplicities of the eigenvalues agree. Moreover, every generalised eigenfunction (root function) of $T$ is also a generalised eigenfunction  of $T_w$. In particular, this provides information on the decay of the eigenfunctions of the non-weighted operator $T$; an alternative proof of the latter, based on Agmon type decay estimates, can be found in \cite{Krejcirik-2017-221} in a setting with magnetic field. 

\begin{theorem}[Invariance of discrete spectra and eigenfunctions]\label{thm:spd.inv}
Let the assumptions of Theorem~{\rm\ref{thm:t}} be satisfied and, for every $\alpha \in [0,1]$, let $T_{w^\alpha}$ in $L^2_{w^\alpha}(\Omega)$ be defined as in~\eqref{T.descr} with weight $w^\alpha$ (which is admissible and satisfies the assumptions of Theorem~{\rm\ref{thm:t}}, see Lemma~{\rm\ref{lem:w.alph}}). Define the region
\begin{equation}
	\Sigma:= \overline{\bigcup_{\alpha \in [0,1]} \sigma_{{\rm e5}} (T_{w^\alpha})} 
\end{equation}
covered by essential spectra of $T_{w^\alpha}$, see \cite[Chap.~IX]{EE} for the definition of $\sigma_{{\rm e}5}(\cdot)$. Then the (discrete) spectra of $T$ and $T_w$ coincide outside $\Sigma$, i.e.
\begin{equation}\label{sp.inv}
	\sigma(T)\setminus \Sigma = \sigma(T_w) \setminus \Sigma \subset \sigma_{\rm disc}(T)\cap \sigma_{\rm disc}(T_w).
\end{equation}
Moreover, for all eigenvalues in the above set, the (finite) algebraic multiplicities with respect to $T$ and $T_w$ coincide and 
\begin{equation}\label{gen.ef}
	\ker (T-\lambda)^k = \ker (T_w-\lambda)^k, \qquad \la \in \sigma(T) \setminus \Sigma, \qquad k \in \N.
\end{equation}
In particular, all generalised eigenfunctions (root functions) of $T$ lie in $L^2_{w} (\Omega)$.
\end{theorem}

\begin{remark}
\begin{enumerate}[\upshape (i), wide]
	\item \label{item:comp.res} If  $\Vv_{1,0}$ is compactly embedded in $L^2(\Omega)$, then the $T_{w^\alpha}$ have compact resolvent for all $\alpha \in [0,1]$, see Theorem~\ref{thm:inv.Sp}, and Theorem~\ref{thm:spd.inv} gives invariance of the (discrete) spectra of $T$ and $T_w$ as well as the related (generalised) eigenfunctions, in whole $\C$ (including multiplicities). This is for instance the case if $P_1$ is uniformly elliptic, see Assumption~\ref{asm:main.gn}~\ref{item:asm.P1.ell} below, and
	\begin{equation}\label{V.unbdd}
		\lim_{R \to \infty} \essinf_{|x|>R,\,x \in \Omega} |V(x)|  = \infty;
	\end{equation} 
	see the standard compactness arguments based on Rellich's criterion e.g.~in~\cite[Thm.\ XIII.65, XIII.67]{Reed4}.
	
	\item It is in fact necessary to exclude the set $\Sigma$ in the statement of Theorem~\ref{thm:spd.inv}. Indeed, Example~\ref{ex:spe} below (related to an advection-diffusion operator, see \cite{Reddy-1994-54,Davies-2002-34}) shows that an eigenvalue can disappear when touched by the essential spectrum. 
	
	\item In Example~\ref{ex:optimality.EF} below, the eigenfunctions of one-dimensional Schr\"odinger operators in $L^2(\R)$ are discussed. In these particular examples, it is shown that the restrictions on the admissible weights in \eqref{asm:w.sep} and \eqref{ellipse} in Theorem~\ref{thm:spd.inv} lead to optimal eigenfunction decay rates.
\end{enumerate}
\end{remark}

\begin{example}\label{ex:spe}
We sketch and slightly adapt an example in \cite[Sec.\ VII.C]{Krejcirik-2015-56}. Consider the standard self-adjoint realisation of $T := - \partial_x^2 + V$ in $L^2(\R)$ with $V \in L^\infty(\R,\R)$ such that $\operatorname{ess} \supp V \subset [-1,1]$ and assume there exists a simple eigenvalue $0> \la_0 \in \spd(T)$; the existence of such potential $V$ follows by well-known min-max arguments, see~e.g.~\cite{Reed4}. It follows from the assumptions that an eigenfunction $\psi_0$ corresponding to $\la_0$ satisfies 
\begin{equation}
	\psi_0(x) = \begin{cases}
		C_- \e^{ \sqrt{|\la_0|}x}, & \,\, x <-1, \\
		C_+ \e^{- \sqrt{|\la_0|}x}, & \,\, x >1,
	\end{cases} \qquad \quad C_\pm \in \C.
\end{equation}

Consider the family of admissible weights $w_\alpha(x) = \exp(\alpha x)$ with $\alpha, x \in \R$ and the corresponding operators $T_{w_\alpha}$ in $L^2_{w_\alpha}(\R)$. Notice that, strictly speaking, the assumptions of Theorem~\ref{thm:t} are only satisfied after shifting the potential by $\|V\|_{L^\infty}$ (with multiplier $\Phi = 0$); the resulting $T_1$ with $\alpha = 0$ is exactly $T$ in $L^2(\R)$ as above. The essential spectra of $T_{w_\alpha}$ can be determined by passing to the family of unitarily equivalent operators 
\begin{equation}
S_\alpha := w_\alpha^{\frac12} T_{w_\alpha} w_\alpha^{-\frac12}
\end{equation}
in $L^2(\R)$, cf.~Lemma~\ref{lem:Sw.form}. Indeed, one can verify that 
\begin{equation}
	S_\alpha = - \partial_x^2 +  \alpha \partial_x - \frac{\alpha^2}4 + V, \qquad \Dom(S_\alpha) = W^{2,2}(\R),
\end{equation}
and further by employing the Fourier transform and the stability of essential spectra under relatively compact perturbations that
\begin{equation}\label{ex.ess.spec}
\se{\textit j}(T_{w_{\alpha}}) = \se{\textit j}(S_{\alpha}) = 	\left\{k^2-\alpha \ii k - \frac{ \alpha^2}4 \, : \, k \in \R \right\}, \qquad  j=1,\dotsc,4,
\end{equation}
see e.g.~\cite[Chap.~IX]{EE} for details. In fact, it follows from a Neumann series argument, the boundedness of $V$ and the structure of $\sigma_{{\rm e4}}(T_{w_{\alpha}})$ that also $\sigma_{{\rm e5}}(T_{w_{\alpha}}) = \sigma_{{\rm e4}}(T_{w_{ \alpha }})$ is as above. 

On the other hand, while for $|\alpha| < 2 \sqrt{|\la_0|}$ one still has $\la_0 \in \spp(T_{w_\alpha})$ (simple with the same eigenfunction $\psi_0$, see Theorem~\ref{thm:spd.inv}), the eigenvalue $\la_0$ is lost for $T_{w_\alpha}$ when it is touched by the essential spectrum (i.e.~for $|\alpha| = 2 \sqrt{|\la_0|}$). Indeed, for $|\alpha| \ge 2 \sqrt{|\la_0|}$ one has $\psi_0 \notin L^2_{w_\alpha}(\R)$ and it is elementary to verify that also no other (distributional) solution of $-\psi'' + V \psi = \la_0 \psi$ lies in $L^2_{w_\alpha}(\R)$.
\end{example}

\begin{example}[Optimality of \eqref{asm:w.sep} and \eqref{ellipse}] \label{ex:optimality.EF}
	In the following, let $\Omega = \R$, $P=1$ and consider $T = -\partial_x^2 + V$ in $L^2(\R)$. Theorem~\ref{thm:spd.inv} states that the eigenfunctions of $T$ lie in $L^2_w(\R)$ for every $w$ which is admissible according to Assumption~\ref{asm:main.LM} and~\eqref{ellipse}. For particular potentials $V$, we show that our restrictions on the weight correspond to the (optimal) ones arising from the known results.
	\begin{enumerate}[\upshape (i), wide]
		\item For real potentials $V(x)= x^{2n}$, $n \in \N$, it is known that for a fixed $\la \in \spp(T)$, any corresponding eigenfunction of $T$ satisfies
		\begin{equation}
			 |\psi(x)| \ls \exp\left(- \frac{|x|^{n+1}}{n+1}\right), \qquad x \in \R, 
		\end{equation}
		and the~power and constant in the exponential are optimal, see e.g.~\cite[Chap~8.2]{Titchmarsh-1962-book1}, \cite[Chap.~22.27]{Titchmarsh-1958-book2}. Thus $\psi \in L^2_w(\R)$ with the weight
		\begin{equation}\label{w.xn}
			w(x) = \exp \left(\kappa \frac{|x|^{n+1}}{n+1}  \right), \qquad x \in \R, \qquad \kappa <2.
		\end{equation}

		In view of Theorem~\ref{thm:spd.inv}, for real potentials, no multiplier $\Phi$ is needed, i.e.~$\Phi=0$ satisfies the conditions in Assumption~\ref{asm:main.LM}. The weight $w$ in \eqref{w.xn} is admissible in the sense of Assumption~\ref{asm:main.LM}~\ref{asm:LM.w} with $\kappa_w = \kappa$ and $\tau_w=0$. The restriction from \eqref{ellipse} reads $\kappa_w<2$, see Remark~\ref{rem:gcoer}~\ref{item.rem.opt}.
		\item  For the Davies oscillator, i.e.~$V(x) = \ii x^2$, the (countable) eigenvalues and eigenfunctions of $T$ are explicit. Namely, $\la_n = \e^{\ii \frac \pi 4}(2n+1)$ with $n \in \N_0$ and
		\begin{equation}\label{EF.Davies}
			\psi_n(x) = \exp\left(-\e^{\ii \frac \pi 4} \frac{x^2}{2} \right) H_n\big(\e^{\ii \frac \pi 8} x\big), \qquad x \in \R, \qquad n \in \N_0,
		\end{equation} 
		where $H_n$ are the Hermite polynomials, see e.g.~\cite[Sec.~14.5]{Davies-2007}. 
		Thus $\psi_n \in L^2_w(\R)$ with the weight
		\begin{equation}\label{w.Davies}
			w(x) = \exp \left(\tau \frac{x^2}{2}  \right), \qquad x \in \R, \qquad \tau <\sqrt 2.
		\end{equation}

		Since $V_2(x) = x^2$ does not change sign, the multiplier $\Phi=1$ satisfies the conditions in Assumption~\ref{asm:main.LM}. The weight $w$ in \eqref{w.Davies} is admissible in the sense of Assumption~\ref{asm:main.LM}~\ref{asm:LM.w} with $\kappa_w=0$,  $\tau_w = \tau$ and the restriction from \eqref{ellipse} reads $\tau_w<\sqrt 2$, see Remark~\ref{rem:gcoer}~\ref{item.rem.opt}. 
		\item For the imaginary cubic oscillator, i.e.~$V(x) = \ii x^3$, the spectrum of $T$ is non-empty and in fact real, see \cite{Dorey-2001-34, Shin-2002-229}. 
		For a fixed $\la \in \spp(T)$, the corresponding eigenfunction of $T$ satisfies
		\begin{equation}
		|\psi(x)| \ls \exp\left(- \frac1{\sqrt 2} \frac 25 |x|^{\frac 52} \right), \qquad x \in \R, 
		\end{equation}
		and the power and constant in the exponential are
		optimal, see e.g.~\cite[\S7.4]{Sibuya-1975} or \cite{Shin-2002-229}. Thus $\psi \in L^2_w(\R)$ with the weight
		\begin{equation}\label{w.ix3}
			w(x) = \exp \left(\tau \frac25 |x|^\frac 52  \right), \qquad x \in \R, \qquad \tau <\sqrt 2.
		\end{equation}

		For $V_2(x) = x^3$, the multiplier $\Phi$ in Example~\ref{ex:Phi.1d} or in \eqref{Phi.AH-like} satisfies the conditions in Assumption~\ref{asm:main.LM}. The weight $w$ in \eqref{w.ix3} is admissible in the sense of Assumption~\ref{asm:main.LM}~\ref{asm:LM.w} with $\kappa_w=0$ and $\tau_w =\tau$ and the restriction on $\tau_w$ in this case from \eqref{ellipse} reads $\tau_w<\sqrt 2$, see Remark~\ref{rem:gcoer}~\ref{item.rem.opt}.
	\end{enumerate}

\end{example}

\subsection{Domain and graph norm separation}

Our last result on the domain and graph norm separation for $T_w$ requires the following set of assumptions.

\begin{asm-sec}\label{asm:main.gn}
	Let $\emptyset \neq \Omega \subset \R^d$ be open and let the following hold.
	\begin{enumerate}[\upshape (i)]
		\item \label{item:asm.GN.reg} \emph{Regularity of coefficients and weight}:
		Assume that
		\begin{equation}\label{eq:asm.gn.reg}
			P_{ij} \in W^{1,\infty} (\Omega), \qquad 1 \le i,j \le d,
		\end{equation}
		that the potential satisfies
		\begin{equation}
			V \in W^{1,\infty}_{\rm loc} (\overline \Omega) := \left\{ f  \in L^\infty_{\rm loc} (\Omega) \, : \, \forall~R>0, \, \,  |f| + |\nabla f| \in L^\infty (\Omega \cap B_R(0)) \right\}
		\end{equation}
		and let the weight
		\begin{equation}
			w \in W^{1,\infty}_{\rm loc}(\Omega)
		\end{equation}
		be uniformly positive on compact subsets of $\Omega$ (such that $w^{-1}$ has the same properties).
		\item \label{item:asm.P1.ell} \emph{Uniform ellipticity of $P_1$}: With $P_1:=\re P$ and $P_2:= \im P$, see~\eqref{eq:def.P1.P2}, let $\delta_P>0$ such that for a.e.~$x \in \Omega$ one has
		\begin{equation}\label{asm:gn.P}
			\forall~\xi \in \C^d \,\,\, : \,\,\,  \langle P_1 (x) \xi, \xi \rangle_{\C^d} \geq \delta_P |\xi|^2.
		\end{equation}
		\item \emph{Generalised accretivity of $V$:} \label{item:VP.accr} Suppose that
		\begin{equation}\label{asm:VP.accr}
			\forall~\xi \in \C^d \,\,\, : \,\,\, \re \langle \e^{-\ii \arg V} P \xi,\xi \rangle_{\C^d} \geq 0.
		\end{equation}
		\item \emph{Control of $\nabla V$:} Assume there exist $\eps_V, C_{V,\eps} \ge 0$ with 		%
		\begin{equation}\label{asm:V.sep.gn}
			\max \left\{ |P_1^{\frac{1}{2}} \nabla V|, |P_1^{\frac{1}{2}} \nabla |V|| \right\}\le \eps_V \abs{V}^\frac{3}{2} + C_{V,\eps}.
		\end{equation}
		\item \emph{Admissibility of $w$}:  Assume there exist $\eps_w, C_{w,\eps} \ge 0$ with
	\begin{equation}\label{asm:w.sep.gn}
		\frac{|P_1^{\frac{1}{2}} \nabla w|}{w} 	\le 
		\eps_{w} |V|^\frac12 + C_{w,\eps};
	\end{equation}
	notice that $w$ is admissible if and only if $w^{-1}$ is admissible (with the same constants $\eps_w$ and $C_{w,\eps}$).
	\end{enumerate}
\end{asm-sec}

\begin{theorem}[Domain and graph norm separation]
	\label{thm:T.graph.sep}
	Let Assumption~{\rm\ref{asm:main.gn}} hold and assume that
	\begin{equation}\label{eps.asm}
		\eps_V + \eps_w < \frac{2 - \sqrt 2}{1 + \delta_P^{-1}\|P_2\|_{L^\infty}}.
	\end{equation}
	Then  $T_w$ introduced in~\eqref{T.descr} is well-defined, its domain satisfies
		\begin{equation}\label{eq:dom.T.gn}
			\Dom (T_w) = \Big\{ f \in \cV_{w,0} \, : \, - \nabla \cdot (P \nabla f) \in L^2_{w} (\Omega), \,  Vf \in L^2_{w} (\Omega) \Big\}
		\end{equation}
		and there exists $a_{V,w}>0$ such that
	\begin{equation}\label{gn.sep.main}
		\| T_w f\|_{w} + \| f\|_{w} \geq a_{V,w} \big(
		\| \nabla \cdot (P \nabla f)\| _{w}+ \| V f\|_{w} + \| f\|_{w}
		\big)
	\end{equation}
for all  $f \in \Dom(T_w)$. 

Moreover, if $\Omega= \Rd$, then $C_0^\infty(\Rd)$ is dense in $\big(\Dom(T_w),(\|T_w \cdot\|_w^2+\|\cdot\|^2_w)^\frac12\big)$.
\end{theorem}

\section{Proofs and lemmas}
\label{sec:proofs}

\subsection{Form domain}

\begin{lemma}[Completeness of $\Vv_w$]\label{lem.cV.compl}
	Let Assumption~{\rm \ref{asm:main.LM}}~\ref{item:LM.reg} be satisfied. Then the space $\cV_w$ in~\eqref{def.big.cV} equipped with the scalar product~\eqref{cV.def} is a Hilbert space. Consequently, also $\cV_{w,0}$ as in~\eqref{cV.0.def} is a Hilbert space which is dense and boundedly embedded in $L^2_{w} (\Omega)$. 
\end{lemma}

\begin{proof}
	Clearly, $\Vv_w$ is a pre-Hilbert space; its completeness is justified by standard arguments. Let   $\{f_n\}_n \subset \cV_w$ be a Cauchy sequence.  By~\eqref{cV.def} and the completeness of $L^2_{w} (\Omega)$ and $L^2_{w} (\Omega)^d$, there exist $f, g \in L^2_{w}(\Omega)$ and $h \in L^2_{w}(\Omega)^d$ such that
	\begin{equation}\label{cV.Cauchy}
		\|f_n - f \|_{w} \to 0, \qquad \||V|^\frac12 f_n - g \|_{w}  \to 0, \qquad \|P_1^\frac12 \nabla f_n- h \|_{w}\to 0,
	\end{equation}
	as $n\to\infty$. It is sufficient to show that 
	\begin{equation}\label{cV.lim}
		g = |V|^\frac12 f, \qquad \nabla f \in \Loneloc (\Omega)^d, \qquad h = P_1^\frac12 \nabla f.
	\end{equation}
	Indeed, together with~\eqref{cV.Cauchy} this implies that $f \in \cV_w$ with $\|f_n - f\|_{\cV_w} \to 0$ and the claimed completeness follows. 
	
	To prove~\eqref{cV.lim}, for arbitrary $\varphi \in\Coo(\Omega)$ we use the second limit in~\eqref{cV.Cauchy} to argue
	\begin{equation}\label{Cauchy.V.lim.1}
		\begin{aligned}
				\langle w g, \varphi\rangle_{\cD'(\Omega) \times \cD(\Omega)} = 		\langle g, \varphi \rangle_{w}  = \lim_{n \to \infty} \langle |V|^\frac12 f_n, \varphi \rangle_{w} & = \langle  f,  |V|^\frac12  \varphi \rangle_{w} \\
				& =  ( w|V|^\frac12 f, \varphi )_{\cD'(\Omega) \times \cD(\Omega)}.
		\end{aligned}
	\end{equation}
	The third equality above is justified since $|V|^\frac12  \varphi \in \Ltw$ by  $V \in L^1_{\rm loc}(\Omega)$ and $w \in L^\infty_{\rm loc}(\Omega)$. It follows that $w g = w |V|^\frac12 f$ in $\cD'(\Omega)$ and thus in $L^1_{\rm loc}(\Omega)$. As $w>0$, we arrive at $g= |V|^\frac12 f$.
	
	It remains to show that $\nabla f$ is a regular distribution which satisfies the last identity in~\eqref{cV.lim}. For any $\psi \in \Coo(\Omega)^d$, by the first limit in~\eqref{cV.Cauchy} and $w^{-1} \in L^\infty_{\loc} (\Omega)$ (which guarantees that $w^{-1}(\nabla \cdot \psi) \in L^2_w(\Omega)^d$) we conclude
	\begin{equation}\label{Cauchy.nabla.1}
		- \int_\Omega f \overline{(\nabla \cdot \psi)} \, \dd x  = - \lim_{n\to \infty} \langle f_n, w^{-1} (\nabla \cdot \psi) \rangle_{w} = \lim_{n \to \infty} \langle \nabla f_n, \psi \rangle,
	\end{equation}
	i.e.~we obtain convergence $\nabla f_n \to \nabla f$ in $\Dd'(\Omega)^d$. Since $P_1^{-1}$ is positive definite and locally integrable, one moreover has
	\begin{equation}
		x \mapsto \|P_1^{-\frac12}(x)\|_2 = \|P_1^{-1}(x)\|_2^\frac12 \in \LtlocOm
	\end{equation}
	and thus the third limit in~\eqref{cV.Cauchy} gives (where again $w^{-1} P_1^{-\frac12} \psi \in L^2_w(\Omega)^d$ follows from $w^{-1} \in L^\infty_{\loc} (\Omega)$)
	\begin{equation}\label{Cauchy.nabla.2}
		\lim_{n\to \infty} \langle P_1^\frac12 \nabla f_n, w^{-1} P_1^{-\frac12}  \psi \rangle_{w} 
		=  \langle h,w^{-1}  P_1^{-\frac12}  \psi \rangle_{w}  =  \int_\Omega \langle P_1^{-\frac12} h ,  \psi \rangle_{\C^d} \, \dd x.
	\end{equation}
	From~\eqref{Cauchy.nabla.1} and~\eqref{Cauchy.nabla.2}, it follows that $\nabla f = P_1^{-\frac12} h \in \Loneloc(\Omega)^d$ and the last identity in \eqref{cV.lim} is immediate. 
	
	Since $\Coo(\Omega) \subset \cV_{w}$ as $P_1$ and $V$ are locally integrable and $w \in L^\infty_{\loc} (\Omega)$, the space $\cV_{w,0}$ is a well-defined Hilbert space. Clearly, it is boundedly embedded in $L^2_{w} (\Omega)$ and the density claim follows from the density of $\Coo(\Omega)$ in $L^2_{w} (\Omega)$, see e.g.~\cite[Ex.~1.5.3~(c)]{BEH}.
\end{proof}

\begin{lemma}[Cores of $\mathbf t_w$]
	\label{lem:W.1.infty}
	Let Assumption {\rm\ref{asm:main.LM}}~\ref{item:LM.reg} be satisfied, let $\Vv_{w,0}$ be as in \eqref{cV.0.def} and denote
	\begin{equation}\label{W.1.inf.comp}
		W^{1,\infty}_{\rm comp} (\Omega) := \Big\{ f \in W^{1,\infty} (\Omega) \, : \, \operatorname{ess}\supp f ~ {\rm is ~compact~ in } ~ \Omega\Big \} .
	\end{equation}
	Then one has the dense inclusions
	\begin{equation}
		w^{-1} \Coo(\Omega) \subset W^{1,\infty}_{\rm comp} (\Omega) \subset \Vv_{w,0}.
	\end{equation}
\end{lemma}

\begin{proof}
	The first inclusion is obvious since $w^{-1} \in W^{1,\infty}_{\loc}(\Omega)$. We prove the second one. Fix $f \in W^{1,\infty}_{\rm comp} (\Omega)$, then  $f \in \Vv_w$ since $P_1$ and $V$ are locally integrable and $w \in L^\infty_{\loc}(\Omega)$, see~\eqref{def.big.cV}. With $\eps>0$, let $\phi_\eps$ be a standard mollifier on $\R^d$, see~\cite[Def.\ 2.28]{Adams-2003}. For all $\eps\le \eps_0$ with $\eps_0$ small enough, we have 
	\begin{equation}
		f_\eps := f * \phi_\eps \in \Coo (\Omega), \qquad \supp f_\eps \subset \overline{B_{\eps_0} (\operatorname{ess} \supp f)} =: K \subset \Omega.
	\end{equation}
	Moreover, as $\eps \to 0$, we have $f_\eps\to f$ in $L^2 (\Omega)$ (and in $\Ltw$ since $w \in L^\infty_{\loc}(\Omega)$) and $\nabla f_\eps \to \nabla f$ in $L^2(\Omega)^d$, see~\cite[Thm.\ 2.29]{Adams-2003}. It further holds that
	\begin{equation}\label{moll.bdd}
		\|f_\eps\|_{L^\infty} \le \|f\|_{L^\infty}, \qquad \|\nabla f_\eps\|_{L^\infty} = \|\nabla f * \phi_\eps \|_{L^\infty} \le \|\nabla f\|_{L^\infty}.
	\end{equation}
	From the dominated convergence theorem, see e.g.~\cite[Thm.\ 1.50]{Adams-2003}, it follows that
	\begin{equation}
		\|P_1^\frac12 \nabla (f-f_\eps)\|_{w}^2 + \||V|^\frac12 (f-f_\eps) \|_{w}^2 + \|f-f_\eps \|_{w}^2  \to 0.
	\end{equation}
	Indeed, after extracting a subsequence such that $f_\eps$ and $\nabla f_\eps$ converge pointwise a.e., the existence of an integrable upper bound follows from \eqref{moll.bdd}, the uniform inclusion $\supp f_\eps \subset K$, $w\in L^\infty_{\loc}(\Omega)$ and the local integrability of $P_1$ and $V$.

	To see the density, we approximate a given $f \in \Coo(\Omega)$ by a sequence in $ w^{-1} \Coo(\Omega)$ with respect to the norm $\|\cdot \|_{\Vv_w}$. The claim  then follows by definition, see~\eqref{cV.0.def}. Let $\phi_\eps$, $\eps>0$, again be a standard mollifier and consider the functions 
	\begin{equation}
		f_\eps := \frac{(fw)* \phi_\eps}{w} \in w^{-1} \Coo(\Omega)
	\end{equation}
	(with compact support in a fixed neighbourhood of $\operatorname{ess} \supp f$ if $\eps$ is small enough). By $w^{-1} \in L^\infty_{\loc} (\Omega)$, we have that $f_\eps \to f$ in both $L^2(\Omega)$ and $\Ltw$ as $\eps \to 0$.  Moreover, again by the local boundedness of $w$, $w^{-1}$ and $\nabla w$ (which implies $\nabla (fw) \in L^2(\Omega)^d$), we also have 
	\begin{equation}
		\nabla f_\eps = \frac{\nabla ((fw)*\phi_\eps)}{w} - \frac{((fw)*\phi_\eps)\nabla w}{w^2} \to  \frac{\nabla(fw)}{w} - \frac{f\nabla w}{w} = \nabla f,
	\end{equation}
	where the convergence is in $L^2(\Omega)^d$. The justification of $f_\eps \to f$ in $\Vv_w$ now follows by dominated convergence similarly to the first part of the proof.
\end{proof}

\subsection{Weighted coercivity}

\begin{lemma}[Boundedness of $\mathbf t_w$] \label{lem:t.bdd}
Let Assumption~{\rm\ref{asm:main.LM}}~\ref{item:LM.reg}--\ref{item:asm.V.accr} and~\ref{asm:LM.w} be satisfied. The form $\mathbf t_w$ in~\eqref{t.def} is well-defined and bounded on ${\cV_{w,0}}$, see~\eqref{cV.0.def}.
\end{lemma}
\begin{proof}
	Observe first the following pointwise identity
	\begin{equation}\label{eq:sec.matrix}
		P = P_1 + \ii P_2 = P_1^\frac12 \left(I_{\C^d} + P_1^{-\frac12} P_2 P_1^{-\frac12} \right) P_1^\frac12.
	\end{equation}
	It is then easily seen from~\eqref{asm:P.sect} that (cf.~also~\cite[Thm.~VI.3.2]{Kato-1966})
	\begin{equation}\label{eq:P.bounds}
		\|P_1^{-\frac12} P_2 P_1^{-\frac12}\|_{L^\infty} \le C_P, \qquad  \| P_1^{-\frac12} P P_1^{-\frac12} \|_{L^\infty} \le 1+C_P.
	\end{equation}

	We show that $\mathbf t_w$ is well-defined and bounded on the larger space $\cV_w$, see~\eqref{def.big.cV} and~\eqref{cV.def}. By a polarisation argument, see e.g.~\cite[Lem.~IV.2.1]{EE}, it is sufficient to show the boundedness of the quadratic form. For arbitrary $f \in \cV_{w}$ one computes
\begin{equation}\label{P.term.1}
		\langle P \nabla f, \nabla (f w) \rangle  =\|P_1^\frac12 \nabla f\|_{w}^2 + \ii \langle P_2  \nabla f, \nabla f \rangle_{w}  + \langle P \nabla f, f\nabla w  \rangle.
\end{equation}
Using \eqref{asm:P.sect}, we have 
\begin{equation}\label{P.term.0}
|\langle  P_2 \nabla f,  \nabla f \rangle_{w} | \leq C_P \| P_1^\frac12 \nabla f\|_{w}^2 .
\end{equation}
Moreover, by~\eqref{eq:P.bounds},~\eqref{asm:w.sep}, Cauchy--Schwarz' and Young's inequalities, we obtain
\begin{equation}\label{P.term.2}
	\begin{aligned}
		|\iprod{ P \nabla f}{f\nabla w}| 
		& \le \|P_1^{-\frac12} P P_1^{-\frac12} \|_{L^\infty} \iprod{  |P_1^\frac{1}{2} \nabla f|}{|f| |P_1^{\frac12} \nabla w| } \\ 
		& \ls \iprod{ |P_1^\frac{1}{2} \nabla f|}{(\kappa_w V_1^\frac12 + \tau_w |V_2|^\frac12 + C_w) |f|}_{w} \\
		&  \lesssim  \| P_1^\frac{1}{2} \nabla f \|_{w}^2 +  \| V_1^\frac12 f\|_{w}^2 +  \||V_2|^\frac12 f\|_{w}^2 + \|f\|_{w}^2 \ls \|f\|_{\Vv_w}^2.
	\end{aligned}
\end{equation}
From \eqref{P.term.1}, \eqref{P.term.0} and \eqref{P.term.2} we see that
\begin{equation}
	|\langle P \nabla f, \nabla (f w) \rangle| \ls \|f\|_{\Vv_w}^2.
\end{equation} 
The boundedness of $\mathbf t_w$ on ${\cV_{w}}$ finally follows from
\begin{equation*}
|\langle  V f,  f \rangle_{w} | \leq \| |V|^\frac12 f\|_{w}^2 \le \|f\|_{\cV_{w}}^2. \qedhere
\end{equation*}
\end{proof}

\begin{lemma}[Boundedness of $\Phi$ on ${\cV_{w,0}}$]\label{lem:Phi.bdd}
Let Assumption~{\rm\ref{asm:main.LM}}~\ref{item:LM.reg}--\ref{asm:LM.nabla.Phi} be satisfied. Then the multiplication by $\Phi$ is a bounded operator on ${\cV_{w,0}}$, see~\eqref{cV.0.def}. 
\end{lemma}
\begin{proof}
We show the claimed boundedness on the larger space $\cV_w$, see~\eqref{def.big.cV} and \eqref{cV.def}. Since $|\Phi| \leq 1$, for all $f \in \cV_{w}$ one has
\begin{equation}\label{Phi.bdd.1}
	\||V|^\frac12 \Phi f \|_{w} + 	\| \Phi f \|_{w} \le \|f\|_{\cV_{w}}
\end{equation}
and moreover, employing \eqref{asm:Phi.sep} in the second step,
\begin{equation}\label{Phi.bdd.2}
	\begin{aligned}
		\|  P_1^\frac12 \nabla (\Phi f)\|_{w}  & \leq 	\|  P_1^\frac12 \nabla f\|_{w}  + \|  f P_1^\frac12 \nabla \Phi \|_{w}\\
		 & \ls \|  P_1^\frac12 \nabla f\|_{w} + \||V|^\frac12 f\|_w + \|f\|_w \ls \|f\|_{\Vv_w}.
	\end{aligned}
\end{equation}
The claim follows by combining the above estimates.
\end{proof}

\begin{proposition}[Generalised weighted coercivity of $\mathbf t_w$]
	\label{prop:g.coer}
Let Assumption~{\rm\ref{asm:main.LM}} hold with~\eqref{eq:rem.rel.bdd.eps} replaced by~\eqref{eq:rem.rel.bdd} and let $\beta$ be as in~\eqref{ellipse.gen.beta} and~\eqref{ellipse.gen}. Then there exist $m_1,m_2,\gamma_1,\gamma_2 >0$,  depending (continuously) only on $\beta$ and the constants $\eps$, $\kappa_w$, $\tau_w$, $\vartheta_R$, $\kappa_R$, $\tau_R$, $C_{\Phi,\eps}$, $C_w$, $C_R$ and $C_P$, such that, for all $f \in \Coo(\Omega)$,
\begin{equation}\label{t.gcoer.1}
	\begin{aligned}
		& \Re \mathbf t_w(f,f) + \Im \mathbf t_w(\beta \Phi f,f) + \gamma_1 \| f\|_{w}^2  \geq m_1 \|f\|^2_{\Vv_w}, \\
		& \Re \mathbf t_w(f,f) + \Im \mathbf t_w(f,\beta \Phi f) + \gamma_2 \| f\|_{w}^2 \geq m_2 \|f\|^2_{\Vv_w}. \\  
	\end{aligned}
\end{equation}
\end{proposition}

\begin{proof}
First, taking the real part of~\eqref{t.def} with $f=g$ gives
\begin{equation} \label{re.est}
\Re \mathbf t_w(f,f) = \| P_1^\frac12 \nabla f\|_{w}^2 + \re \langle P \nabla f, f \nabla w \rangle + \|  V_1^\frac12f \|_{w}^2
\end{equation}
and we see that for the desired lower bound, the term in the middle of the right hand side needs to be controlled by the terms in $\|f\|_{\cV_w}^2$. Taking the imaginary part of~\eqref{t.def} with $(f,g) = (\Phi f,f)$, applying the product rule and~\eqref{eq:phi.sgn.rem} lead to
\begin{equation}\label{ImP.0}
	\begin{aligned}
		\im \mathbf t_w (\Phi f,f) & \geq \|  |V_2|^\frac 12 f\|_{w}^2  + \im \langle\Phi P  \nabla f , f \nabla w  \rangle  - \|R^\frac12 f\|_w^2 \\
		& \qquad  -  |\langle \Phi  P_2  \nabla f,  \nabla f \rangle_{w}|
		 - |\langle f P \nabla \Phi,  \nabla f \rangle_{w}| - |\langle f P \nabla \Phi , f \nabla w \rangle|.
	\end{aligned}
\end{equation}

While the last four terms on the right hand side will be estimated separately, the second term (times $\beta$) can be combined with the middle term in the right hand side of~\eqref{re.est}. Indeed, one  computes
\begin{equation}
	\begin{aligned} 
		\re \langle P \nabla f, f \nabla w \rangle +  \im \langle\beta\Phi P  \nabla f , f \nabla w \rangle & = \re \langle P \nabla f, f \nabla w \rangle - \re \left( \ii \langle\beta\Phi P  \nabla f , f \nabla w \rangle \right) \\
		& = \re \langle (1-\ii \beta \Phi) P \nabla f, f \nabla w \rangle. 
	\end{aligned}
\end{equation}
We can thus employ $|\Phi| \le 1$, inequalities~\eqref{eq:P.bounds} and~\eqref{asm:w.sep}, together with Cauchy--Schwarz' and Young's inequalities with $\delta_1, \delta_2, \delta >0$ (to be chosen suitably) to estimate the above as follows
\begin{equation}\label{real.im.est}
	\begin{aligned}
		& | \re \langle P \nabla f,  f \nabla w \rangle + \im \langle\beta \Phi P  \nabla f , f \nabla w \rangle| \\ 
		& \qquad \qquad  \le \|1-\ii \beta  \Phi\|_{L^\infty} \|P_1^{-\frac12} P P_1^{-\frac12} \|_{L^\infty} \langle |P_1^\frac12 \nabla f|, |f| |P_1^\frac12 \nabla w |  \rangle \\
		& \qquad \qquad  \le \sqrt{1+\beta^2} (1+C_P) \langle |P_1^\frac12 \nabla f|, (\kappa_w V_1^\frac12 + \tau_w |V_2|^\frac12 + C_w ) |f| \rangle_{w} \\
		& \qquad \qquad  \le \sqrt{1+\beta^2}(1+C_P)   \left( (\delta_1 + \delta_2 + \delta) \| P_1^\frac12 \nabla f\|^2_{w} + \frac{\kappa_w^2}{4\delta_1} \| V_1^\frac12 f \|_{w}^2 \right. \\
		& \qquad \qquad \qquad \qquad \qquad \qquad \qquad \qquad \left. + \frac{\tau_w^2}{4\delta_2} \| |V_2|^\frac12 f \|_{w}^2 + \frac{C_w^2}{4\delta} \| f \|_{w}^2  \right).
	\end{aligned}
\end{equation}

Next we estimate the remaining terms on the right hand side of~\eqref{ImP.0}. While the third term satisfies~\eqref{eq:rem.rel.bdd}, we deal with the fourth term by using $|\Phi|\le 1$ and~\eqref{asm:P.sect} to obtain
\begin{equation}\label{ImP.sec}
		|\langle \Phi P_2  \nabla f ,  \nabla f \rangle_{w}|
		\leq 	C_P \langle  P_1  \nabla f ,  \nabla f \rangle_{w} =  C_P \| P_1^\frac12 \nabla f\|_{w}^2.
\end{equation}
For the fifth term, we use~\eqref{eq:P.bounds} and~\eqref{asm:Phi.sep} together with $|V|^\frac12 \le V_1^\frac12 + |V_2|^\frac12$, as well as Cauchy--Schwarz' and Young's inequalities to arrive at 
\begin{equation}\label{Im.est}
	\begin{aligned}
		|\langle f P \nabla \Phi ,  \nabla f \rangle_{w}| & \le \|P_1^{-\frac12} P P_1^{-\frac12}\|_{L^\infty} \langle  |f| |P_1^{\frac12} \nabla \Phi|,  |P_1^\frac12 \nabla f| \rangle_{w} \\
		& \le (1+C_P) \langle   ( \eps V_1^\frac12 + \eps |V_2|^\frac12  + C_{\Phi,\eps}) |f|,  |P_1^\frac12 \nabla f|  \rangle_{w} \\
		& \le (1+C_P) \bigg( (\eps + \delta) \| P_1^\frac12 \nabla f\|^2_{w} + \frac\eps2 \| V_1^\frac12 f \|_{w}^2 \\
		& \qquad \qquad\qquad \qquad \qquad \quad + \frac\eps2 \| |V_2|^\frac12 f \|_{w}^2 + \frac{C_{\Phi,\eps}^2}{4\delta} \| f \|_{w}^2 \bigg)  .
	\end{aligned}
\end{equation}
The sixth term is estimated by means of~\eqref{eq:P.bounds},~\eqref{asm:Phi.sep} (again with $|V|^\frac12 \le V_1^\frac12 + |V_2|^\frac12$) and~\eqref{asm:w.sep}, Cauchy--Schwarz' and Young's inequalities as follows
\begin{equation}\label{Im.est.2}
\begin{aligned}
		& |\langle f P \nabla \Phi, f \nabla w \rangle| \\
		& \qquad  \le\|P_1^{-\frac12} P P_1^{-\frac12} \|_{L^\infty} \langle |f| |P_1^\frac12 \nabla \Phi|, |f| |P_1^{\frac12} \nabla w| \rangle \\
		&  \qquad \le (1+C_P) \langle (\eps V_1^\frac12 + \eps |V_2|^\frac12 + C_{\Phi,\eps})|f| ,  (\kappa_w V_1^\frac12 + \tau_w |V_2|^\frac12 + C_w)|f|\rangle_{w} \\
		& \qquad = (1+C_P) \, \bigg\{ \eps \kappa_w \| V_1^\frac12 f \|_{w}^2 + \eps \tau_w \| |V_2|^\frac12 f \|^2_{w} + C_{\Phi,\eps} C_w \|f\|_{w}^2 \\
		& \qquad \qquad  \qquad + \eps \left( \langle  V_1^\frac12 |f|, ( \tau_w |V_2|^\frac12 + C_w)|f|\rangle_{w} +  \langle  |V_2|^\frac12 |f|, (\kappa_w V_1^\frac12 + C_w)|f|\rangle_{w} \right)   \\
		& \qquad \qquad \qquad +  \langle C_{\Phi,\eps}  |f| ,  (\kappa_w V_1^\frac12 + \tau_w |V_2|^\frac12)|f|\rangle_{w} \bigg\}  \\
		& \qquad \le (1+C_P) \, \bigg\{ \eps \kappa_w \| V_1^\frac12 f \|_{w}^2 + \eps \tau_w \| |V_2|^\frac12 f \|^2_{w} + C_{\Phi,\eps} C_w \|f\|_{w}^2 \\
		& \qquad \qquad  \qquad + \eps \bigg( \bigg(1 + \frac{\kappa_w^2}2 \bigg) \|V_1^\frac12 f\|^2_{w}  + \bigg(1 + \frac{\tau_w^2}2 \bigg) \||V_2|^\frac12 f\|^2_{w}  +  C_w^2 \|f\|_{w}^2 \bigg) \\
		& \qquad \qquad \qquad + \delta \|V_1^\frac12 f\|^2_{w} + \delta \||V_2|^\frac12 f\|_{w}^2 + \frac{C_{\Phi,\eps}^2\kappa_w^2+ C_{\Phi,\eps}^2\tau_w^2}{4\delta}   \| f\|_{w}^2  \bigg\}. 
\end{aligned}
\end{equation}
The last inequality can be written as
\begin{equation}\label{ImP.2nd.2}
		|\langle f P \nabla \Phi, f \nabla w \rangle| \le  \eta_\kappa \| V_1^\frac12 f\|_{w}^2 + \eta_\tau \|| V_2|^\frac12 f\|_{w}^2 + C_{\eps,\delta}\|f\|_{w}^2 
\end{equation}
where we have set 
\begin{equation}\label{eta.5}
	\begin{aligned}
	\eta_\kappa & := (1+C_P) \left( \eps \left( 1 + \kappa_w + \frac{\kappa_w^2}{2} \right) + \delta \right),	\\
	\eta_\tau & := (1+C_P) \left( \eps \left( 1 + \tau_w + \frac{\tau_w^2}{2} \right) + \delta \right),	\\
	C_{\eps,\delta} & := (1+C_P) \left( C_{\Phi,\eps} C_w + \eps C_w^2 + \frac{C_{\Phi,\eps}^2\left(\kappa_w^2+\tau_w^2\right)}{4 \delta}\right). 
	\end{aligned}
\end{equation}
Note that $\eta_\kappa = \BigO(\eps+\delta)$ and $\eta_\tau = \BigO(\eps+\delta)$ as $\eps+\delta \to 0^+$.
\noeqref{real.im.est,ImP.sec,Im.est}
Combining the inequalities~\eqref{re.est}--\eqref{ImP.2nd.2} and~\eqref{eq:rem.rel.bdd}, we finally arrive at an estimate for the desired quantity which reads
\begin{equation}\label{gen.coer.final}
	\begin{aligned}
		& \re \mathbf t_w (f,f) + \im \mathbf t_w (\beta \Phi f,f) + \widetilde C_{\eps,\delta} \|f \|_{w}^2 \\
		& \qquad \qquad \quad  \ge \left( 1 - \eta_1 \right) \|P_1^\frac12 \nabla f \|_{w}^2  + (1- \eta_2) \| V_1^\frac 12 f \|_{w}^2 + (\beta - \eta_3) \| |V_2|^\frac12 f \|_{w}^2 
	\end{aligned}
\end{equation}
with the constants 
\begin{equation}\label{gcoer.mu.alph}
	\begin{aligned}
	\eta_1 & := \beta ( C_P + \vartheta_R) + (1+C_P) \sqrt{1+\beta^2} (\delta_1 + \delta_2)
	\\
	& \qquad \qquad \qquad \qquad \qquad + (1+C_P) \left( \beta \eps + (\sqrt{1+\beta^2} + \beta) \delta  \right), \\
	\eta_2 & := \beta \kappa_R +  \frac{ (1+C_P) \sqrt{1+\beta^2}\kappa_w^2}{4 \delta_1} + \beta \left( \frac{1+C_P }2 \eps+ \eta_\kappa \right),
	\\
	\eta_3 & := \beta \tau_R + \frac{ (1+C_P)  \sqrt{1+\beta^2} \tau_w^2}{4 \delta_2} 
	+ \beta \left( \frac{1+C_P }2 \eps+ \eta_\tau \right), 
	\\
	\widetilde C_{\eps,\delta} & := \beta C_R + (1+C_P)   \frac{ \sqrt{1+\beta^2}C_w^2+ \beta C_{\Phi,\eps}^2}{4 \delta}  + \beta C_{\eps,\delta}. \\
	& 
	\end{aligned}
\end{equation}

Since $V_1+|V_2| \ge |V|$, it remains to explain that the parameters $\delta_1, \delta_2, \delta, \eps$ can be selected such that $\eta_1, \eta_2<1$ and $\eta_3<\beta$. Considering that the $\eta_i = \eta_i(\eps,\delta)$ are continuous in $\eps$ and $\delta$ and setting
\begin{equation}
	\label{eq:gcoer.mu}
	\begin{aligned}
		\mu_1 & := \eta_1(0,0) =  \beta ( C_P + \vartheta_R) + (1+C_P) \sqrt{1+\beta^2} (\delta_1 + \delta_2), \\
		\mu_2 &  := \eta_2(0,0) = \beta \kappa_R + \frac{  (1+C_P) \sqrt{1+\beta^2}\kappa_w^2}{4 \delta_1},  \\
		\mu_3 & := \eta_3(0,0) = \beta \tau_R +  \frac{(1+C_P)  \sqrt{1+\beta^2} \tau_w^2}{4 \delta_2},
	\end{aligned}
\end{equation}
this can be done if $\mu_1, \mu_2 <1 $ and $\mu_3 <\beta$. According to Lemma~\ref{lem:ellipse} below, however, the latter are equivalent to~\eqref{ellipse.gen.beta} and~\eqref{ellipse.gen} in the assumptions.
Hence, the constants in the first line of~\eqref{t.gcoer.1} can be chosen as
\begin{equation}\label{eq:m.gamma.1}
	m_1 := \min \{1-\eta_1, 1-\eta_2, \beta-\eta_3\} >0, \quad \gamma_1 := \widetilde C_{\eps,\delta} + m_1 >0.
\end{equation}
Their continuous dependence on the parameters therein is obvious.

To verify the second inequality in \eqref{t.gcoer.1}, we take the imaginary part of~\eqref{t.def} with $(f,g) = (f,\Phi f)$ and employ~\eqref{eq:phi.sgn.rem} to estimate 
\begin{equation}\label{ImP.2nd.0}
	\begin{aligned}
\Im \mathbf t_w(f, \Phi f) & 
\geq \|  |V_2|^\frac 12 f\|_{w}^2 + \im \langle P \nabla f, \Phi f \nabla w \rangle
 -  \|R^\frac12 f\|_w^2 \\
 &\qquad \qquad \qquad \qquad \quad  - |\langle   P_2  \nabla f,  \Phi \nabla f \rangle_{w}| - |\langle  P  \nabla f, f \nabla \Phi \rangle_{w}| .
 \end{aligned}
\end{equation}
The terms on the right hand side are analogous to the ones estimated in the previous part (with the exception of~\eqref{ImP.2nd.2} which does not appear here). 
Using~\eqref{re.est}, \eqref{ImP.2nd.0}, \eqref{real.im.est}, \eqref{eq:rem.rel.bdd}, \eqref{ImP.sec} and an  inequality analogous to~\eqref{Im.est}, we arrive at
\begin{equation} 
	\begin{aligned}
		& \re \mathbf t_w (f,f) + \im \mathbf t_w(f, \beta \Phi f) + \widehat C_{\eps,\delta} \|f\|_{w}^2 \\
		& \qquad \qquad \qquad   \ge \left( 1 - \eta_{1} \right) \| P_1^\frac12 \nabla f\|_{w}^2  + ( 1 - \widetilde \eta_{2} ) \|V_1^\frac12 f\|_{w}^2 + (\beta - \widetilde \eta_{3}) \| |V_2|^\frac12 f\|_{w}^2 \\
	\end{aligned}
\end{equation}
with $\eta_1$ as in~\eqref{gcoer.mu.alph} and the modified constants
\begin{equation}
	\begin{aligned}
		\widetilde \eta_2 & := \beta \kappa_R +   \frac{(1+C_P)\sqrt{1+\beta^2} \kappa_w^2}{4 \delta_1} +   \frac{\beta (1+C_P)}2\eps  && < \eta_2,
		\\
		\widetilde \eta_3 & := \beta \tau_R + \frac{(1+C_P) \sqrt{1+\beta^2}\tau_w^2}{4 \delta_2} + \frac{\beta (1+C_P)}2 \eps && < \eta_3, 
		\\
		\widehat C_{\eps,\delta} & := \beta C_R + (1+C_P)  \frac{\sqrt{1+\beta^2}C_w^2+ \beta C_{\Phi,\eps}^2}{4 \delta} && \le \widetilde C_{\eps,\delta}.
	\end{aligned}
\end{equation}
In fact, the above correspond to setting $\eta_\kappa = \eta_\tau = C_{\eps,\delta}=0$ in~\eqref{gcoer.mu.alph}. Analogously to the first part of the proof, it follows that one can choose the parameters  $\delta_1, \delta_2, \delta,\eps>0$ such that $\eta_1<1$, $\widetilde \eta_2<1$ and $\widetilde \eta_3<\beta$. Hence, there exist $m_2, \gamma_2> 0$ such that the second inequality in~\eqref{t.gcoer.1} holds (one can in particular choose $m_2:=m_1$ and $\gamma_2:=\gamma_1$ as above). The proof is complete.
\end{proof}

\begin{lemma}
	\label{lem:ellipse}
	Let $\kappa_w,\tau_w,\vartheta_R,\kappa_R,\tau_R,C_P\ge 0$ and let $\beta$ be as in \eqref{ellipse.gen.beta}. Then condition~\eqref{ellipse.gen} is equivalent to the existence of $\delta_1, \delta_2>0$ such that $\mu_1<1$, $\mu_2<1$ and $\mu_3 <\beta$, where $\mu_1$, $\mu_2$ and $\mu_3$ are defined as in~\eqref{eq:gcoer.mu}. 	
\end{lemma}

\begin{proof}
	The condition $\mu_1 <1$ can be written as
		\begin{equation}\label{delta.ell1}
		\delta_1 + \delta_2 < \frac{1-\beta(C_P + \vartheta_R)}{(1+C_P)\sqrt{1 + \beta^2}},
	\end{equation} 
	while $\mu_2<1$ and $\mu_3<\beta$ are equivalent to
	\begin{equation}\label{delta.ell2}
		\delta_1 > \frac{(1+C_P)\sqrt{1+\beta^2}\kappa_w^2}{4 (1-\beta\kappa_R)}, \qquad 
		\delta_2 > \frac{(1+C_P)\sqrt{1+\beta^2}\tau_w^2}{4\beta (1-\tau_R)}.
	\end{equation}
	While~\eqref{delta.ell2} describes an open quadrant in the $(\delta_1,\delta_2)$-plane, \eqref{delta.ell1} is the equation of an open half plane. Their intersection is non-empty if and only if the corner of the quadrant lies in the half plane, which is precisely the condition~\eqref{ellipse.gen}.
\end{proof}

\begin{remark}
	\label{rem:ecrit}
	
	The value of the critical constant $\ec$ in~\eqref{asm:Phi.sep}, see Remark~\ref{rem:gcoer}~\ref{item.ecrit}, can be determined from the proof of Proposition~\ref{prop:g.coer}. The arising restrictions on $\eps$ amount to 
 	\begin{equation}\label{delta.small.2}
 	\begin{aligned}
 		\eps & < (1-\mu_1) \frac1{\beta(1+C_P)}  , \\
 		\eps & < (1-\mu_2)\frac{2}{\beta (1+C_P)(3 + 2 \kappa_w + \kappa_w^2)} , \\
 		\eps  &  < (\beta-\mu_3)\frac{2}{\beta (1+C_P)(3 + 2 \tau_w + \tau_w^2)},
 	\end{aligned}
 	\end{equation}
	where $\mu_1,\mu_2 \in (0,1)$ and $\mu_3 \in (0,\beta)$ are as in~\eqref{eq:gcoer.mu}. Indeed, this can be seen by setting $\delta=0$ in the constrains $\eta_1,\eta_2<1$ and $\eta_3<\beta$, see~\eqref{gcoer.mu.alph}. The critical value $\ec$ can be obtained by maximising the above constraints with respect to the admissible $\delta_1$, $\delta_2$ and $\beta$ as in~\eqref{ellipse.gen.beta},~\eqref{delta.ell1},~\eqref{delta.ell2}.
\end{remark}

\begin{proof}[Proof of Theorem~{\rm\ref{thm:t}}]	
	We first remark that, assuming~\eqref{eq:rem.rel.bdd.eps} together with the existence of $0<\beta<1/C_P$ satisfying~\eqref{ellipse} imply the more general condition~\eqref{eq:rem.rel.bdd} with arbitrarily small constants $\vartheta_R$, $\kappa_R$, $\tau_R$ and thus also~\eqref{ellipse.gen.beta},~\eqref{ellipse.gen}.	Hence, Proposition~\ref{prop:g.coer} applies under the present assumptions.
	
	Let $\mathbf t_w$ be as in~\eqref{t.def} and ${\cV_{w,0}}$ as in~\eqref{cV.0.def}. By Lemmas~\ref{lem.cV.compl} and~\ref{lem:t.bdd}, the form $\mathbf t_w$ is bounded on the Hilbert space $\Vv:= \Vv_{w,0}$, the latter being dense and boundedly embedded in $\Hh:=\Ltw$. By Lemma~\ref{lem:Phi.bdd} and since $|\Phi|\le1$, the multiplier
	\begin{equation}
		\Phi_1=\Phi_2:=\beta \Phi \in \Bb(\Vv_{w,0})
	\end{equation}
	extends to a bounded operator on $\Ltw$ and we see that Theorems~\ref{thm:lm.0} and~\ref{thm:lm.1} apply with 
	\begin{equation}
		\mathbf a := \mathbf t_w + 
		\gamma \|\cdot \|_w^2, \qquad \gamma:=\max \{\gamma_1,\gamma_2\}, \qquad  m:= \min\{m_1,m_2\};
	\end{equation}
	note hereby that by the continuity of $\mathbf t_w$ on $\Vv_{w,0}$ it is sufficient to verify~\eqref{lm.gen.coev} on the dense subspace $\Coo(\Omega)$, that we have \eqref{t.gcoer.1}, and that
		\begin{equation}
			|\mathbf a(f,f)| + |\mathbf a(\beta \Phi f,f)| \geq \Re \mathbf t_w(f,f) + \Im \mathbf t_w(\beta \Phi f, f) + \gamma \|f\|_{w}^2,
		\end{equation}
	and analogously for the second inequality in~\eqref{lm.gen.coev} and~\eqref{t.gcoer.1}. It follows from~\eqref{lm.op} that the operator $A$ in $\Ltw$ given by
	\begin{equation}\label{eq:A.descr}
		\begin{aligned}
			\Dom(A) & = \Big\{  f \in {\cV_{w,0}} \,: \, \exists~\eta_f \in L^2_{w} (\Omega), \,\, \forall~ g \in \cV_{w,0}, \, \,\mathbf t_w (f,g) = \langle \eta_f, g \rangle_{w} \Big\}, \\
			A f &= \eta_f,
		\end{aligned}
	\end{equation}
	is closed, has non-empty resolvent set and dense domain in both $\Vv_{w,0} $ and $\Ltw$
	(the independence on the shift $\gamma$ of the operator domain and action  are easy to justify).
	
	It remains to prove that $A=T_w$ as in~\eqref{T.descr}. To this end, notice first that $-\nabla \cdot (P\nabla) + V$ in the sense of distributions is well-defined on whole $\Vv_{w}$; this follows in a standard way from Assumption~\ref{asm:main.LM}~\ref{item:LM.reg} and~\ref{item:asm.P}, in particular from $w^{-1} \in L^\infty_{\loc}(\Omega)$, the local integrability of $P_1$ and $V$, as well as~\eqref{eq:P.bounds} due to the sectoriality of $P$.
	Fix $f \in \Vv_{w,0}$. By Lemma~\ref{lem:W.1.infty} and~\eqref{eq:A.descr},  $f \in \Dom (A)$ if and only if there exists $\eta_f \in \Ltw$ such that 
	\begin{equation}
		\forall ~w^{-1}\varphi \in w^{-1}\Coo(\Omega) \,\,\, : \,\,\, \mathbf t_w (f,w^{-1}\varphi ) = \langle \eta_f, w^{-1} \varphi \rangle_{w} = \langle \eta_f, \varphi \rangle.
	\end{equation}
	However, since it holds  that
	\begin{equation}
	\begin{aligned}
		\langle \eta_f, \varphi \rangle  =	\mathbf t_w (f,w^{-1} \varphi) &  =\langle P \nabla f, \nabla (ww^{-1}\varphi) \rangle + \langle V f, w^{-1} \varphi \rangle_w \\
		& =   \int_\Omega \langle P\nabla f, \nabla \varphi \rangle_{\C^d} \, \dd x + \int_\Omega Vf \ov \varphi \, \dd x 
		= ( T_w f, \varphi )_{\cD'(\Omega) \times \cD(\Omega)}
	\end{aligned}
	\end{equation}
	 for any $\varphi \in \Coo(\Omega)$,
	 we have that $f \in \Dom (A)$ if and only if the distribution $T_w f =-\nabla \cdot (P \nabla f) + Vf$ is regular, belongs to $\Ltw$ and is equal to $\eta_f$. Consequently, $A=T_w$ and all claims are proven.
\end{proof}

\begin{proof}[Proof of Corollary~{\rm\ref{cor:res.bd}}]
		Given an arbitrary $\delta\in (0,1/C_P)$, it is clear that one can choose $\tau_w, \kappa_w, \beta < \delta$ such that~\eqref{ellipse} holds. By Theorem~\ref{thm:t}, $T_w$ has non-empty resolvent set and is densely defined in $\cV_{w,0}$ and $\Ltw$.
		
		To see the resolvent bound, consider $\la=r \e^{\ii \varphi}$ with a fixed angle $\varphi\in (\pi/2, 3\pi/2)$ and $r>0$.
		Applying Proposition~\ref{prop:g.coer}, it then follows from the second inequality in~\eqref{t.gcoer.1} that, for all $f \in \Vv_{w,0}$, 
		\begin{equation}
			\begin{aligned}
				& \re (\mathbf t_w - \la \|\cdot\|_w^2) (f,f) + \beta \im (\mathbf t_w - \la \|\cdot\|_w^2) ( f, \Phi f) \\
				& \qquad \qquad \qquad \ge 
				m_2 \|f\|^2_{\cV_w} + \big(-\gamma_2 + r (|\cos \varphi| - \beta |\sin\varphi|)  \big) \|f\|_w^2.
			\end{aligned}
		\end{equation}
		Upon selecting $\kappa_w$, $\tau_w$ and $\beta$ small enough such that $|\cos \varphi| - \beta |\sin\varphi| >0$, applying Theorem \ref{thm:lm.0} to $\mathbf a := \mathbf t_w - \la \|\cdot\|_w^2$, it follows that 
		\begin{equation}
			\label{eq:T.hat.res}
			(\widehat T_w - \la \iota_w^* \iota_w )^{-1} \in \Bb(\cV_{w,0}^*,\cV_{w,0}), \qquad \iota_w:= I_{\Vv_{w,0} \to L^2_w(\Omega)},
		\end{equation}
		for $r$ large enough, which in turn gives $\la \in \rho(T_w)$ by Theorem \ref{thm:lm.1}. Whenever $f \in \Dom (T_w)$, using Cauchy-Schwarz' inequality we further derive
		\begin{equation}
			\begin{aligned}
			\|(T_w- \la) f\|_w \|f\|_w & \gs |\langle (T_w- \la) f, f\rangle_w| + \beta | \langle (T_w- \la) f, \Phi f\rangle_w| \\
				& = |(\mathbf t_w - \la) (f,f)| + \beta| (\mathbf t_w - \la) (f,\Phi f)| \gs r \|f\|_w^2
			\end{aligned}
		\end{equation}
		for sufficiently large $r$. This implies the claimed resolvent bound.
\end{proof}

\subsection{Assumptions on $\Phi$ in terms of $V$}

\begin{lemma}
	\label{lem:R}
	Let $R \in L^1(\R)$ with $\operatorname{ess} \supp R \subset [x_0,x_0]$ for some $x_0>0$, and let $0<w \in L^\infty_{\rm loc}(\R)$ such that $w^{-1} \in L^\infty_{\rm loc}(\R)$. Then for every $\eps >0$ there exists $C_\eps\ge 0$ such that
	\begin{equation}
		\|R^\frac12 f\|_w^2 \le \eps \|f'\|_w^2 + C_\eps \|f\|_w^2, \qquad f \in \Coo(\R).
	\end{equation}
\end{lemma}

\begin{proof}
	Consider $\eta \in C_0^\infty(\R)$ satisfying $0 \leq \eta \leq 1$, $\eta =1$ on $[-x_0,x_0]$ and $\eta = 0$ outside $(-2x_0, 2 x_0)$. It is well-known that for every $\eps>0$ there exists $C_\eps \ge0$ such that 
	\begin{equation}
		\|\eta f \|_{L^\infty}^2 \leq \eps \|(\eta f)'\|^2 + C_\eps\| \eta f\|^2, \qquad f \in C_0^\infty(\R).
	\end{equation}
	Hence we arrive at 
	\begin{equation}
		\begin{aligned}
			\|R^\frac12 f \|_w^2 &\leq \|w\|_{L^\infty(-x_0,x_0)} \|R\|_{L^1} \|\eta f\|^2_{L^\infty} \\
			& \le C_{w,R} \left(\eps \|(\eta f)'\|^2 + C_\eps \|\eta f\|^2 \right)
			\\
			&	
			\leq C_{w,R} \|w^{-1}\|_{L^\infty(-2x_0,2x_0)} \left(\eps \|(\eta f)'\|^2_w + C_\eps \|\eta f\|^2_w\right)
			\\
			& \leq \eps C_{w,R,\eta} \|f'\|_w^2 + C_{w,R,\eta,\eps} \|f\|_w^2
		\end{aligned}
	\end{equation}
	for all $f \in C_0^\infty(\R)$.
\end{proof}

\begin{proof}[Proof of Proposition~{\rm\ref{prop:Phi.AH-like}}]
	By construction $|\Phi| \leq 1$. Since
	\begin{equation}
		\begin{aligned}
			\nabla \Phi &=  \frac{\nabla V_{2,r}}{(1+V_{1,r}^2 + V_{2,r}^2)^\frac12} -
			\frac{V_{2,r}(V_{1,r} \nabla V_{1,r} + V_{2,r} \nabla V_{2,r}  )}{(1+V_{1,r}^2 + V_{2,r}^2)^\frac32} 
			\\
			& = \frac{(1+ V_{1,r}^2) \nabla V_{2,r} - V_{1,r} V_{2,r} \nabla V_{1,r}}{(1+V_{1,r}^2 +  V_{2,r}^2)^\frac32}
		\end{aligned}
	\end{equation} 
	and $\nabla V_{j,r} \in L^1_{\rm loc} (\Omega)^d$, it follows that $\nabla \Phi \in \Loneloc(\Omega)$. Moreover, we have from \eqref{V.nabla.asm}
	\begin{equation}
		|P_1^\frac 12\nabla \Phi| \leq 
		\frac{(1+ V_{1,r}^2) |P_1^\frac 12 \nabla V_{2,r}|  + V_{1,r}|V_{2,r}| | P_1^\frac 12 \nabla V_{1,r}|}{(1+V_{1,r}^2 + V_{2,r}^2)^\frac32}
		\leq 	\eps  |V|^\frac12 +  C_{\nabla,\eps}; 
	\end{equation}
	thus \eqref{asm:Phi.sep} is satisfied.
	Regarding \eqref{eq:phi.sgn.rem} and \eqref{eq:rem.rel.bdd}, we use $|\Phi|\le 1$ and estimate
	\begin{equation}
		\begin{aligned}
			\Phi V_2  = \Phi V_{2,r} + \Phi V_{2,s} & \geq \frac{V_{2,r}^2}{(1+V_{1,r}^2 + V_{2,r}^2)^\frac12} - |V_{2,s}|
			\\
			& \ge |V_{2,r}| - \frac{1 + V_{1,r}^2}{(1+V_{1,r}^2 + V_{2,r}^2)^\frac12} - |V_{2,s}| 
			\\
			& \geq 
			|V_{2}| - V_{1,r} -2 |V_{2,s}|-1. 
		\end{aligned}
	\end{equation}
	Hence, employing \eqref{V1.rs.delta} and \eqref{V2s.rel.bdd}, we obtain that  $R:=V_{1,r} +2 |V_{2,s}|+1$ satisfies \eqref{eq:rem.rel.bdd} with the claimed constants.
\end{proof}

\subsection{Boundedness of compositions}

\begin{lemma}[Extension property for comparable weights]
	\label{lem:ext}
	Let Assumption~{\rm\ref{asm:main.LM}}~\ref{item:LM.reg}--\ref{item:asm.V.accr} and~\ref{asm:LM.w} hold for two weights $w_1$ and $w_2$ satisfying $w_1 \le C w_2$ with some $C>0$. Let $T_{w_1}$ in $L^2_{w_1}(\Omega)$ and $T_{w_2}$ in $L^2_{w_2}(\Omega)$ be defined as in \eqref{T.descr} with the respective weights. Then $T_{w_1}$ extends $T_{w_2}$, more precisely,
	\begin{equation}
		T_{w_2} f = T_{w_1} f, \qquad f \in \Dom(T_{w_2}) \subset \Dom (T_{w_1}).
	\end{equation}
\end{lemma}
\begin{proof}
	Clearly, the assumptions on $w_1$ and $w_2$ imply that $L_{w_2}^2(\Omega) \subset L_{w_1}^2(\Omega)$, and it is easy to see that also $\Vv_{w_2} \subset \Vv_{w_1}$ and $\Vv_{w_2,0} \subset \Vv_{w_1,0}$, see \eqref{def.big.cV} and \eqref{cV.0.def}. The claim then readily follows from \eqref{T.descr}.
\end{proof}

\begin{lemma}[Cut-off of admissible weight]
	\label{lem:w.cut-off}
	Let Assumption~{\rm\ref{asm:main.LM}}~\ref{item:LM.reg}, \ref{item:asm.V.accr} and~\ref{asm:LM.w} hold. Then for any $n \in \N$
	\begin{equation}\label{chi}
		w_n := \chi_n (w), \qquad \chi_n : [0,\infty) \to [0,\infty)  , \qquad \chi_n (x) := 
		\begin{cases}
			x, & x \in [0,n], \\
			n, & x \in (n, \infty),
		\end{cases}
	\end{equation}
	is an admissible weight. More precisely, $w_n \in W^{1,\infty}_{\rm loc} (\Omega) \cap L^\infty (\Omega)$ is uniformly positive on compact subsets of $\Omega$ and
	it satisfies \eqref{asm:w.sep} with the same constants $\kappa_w$, $\tau_w$ and $C_w$ as the original weight $w$. 
\end{lemma}
\begin{proof}
	The claims follow easily from the properties of $\chi_n$ and $w$, cf.~\cite[Prop.~V.2.6]{EE} and consider that $w_n = n - (n-w)_+$.
\end{proof}

\begin{lemma}[Construction of extension $G_{\la}$]
	\label{lem:bdd.ext}
	Let the assumptions of Theorem~{\rm\ref{thm:comp}} be satisfied. Let moreover  $T_w$, $\Vv_{w,0}$ and $\mathbf t_w$ be as in~\eqref{T.descr},~\eqref{cV.0.def} and~\eqref{t.def} and let $\widehat T_w \in \Bb(\Vv_{w,0},\Vv_{w,0}^*)$ be the corresponding distributional operator, cf.~\eqref{T.hat}. Then 
	\begin{equation}
		B_1 f:= b_1 \cdot \nabla f + c_1 f, \qquad f \in \Vv_{w,0},
	\end{equation}
	is a bounded operator from $\Vv_{w,0}$ to $L^2(\Omega)$ and  
	\begin{equation}\label{B2.def}
		(B_2 f, \varphi)_{\Vv_{w,0}^*\times \Vv_{w,0}} : = - \langle f b_2, \nabla (w \varphi) \rangle + \langle c_2 f , \varphi \rangle_{w}, \quad f \in L^2(\Omega), \quad \varphi \in \Vv_{w,0},
	\end{equation}
	is a bounded operator from $L^2(\Omega)$ to $\Vv_{w,0}^*$. Moreover, with $\iota_w$ as in~\eqref{eq:T.hat.res}, we have
	\begin{equation}\label{ext}
		G_\la:=  B_1 (\widehat T_{w} - \la \iota_w^*\iota_w)^{-1}  B_2 \in \cB (L^2(\Omega)), \qquad \lambda \in \rho(T_{w}).
	\end{equation}
\end{lemma}

\begin{proof}
	Let $\la \in \rho(T_w)$ be arbitrary. It follows from Proposition~\ref{prop:g.coer} that $\widehat T_w$ satisfies the assumptions of Lemma~\ref{lem:res.id} and we thus have 
	\begin{equation}
		(\widehat T_w - \la \iota_w^*\iota_w)^{-1} \in \Bb(\Vv_{w,0}^*,\Vv_{w,0}).
	\end{equation}
	We show below that $B_1$ and $B_2$ are well-defined and 
	\begin{equation}
		B_1 \in \Bb(\Vv_{w,0},L^2(\Omega)), \qquad B_2 \in \Bb(L^2(\Omega), \Vv_{w,0}^*).
	\end{equation} 
	Altogether, this implies that the composition $G_\la$ is everywhere defined and bounded on $L^2(\Omega)$.

	To verify the claims on $B_1$, let $f \in \Vv_{w,0} \subset \Vv_{w}$, see~\eqref{def.big.cV}. Using~\eqref{comp.asm}, it then follows that
	\begin{equation}
	\begin{aligned}
		\| B_1 f \|^2  & \lesssim \int_\Omega | \langle w^{-\frac12} P_1^{-\frac12} \overline b_1  , w^\frac12 P_1^\frac12 \nabla f \rangle_{\C^d}|^2 \, \dd x + \int_\Omega \frac{|c_1|^2}{w (|V|+1)} w (|V|+1) |f|^2 \, \dd x\\
		& \le \|w^{-\frac12} P_1^{-\frac12} \overline b_1\|^2_{L^\infty} \|P_1^\frac12 \nabla f\|_{w}^2 \\
		& \qquad \qquad \qquad  + \| c_1  w^{-\frac12} (|V|+1)^{-\frac12} \|_{L^\infty}^2 \|(|V|+1)^\frac12 f \|_{w}^2  \lesssim \|f\|^2_{\cV_{w}}.
	\end{aligned}
	\end{equation}
	For $B_2$, fix $f \in L^2(\Omega)$ and let $\varphi \in \Vv_{w,0}$ be arbitrary. Then, employing~\eqref{comp.asm} and~\eqref{asm:w.sep}, we conclude
	\begin{equation}\label{B2.bdd}
	\begin{aligned}
		& |(B_2 f, \varphi)_{\Vv_{w,0}^* \times \Vv_{w,0}}| \\
		& \qquad  \qquad \lesssim | \langle f P_1^{-\frac12}b_2 , \varphi P_1^\frac12 \nabla w \rangle | + |\langle f P_1^{-\frac12} b_2, P_1^\frac12 \nabla \varphi \rangle_w| + | \langle c_2 f, \varphi \rangle_{w} |\\
		& \qquad \qquad  \ls \|f\| \|w^{\frac12} P_1^{-\frac12} b_2\|_{L^\infty} \Big(  \|(|V| + 1)^\frac12 \varphi\|_w + \|P_1^\frac12 \nabla \varphi\|_w \Big) \\
		& \qquad \qquad \quad \qquad \qquad + \|f\| \|w^\frac12 (|V|+1)^{-\frac12} c_2\|_{L^\infty} \|(|V|+1)^{\frac12} \varphi \|_w  \lesssim \|f\|\|\varphi \|_{\cV_{w}}. \\
	\end{aligned}
	\end{equation}
	This gives $B_2f \in \Vv_{w,0}^*$ with $\|B_2 f\|_{\Vv_{w,0}^*} \ls \|f\|$, thus the proof is complete.
\end{proof}

\begin{proof}[Proof of Theorem~{\rm\ref{thm:comp}}]
	First note that clearly 
	\begin{equation}
		(T-\la)^{-1} : L^2(\Omega) \to \Vv_{1,0} , \qquad \la \in \rho(T).
	\end{equation}
	Moreover, the assumptions~\eqref{eq:ext.reg} and~\eqref{comp.asm} together with $\Vv_{1,0} \subset \Vv_1$, see \eqref{def.big.cV}, guarantee that one has the mapping properties
	\begin{equation}
		\begin{aligned}
			\widetilde B_1 & := b_1 \cdot \nabla + c_1 && \hspace{-3mm} : \, \Vv_{1,0} && \hspace{-3mm} \to \,  \Loneloc(\Omega),  \\
			\widetilde B_2 & := \nabla \cdot b_2 + c_2 && \hspace{-3mm} : \, \Coo(\Omega) && \hspace{-3mm} \to \, L^2(\Omega);
		\end{aligned}
	\end{equation}
	indeed the claim for $\wt B_2$ is immediate and for $\wt B_1$, we notice that for any $f \in \Vv_1$
	\begin{equation}
		b_1 \cdot \nabla f = w^\frac12 \langle P_1^\frac12 \nabla f , w^{-\frac12}P_1^{-\frac12 } \overline b_1\rangle_{\C^d}  \in \Loneloc(\Omega)
	\end{equation}
	since $w \in L^\infty_{\rm loc}(\Omega)$, which in particular implies that $\Ran(F_\la) \subset \Loneloc(\Omega)$.
	
	 With $G_\la$ defined as in~\eqref{ext}, we will conclude the proof by showing that 
	\begin{equation}\label{ext.G.la}
		F_{\la_0} = \widetilde B_1 (T-\la_0)^{-1} \widetilde B_2  \subset G_{\la_0}
	\end{equation}
	for some $ \la_0 \in \rho(T) \cap \rho(T_w)$, which in particular implies that $\Ran (F_{\la_0})$ lies in $L^2(\Omega)$.
	To prove~\eqref{ext.G.la}, we first define an auxiliary operator for the cut-off weight $w_1$, see Lemma~\ref{lem:w.cut-off}. 
	Then by Proposition~\ref{prop:g.coer}, one can choose $\la_0<0$ with $|\la_0|$ sufficiently large such that 
	\begin{equation}
		\la_0 \in \rho(T) \cap \rho(T_w) \cap \rho(T_{w_1}),
	\end{equation}
	cf.~Theorems~\ref{thm:lm.0} and \ref{thm:lm.1}. For a fixed $f \in \Coo(\Omega)$ and arbitrary $\varphi \in \Coo(\Omega)$ we then conclude that, see \eqref{B2.def},
	\begin{equation}
		(B_2f, w^{-1}\varphi)_{\Vv_{w,0}^* \times \Vv_{w,0}} = - \langle b_2 f, \nabla \varphi \rangle + \langle c_2f,\varphi\rangle = \langle \widetilde B_2 f, \varphi \rangle = \langle \widetilde B_2 f, w^{-1}\varphi \rangle_w.
	\end{equation}
	Since $w^{-1} \Coo(\Omega)$ is dense in $\Vv_{w,0}$, see Lemma~\ref{lem:W.1.infty}, and since $f$ has compact support, we have
	\begin{equation}
		B_2 f = \widetilde B_2f \in L^2(\Omega) \cap \Ltw \subset L^2_{w_1}(\Omega);
	\end{equation}
	recall that the embedding $\iota_w^*$ of $\Ltw$ into $\Vv_{w,0}^*$, see~\eqref{eq:T.hat.res}, is realised by means of $f \equiv \langle f, \cdot \rangle_w$, cf.~\eqref{Gelfand}. Since  $T_{w_1}$ extends both $T$ and $T_w$ by Lemma~\ref{lem:ext}, we conclude
	\begin{equation}
		(\widehat T_w - \iota_w^*\iota_w \la_0)^{-1} B_2f = ( T_w - \la_0)^{-1} \widetilde B_2f =  ( T_{w_1} - \la_0)^{-1} \widetilde B_2f  =( T- \la_0)^{-1} \widetilde B_2f .
	\end{equation}
	This finally implies that $G_{\la_0}$ extends $F_{\la_0}$ since for $g \in \Vv_{w,0} \cap \Vv_{1,0}$ we clearly have $B_1 g = \widetilde B_1 g $.
\end{proof}

\subsection{Schatten classes}

The next lemma states that a  holomorphic family of forms which are coercive in the generalised sense generates an analytic family of unbounded operators in the sense of Kato~\cite[Sec.~VII.1.2]{Kato-1966}. This is a generalisation of the analogous result for holomorphic families of type (B), see~\cite[Sec.~VII.4.2]{Kato-1966}, where closed sectoriality (or equivalently coercivity) of the forms is imposed.

\begin{lemma}[Holomorphic families of generalised coercive forms]\label{lem:hol.form}
Let $\Theta \subset \C$ be open and non-empty. Let $\Vv$ be a Hilbert space and let $\cD \subset \Vv$ be a dense subspace of $\Vv$. Let $\mathbf a_\alpha$, $\alpha \in \Theta$, be a family of locally uniformly bounded sesquilinear forms on $\Vv$ such that for all $f \in \cD$ the function 
\begin{equation}\label{eq:hol.form}
\Theta \ni  \alpha \mapsto \mathbf a_\alpha (f, f) \in \C
\end{equation}
is holomorphic on $\Theta$. Let moreover $\cV$ be continuously embedded and dense in another Hilbert space $\cH$ and assume that Theorems {\rm\ref{thm:lm.0}} and {\rm\ref{thm:lm.1}} apply to $\mathbf a_\alpha$ for all $\alpha \in \Theta$ with multipliers $\Phi_1$, $\Phi_2$ and the constant $m>0$ independent of $\alpha$.
 Then the family $\{ A_\alpha \, : \, \alpha \in \Theta\}$ of associated unbounded (and boundedly invertible) operators in $\Hh$ defined as in~\eqref{lm.op} is analytic in the sense of \cite[Sec.~VII.4.2]{Kato-1966}.
\end{lemma}
\begin{proof}
	Notice first that the operators $\widehat A_\alpha \in \cB (\cV, \cV^*)$ corresponding to $\mathbf a_\alpha$ as in~\eqref{T.hat} form a holomorphic family of bounded operators in the sense of~\cite[Sec.\ VII.1.1]{Kato-1966}, see~\cite[Thm.~III.3.12]{Kato-1966} and the paragraph below, together with a standard polarisation argument. The inverses of the associated operators
	\begin{equation}
		A_\alpha^{-1} = I_{\cV\to \Hh} \widehat A_\alpha^{-1} (I_{\cV\to \Hh})^* \in \cB (\cH),
	\end{equation}
	cf.~\eqref{res.A.hat}, are then a holomorphic family of bounded operators on $\Hh$. The claim then follows from \cite[Thm.~VII.1.3]{Kato-1966}.
\end{proof}

\begin{lemma}[Boundedness of $w^{\frac {\alpha}2}$ and $w^{-\frac\alpha2}$]
	\label{lem:w.alph}
	Let Assumption {\rm\ref{asm:main.LM}}~\ref{item:LM.reg},~\ref{item:asm.V.accr} and~\ref{asm:LM.w} hold and let $\alpha \in \C$. Then $w^{\re \alpha}$ is an admissible weight. More precisely, $w^{\re \alpha} \in W^{1,\infty}_{\rm loc} (\Omega)$ is uniformly positive on compact subsets of $\Omega$ and 
 	\begin{equation}
		\label{re.alph.asm}
		\frac{| P_1^{\frac12}\nabla w^{ \re \alpha}|}{w^{\re \alpha}} \le  |\re \alpha| (\kappa_w V_1^\frac12 + \tau_w |V_2|^\frac12 + C_w).
	\end{equation}
	Moreover, the spaces $\cV_{1,0}$ and $\cV_{w^{\re \alpha},0}$ defined as in \eqref{cV.0.def} with the respective weights are isomorphic via
	\begin{equation}\label{w.alph.bdd}
		w^{\frac\alpha2} \in \cB(\cV_{w^{\re \alpha},0}, \cV_{1,0}), \qquad w^{-\frac\alpha2} \in \cB(\cV_{1,0}, \cV_{w^{\re \alpha},0}).
	\end{equation}
\end{lemma}

\begin{proof}
	The claimed regularity of $w^{\re\alpha}$ follows from $w \in W^{1,\infty}_{\rm loc} (\Omega)$	and  
	\begin{equation}
		\nabla (w^{\re \alpha}) =  \re \alpha \, w^{\re \alpha -1}\, \nabla w 
	\end{equation}
	which, using \eqref{asm:w.sep}, moreover gives the bound in \eqref{re.alph.asm}. Clearly, $w^{\re \alpha}$ is uniformly positive on compact subsets of $\Omega$ since $w$ is.
	
	To prove the claims in \eqref{w.alph.bdd}, it is sufficient to show the required inequalities on the dense subspace $\Coo(\Omega)$. For the first claim, with $f \in \Coo(\Omega)$  we estimate
	\begin{equation} \label{w.alph.1}
		\begin{aligned}
			\| w^{\frac\alpha2} f \|_{\cV_1}^2 & = \| P_1^\frac12 \nabla (w^{\frac\alpha2} f)\|^2 + \| |V|^\frac12 w^{\frac\alpha2} f\|^2 + \|w^{\frac\alpha2} f\|^2 \\
			& \lesssim \|f P_1^\frac12 \nabla (w^{\frac\alpha2}) \|^2 + \| f \|^2_{\cV_{w^{\re \alpha}}}.
		\end{aligned}
	\end{equation}
	Notice  that $w^{\frac\alpha2} f \in W^{1,\infty}_{\operatorname{comp}} (\Omega) \subset \Vv_{1,0}$ by Lemma \ref{lem:W.1.infty} and 
	\begin{equation}\label{nabla.w.alph}
		\nabla(w^{\frac\alpha2}) = {\frac\alpha2} w^{{\frac\alpha2}-1} \nabla w.
	\end{equation}
	Using \eqref{nabla.w.alph}, \eqref{asm:w.sep}, Cauchy--Schwarz' and Young's inequalities, it follows that
	\begin{equation}\label{w.alph.2}
		\begin{aligned}
			\| f P_1^\frac12 \nabla w^{\frac\alpha2}\|^2 & \le \langle |f| |P_1^\frac12 \nabla (w^{\frac\alpha2})|, |f P_1^\frac12 \nabla (w^{\frac\alpha2})|\rangle \\
			& \le \frac{|\alpha|}{2} \langle  w^{\frac{\re \alpha}2} (\kappa_w V_1^\frac12 + \tau_w |V_2|^\frac12 + C_w)  |f|, |f P_1^\frac12 \nabla (w^{\frac\alpha2}) | \rangle \\
			& \le \frac{|\alpha|^2}{8\delta} \left( \kappa_w^2 \| V_1^\frac12 f\|^2_{w^{\re \alpha}} + \tau_w^2 \|  |V_2|^\frac12 f\|^2_{w^{\re \alpha}} +  C_w^2 \| f \|^2_{w^{\re \alpha}} \right) \\
			& \quad \qquad\qquad \qquad \qquad \qquad \qquad + \frac{ 3 \delta}{2} \| P_1^\frac12 \nabla (w^\frac\alpha 2) f\|^2
		\end{aligned}
	\end{equation}
	for any $\delta> 0$. Hence, choosing $\delta$ small enough such that $3\delta <2$ and combining \eqref{w.alph.1} with \eqref{w.alph.2}, we arrive at the first claim in~\eqref{w.alph.bdd}. The proof of the second one is analogous.
\end{proof}

\begin{lemma}[Unitary equivalence of $T_{w^{\re \alpha}}$ and  $S_\alpha$]\label{lem:Sw.form}
Let the hypotheses of Theorem~{\rm\ref{thm:t}} hold (where \eqref{eq:rem.rel.bdd.eps} and \eqref{ellipse} are replaced by~\eqref{eq:rem.rel.bdd}, \eqref{ellipse.gen.beta} and~\eqref{ellipse.gen}). Let $\cV_{1,0}$ and  $\mathbf t_{w^{\re \alpha}}$ be as in \eqref{cV.0.def} and \eqref{t.def} with the respective weights. Define the bounded sesquilinear forms
\begin{equation}
	\mathbf s_\alpha (f,g) := \mathbf t_{w^{\re \alpha}} (w^{-\frac\alpha2} f, w^{-\frac\alpha2} g), \qquad f,g \in \cV_{1,0}, \qquad \alpha \in\C,
\end{equation}
and let $S_\alpha$ be the associated operator in $L^2(\Omega)$ defined as in~\eqref{lm.op}.
Then there exists $\delta>0$ such that the family $\{S_\alpha \, : \, \alpha \in M_{\delta}\}$, where 
\begin{equation}
	M_{\delta} := \{ z \in \C \, : \, |\re z| < 1 + \delta, \, \,  |\Im z|<\delta \}, 
\end{equation} 
is analytic in the sense of \cite[Sec.~VII.1.2]{Kato-1966}. Moreover, every $S_\alpha$ is unitarily equivalent to $T_{w^{\re \alpha}}$ in $L^2_{w^{\re \alpha}}(\Omega)$, defined as in~\eqref{T.descr}, by means of
\begin{equation}\label{un.equiv}
	S_\alpha = w^{\frac\alpha2} T_{w^{\re \alpha}} w^{-\frac\alpha2}, \qquad \alpha \in M_\delta.
\end{equation}
\end{lemma}
\begin{proof}
	By Lemma~\ref{lem:w.alph}, $w^{\re \alpha}$ is an admissible weight and in particular satisfies condition \eqref{asm:w.sep} with $\kappa_w, \tau_w$ and $C_w$ replaced by
	\begin{equation}\label{w.alpha.const}
			\kappa_\alpha := |\re\alpha| \kappa_w, \qquad \tau_\alpha := |\re\alpha| \tau_w, \qquad C_\alpha := |\re\alpha| C_w,
	\end{equation}
	see \eqref{re.alph.asm}. Hence, from Lemmas \ref{lem:t.bdd} and \ref{lem:w.alph}, it follows that the forms $\mathbf s_\alpha$ are well-defined and bounded on $\cV_{1,0}$. To prove the analyticity of the associated operators, we apply Lemma~\ref{lem:hol.form} after a suitable $\alpha$-independent shift. The needed holomorphy of the forms $\mathbf s_\alpha$ on $\Theta:=M_\delta$ (in fact even on $\C$) follows readily since for every $f  \in\cD:=\CcOm$ one computes
	 \begin{equation}
	 \label{s.hol.2}
	 \begin{aligned}
	 		\mathbf s_\alpha (f,f) & = \langle P \nabla (w^{-\frac\alpha2}f), \nabla (w^{-\frac\alpha2}fw^{\re \alpha}) \rangle + \langle V w^{-\frac\alpha2}f, w^{-\frac\alpha2}f\rangle_{w^{\re \alpha}} \\
	 		& =  \langle P \nabla (w^{-\frac\alpha2} f), \nabla (w^{\frac{\ov\alpha}2} f  )\rangle + \langle V f, f\rangle \\
	 		& = \frac{\alpha}2 \big(\langle P \nabla f,  w^{-1} f \nabla w \rangle - \langle  f P \nabla w , w^{-1} \nabla f \rangle\big)  \\
	 		& \qquad \qquad - \frac{\alpha^2}4 \langle f P \nabla w,w^{-2} f\nabla w \rangle  + \langle P \nabla f, \nabla f \rangle + \langle V f, f\rangle.
	 \end{aligned}
	 \end{equation}
	Moreover, using \eqref{eq:P.bounds} and \eqref{asm:w.sep}, it is straightforward to see that all quadratic forms in the last line above are bounded w.r.t.~$\|\cdot\|_{\cV_{1}}$, and thus the $\mathbf s_\alpha$ are locally uniformly bounded in $\alpha$.
	To apply Lemma~\ref{lem:hol.form} (with $\Vv=\Vv_{1,0}$ and $\Hh=L^2(\Omega)$), it suffices to find a shift $\gamma \in \C$, uniform in $\alpha$, such that the forms $\mathbf a_\alpha := \mathbf s_\alpha + \gamma \|\cdot \|^2$ are generalised coercive with a suitable multiplier. The size of $\delta>0$ is determined such that this can be achieved via the estimates in Proposition~\ref{prop:g.coer}.
	
	Indeed, we choose $\delta >0$ so small that, with $\beta$ as in~\eqref{ellipse.gen.beta}, condition~\eqref{ellipse.gen} holds for all $\alpha \in M_{2\delta}$ with $\kappa_\alpha$ and $\tau_\alpha$ as in~\eqref{w.alpha.const} instead of $\kappa_w$ and $\tau_w$; this is possible due to the continuity of~\eqref{ellipse.gen} in $\kappa_w$ and $\tau_w$. Proposition~\ref{prop:g.coer} is thus applicable to the weighted form $\mathbf t_{w^{\re \alpha}}$ and it follows that there exist $m>0$ and $\gamma\ge0$ (note that one can choose $m_1=m_2$ and $\gamma_1=\gamma_2$, see the proof of Proposition~\ref{prop:g.coer}) such that for all $\alpha \in M_\delta$ and $g \in \cV_{w^{\re \alpha},0}$ we have
	\begin{equation}\label{w.alph.coerc.1}
		\begin{aligned}
			\re t_{w^{\re \alpha}} (g,g) + \im \mathbf t_{w^{\re \alpha}} (\beta \Phi g, g) + \gamma \|g\|_{w^{\re \alpha}}^2 & \ge m \|g \|_{\cV_{w^{\re \alpha}}}^2, \\
			\re t_{w^{\re \alpha}} (g,g) + \im \mathbf t_{w^{\re \alpha}} ( g, \beta \Phi g) + \gamma\|g\|_{w^{\re \alpha}}^2 & \ge m \|g\|_{\cV_{w^{\re \alpha}}}^2.
		\end{aligned}
	\end{equation}
	Here the constants $m$ and $\gamma$ can be chosen independently of $\alpha$ due to continuity and since $\overline{M_{\delta}} \subset M_{2\delta}$ is compact. Considering Lemma~\ref{lem:w.alph}, this leads to
	\begin{equation}\label{w.alph.coerc.2}
		\begin{aligned}
			|\mathbf s_{\alpha} (f,f) + \gamma \|f\|^2 |  + | \mathbf s_\alpha ( \beta\Phi f, f) + \gamma  \langle \beta\Phi f, f \rangle |& \ge m \|w^\frac{\alpha}{2} \|_{\Vv_{w^{\re\alpha},0}\to \Vv_{1,0}}^{-2} \|f \|_{\cV_1}^2, \\
			|\mathbf s_\alpha (f,f) + \gamma \|f\|^2 |  + | \mathbf s_\alpha (f,\beta\Phi f) +  \gamma \langle f, \beta\Phi f \rangle |& \ge m \|w^\frac{\alpha}{2} \|_{\Vv_{w^{\re\alpha},0}\to \Vv_{1,0}}^{-2} \|f\|_{\cV_1}^2,
		\end{aligned}
	\end{equation}
	for all $\alpha \in M_\delta$ and $f =w^\frac{\alpha}{2}g\in \cV_{1,0}$. Since $|\Phi|\le 1$ is bounded both on $\cV_{1,0}$ and $L^2(\Omega)$, see Lemma~\ref{lem:Phi.bdd}, it follows from Lemma~\ref{lem:hol.form} that $S_\alpha + \gamma$, and thus also $S_\alpha$ form a holomorphic family for $\alpha \in M_\delta$.
	
	It remains to show the identity in \eqref{un.equiv}, which yields the claimed unitary equivalence since
	\begin{equation}
		w^{\frac{\alpha}{2}} : L^2_{w^{ \re \alpha}} (\Omega) \to L^2(\Omega)
	\end{equation}
	is clearly unitary. By definition, $f \in \Dom (S_\alpha)$ with $S_\alpha f = \eta$ if and only if one has $f \in \Vv_{1,0}$ and
	\begin{equation}
		\mathbf s_\alpha (f, \varphi) = \langle \eta, \varphi\rangle
	\end{equation}
	for all $\varphi \in \Vv_{1,0}$, cf.~\eqref{lm.op}. By Lemma~\ref{lem:w.alph}, the multiplication $w^{-\frac{\alpha}{2}} : \Vv_{1,0}\to \Vv_{w^{\re\alpha},0}$ is bijective, thus this is further equivalent to $g:= w^{-\frac\alpha2} f \in \Vv_{w^{\re\alpha},0}$ such that
	\begin{equation}
		\mathbf t_{w^{\re \alpha}} (g, \psi)  = \langle  w^{-\frac\alpha2} \eta, \psi \rangle_{w^{\re \alpha}}
	\end{equation}
	for all $\psi := w^{-\frac{\alpha}{2}} \varphi \in \Vv_{w^{\re \alpha},0}$. This, however, is equivalent to
	\begin{equation}
		g = w^{-\frac\alpha2} f\in \Dom (T_{w^{\Re \alpha}}), \qquad T_{w^{\Re \alpha}} g = T_{w^{\Re \alpha}} w^{-\frac\alpha2} f = w^{-\frac{\alpha}{2}} \eta = w^{-\frac{\alpha}{2}} S_\alpha f.
	\end{equation}
	(Recall that in the proof of Theorem~\ref{thm:t}, it is justified that the operator associated with $\mathbf t_{w^{\re \alpha}}$ in $L^2_{w^{\re \alpha}}(\Omega)$ coincides with $T_{w^{\re \alpha}}$ defined as in~\eqref{T.descr}.) This proves identity~\eqref{un.equiv}.
\end{proof}

\begin{proof}[Proof of Theorem~{\rm\ref{thm:inv.Sp}}]
	Let $T_{w}$ in $\Ltw$ be as in~\eqref{T.descr}, $S_1 = w^\frac12  T_w w^{-\frac12}$ in $L^2(\Omega)$ as in Lemma~{\rm\ref{lem:Sw.form}} with $\alpha=1$. By unitary equivalence of $T_w$ and $S_1$, it suffices to prove the claims for $S_1$. The resolvent of $S_1$ is given by
	\begin{equation}\label{res.emb}
		(S_1 - \lambda)^{-1} = \iota_1 (\widehat S_1 -\lambda \iota_1^*\iota_1 )^{-1} \iota_1^*, \qquad \iota_1= I_{\cV_{1,0} \to L^2(\Omega)},
	\end{equation}
	for $\la \in \rho(S_1)$, see~\eqref{T.hat},~\eqref{Gelfand},~\eqref{res.A.hat} and Lemma~\ref{lem:res.id}. 
	\begin{enumerate}[\upshape (i), wide]
	\item The resolvent of $S_1$ is compact since $(\widehat S_1 - \lambda \iota_1^*\iota_1)^{-1}$ is bounded from $\cV_{1,0}^*$ to $\cV_{1,0}$ and $\iota_1$ is compact by assumption (and so is its adjoint).
		\item 	The assumption implies that also $\iota_1^*\in \cS_{2p} (L^2(\Omega), \cV_{1,0}^*)$. Hence by \eqref{res.emb} and H\"older's inequality for Schatten classes, see~e.g.~\cite[Thm.\ 3.23]{Weidmann-2003}, we arrive at
		\begin{equation*}
			\|(S_1 - \lambda)^{-1} \|_{\cS_p} \ls \| \iota_1  \|_{\cS_{2p}}^2 \|(\widehat S_1 - \lambda \iota_1^*\iota_1)^{-1}\| < \infty.
		\end{equation*}
		\item Using the assumption and standard inequalities for the singular values, see \cite[Cor.~II.2.2]{Gohberg-1969} and~\cite[Prop.~VI.1.3]{Gohberg-1990}, we obtain
	\begin{equation}
	\begin{aligned}
	s_{m+n-1} ( \iota_1  (\widehat S_1 - \lambda \iota_1^*\iota_1 )^{-1} \iota_1^* ) &  \leq \|(\widehat S_1 - \lambda \iota_1^*\iota_1 )^{-1}\| s_m( \iota_1) s_n( \iota_1) \ls \frac{1}{m^ \frac1{2p}} \frac{1}{n^\frac1{2p}}
	\end{aligned}
	\end{equation}
	for $m,n \in \N$. Applying the above with $m=k \in\N$ and $n=k$ or $n= k+1$, we conclude that 
		\begin{equation}
			s_l ((S_1-\la)^{-1})\ls l^{-\frac1p}, \qquad l \in \N,
		\end{equation}
		i.e.~$(S_1 - \lambda)^{-1}  \in S_{p,\infty}(L^2(\Omega))$. \qedhere
	\end{enumerate}
\end{proof}

\begin{proof}[Proof of Corollary~{\rm\ref{cor:Sp}}]
	Let $S:= -\Delta + |V|$ be the standard self-adjoint Dirichlet realisation in $L^2(\Omega)$ defined by the Friedrichs extension, and let $\iota_1 := I_{\cV_{1,0} \to L^2(\Omega)}$, see~\eqref{cV.0.def}.
\begin{enumerate}[\upshape (i), wide]
\item Applying Theorem~\ref{thm:Sp}, we arrive at
\begin{equation}\label{eq:Sp}
	(S+1)^{-1} \in \cS_p(L^2(\Omega)), \qquad (S+1)^{-\frac 12} \in \cS_{2p}(L^2(\Omega)).
\end{equation}
Moreover, by the second representation theorem~\cite[Thm.~VI.2.23]{Kato-1966} and the construction of $\cV_{1,0}$, see~\eqref{cV.0.def}, we have
\begin{equation}
	\|(S+1)^\frac12 f \|^2 = \|f\|_{\cV_1}^2, \qquad f \in \Dom ((S+1)^\frac12) = \cV_{1,0}.
\end{equation}
Since~$(S+1)^\frac12$ is isometric from $\Vv_{1,0}$ to $\LOm$ and we have \eqref{eq:Sp}, we obtain that
\begin{equation*}
	(S+1)^{-\frac 12}(S+1)^\frac12 = \iota_1 \in S_{2p} (\cV_{1,0}, L^2(\Omega)). 
\end{equation*}	
Hence \eqref{Tw.Sp.cor} follows from Theorem~\ref{thm:inv.Sp}.
\item Assume next that  $|V(x)|  + 1 \gs \langle x\rangle^\gamma$ for a.e.~$x \in \Omega$ and with some $\gamma>0$. 
Let $\wt S = -\Delta + |x|^\gamma$ be the self-adjoint Dirichlet realisation in $\LOm$. Since $(S+1)^\frac12 $ is isometric from $\Vv_{1,0}$ to $\LOm$ and $(\wt S+1)^\frac 12 (S+1)^{-\frac12}$ is bounded on $\LOm$, we have, see \cite[Prop.VI.1.3]{Gohberg-1990},
\begin{equation}\label{sk.ineq.gam}
	\begin{aligned}
		s_k\left( \iota_1 \right) & = 
		s_k \left(  (\wt S+1)^{-\frac 12} (\wt S+1)^\frac 12 (S+1)^{-\frac 12}(S+1)^\frac12 \right)
		\\
		& 
		\leq 
		\|(\wt S+1)^\frac 12 (S+1)^{-\frac 12}\|
		s_k \left( (\wt S+1)^{-\frac 12} \right)
		\\
		& = \|(\wt S+1)^\frac 12 (S+1)^{-\frac 12}\|
		s_k \left( (\wt S+1)^{-1} \right)^\frac 12, \qquad k \in \N.
	\end{aligned}	
\end{equation}	
Thus it remains to investigate the eigenvalues of $\wt S$.

First by Young's inequality and using polar coordinates, we get
\begin{equation}\label{int.schat.gam}
	\begin{aligned}
		\int_{\R^d} \int_{\Omega} \left( \langle \xi\rangle^2 + \langle x \rangle^\gamma \right)^{-p} \dd x \, \dd \xi & \ls \int_{\R^d} \int_{\Omega} \left( \left( \langle \xi\rangle  \langle x \rangle \right)^{\frac{2\gamma}{2+\gamma}} \right)^{-p} \dd x \, \dd \xi \\
		& \ls \left(\int_0^\infty  (r+1)^{d-1-\frac{2\gamma p}{2+\gamma}}\dd r\right)^{2}.
	\end{aligned}
\end{equation}
Since~\eqref{int.schat.gam} converges for $p>p_{\gamma,d}$, we have $(\wt S+1)^{-1} \in \cS_p(\LOm)$ by Theorem~\ref{thm:Sp}. Hence $\iota_1 \in \cS_{2p}(\LOm)$ by \eqref{sk.ineq.gam} and \eqref{Tw.Sp.cor.gam} follows by Theorem~\ref{thm:inv.Sp}.

Finally, let $\Omega = \Rd$ and \eqref{V.lb.gamma} hold with $\gamma>1$. 
 By Theorem~\ref{thm:Sp.inf} and straightforward manipulations using polar coordinates, we obtain that 
\begin{equation}
\begin{aligned}
N(\la) = \{ \la_k \in \sigma(\wt S) \, : \, \la_k < \la \} 
& \approx
\int_{\{|x|^\gamma \leq \la\}} (\la - |x|^\gamma)^\frac d2 \, \dd x 
\\
& \approx
\la^{p_{\gamma,d}} \int_0^1 (1-t^\gamma)^\frac d2 t^{d-1} \, \dd t, \qquad \la>1;
\end{aligned}
\end{equation}
for $d=1$, the claim is valid as well, cf.~\cite[Rem.~after Thm.~XIII.81]{Reed4} or more precise asymptotics in \cite[Chap.~VII]{Titchmarsh-1962-book1}. It follows that there exists $k_0 \in \N$ such that, for every fixed $k \ge k_0$ and $\eps>0$ small enough, we have
\begin{equation}
	k = N(\la_k+\eps) \ls (\la_k+\eps)^{p_{\gamma,d}} \to  \la_k^{p_{\gamma,d}}, \qquad \eps \to 0.
\end{equation}
Hence
\begin{equation}
s_k\left((\wt S+1)^{-1}\right) = (\la_k+1)^{-1} \ls k^{-\frac{1}{p_{\gamma,d}}}, \qquad k \in \N,	
\end{equation}
and so \eqref{Tw.Sp.cor.gam.inf} is justified by \eqref{sk.ineq.gam} and Theorem~\ref{thm:inv.Sp}.
\qedhere
\end{enumerate}
\end{proof}

\subsection{Invariance of discrete spectra and eigenfunctions}

\begin{proof}[Proof of Theorem~{\rm\ref{thm:spd.inv}}]
	Our method is based on the analytic family  $\{S_\alpha\, :\, \alpha \in M_\delta\}$ from Lemma~\ref{lem:Sw.form}, which satisfies
		\begin{equation}\label{eq:sp.uni.equ}
			S_{\alpha_0+ \ii \mu} = w^{\frac{\alpha_0+ \ii \mu}2} T_{w^{\alpha_0}} w^{-\frac{\alpha_0+ \ii \mu}2} = w^{\frac{\ii \mu}2} S_{\alpha_0} w^{-\frac{\ii \mu}2}
		\end{equation}
		for every $\alpha_0 \in [0,1]$ and $\mu \in (-\delta, \delta)$. This in particular implies that $S_{\alpha_0+ \ii \mu}$ and $S_{\alpha_0}$ are unitarily equivalent and thus have identical spectral properties.
	
	\begin{enumerate}[\upshape (a), wide]
		
		\item Invariance of discrete eigenvalues and their multiplicities:
		
	We first prove a local result. For every fixed $r>0$ and $z \in \C\setminus \Sigma$ with $\overline{B_r(z)} \cap \Sigma = \emptyset$ and every $\alpha_0 \in [0,1]$, we show that there exists $\nu = \nu(\alpha_0,r,z)>0$ such that 
	\begin{equation}\label{eq:ev.local}
		m_{\rm a} (\la, S_\alpha) = m_{\rm a} (\la, S_{\alpha_0}), \qquad \la \in \sigma (S_{\alpha_0}) \cap B_r(z) = \sigma (S_{\alpha}) \cap B_r(z),
	\end{equation}
	for all $\alpha \in B_\nu(\alpha_0)$.
	To justify this, we suitably adapt a complex scaling argument in~\cite[Sec.~XIII.10]{Reed4}, cf.~\cite[Sec.~VII.1.3]{Kato-1966}. 
	
	Since $S_{\alpha_0}$ has discrete spectrum in $\C\setminus\Sigma$ and $\{S_\alpha\, :\, \alpha \in M_\delta\}$ is analytic, there exist $\eps\ge0$ and $\nu>0$ such that 
	\begin{equation}
		\Gamma:= \partial B_{r+\eps} (z) \subset \rho(S_{\alpha}), \qquad \alpha \in B_\nu(\alpha_0),
	\end{equation}
	and the rank of the corresponding spectral projections is constant, i.e.
	\begin{equation}
		\label{P.alph.def}
		P_\alpha := - \frac{1}{2\pi \ii} \int_{\Gamma} (S_\alpha - \la)^{-1} \dd \la, \qquad \operatorname{rank} P_\alpha =: k \in \N,
	\end{equation}
	see~\cite[Sec.~VII.1.3]{Kato-1966}.
	For every $\alpha \in B_\nu(\alpha_0)$, the $k$-dimensional subspace $\Ran (P_\alpha)$ is left invariant by $S_\alpha$ and  the spectrum of $S_\alpha$ within the contour $\Gamma$ consists of the $k$ eigenvalues (counted with algebraic multiplicities) of the finite rank operator $A_\alpha:= S_\alpha P_\alpha$, namely
	\begin{equation}
		\sigma(S_\alpha) \cap B_{r+\eps} (z) = \big\{\la \in \C \, : \, p_\alpha(\la) := \det (A_\alpha - \la) = 0 \big\},
	\end{equation}
	where $p_\alpha$ is a polynomial of degree $k$. Since the coefficients of  $p_\alpha$ can be written as polynomials in $\operatorname{tr} (A_\alpha^l)$ with different $l \in \N$, see~\cite{Kalman-2000-73} or \cite[Thm.~XIII.108]{Reed4}, they depend holomorphically on $\alpha \in B_\nu(\alpha_0)$. For $\mu \in (-\nu, \nu)$, it follows from \eqref{eq:sp.uni.equ} that
	\begin{equation}
		A_{\alpha_0+\ii \mu} = w^{\frac{\ii \mu}2} S_{\alpha_0} w^{-\frac{\ii \mu}2} P_{\alpha_0+\ii \mu} =  w^{\frac{\ii \mu}2} S_{\alpha_0} P_{\alpha_0}w^{-\frac{\ii \mu}2}  = w^{\frac{\ii \mu}2} A_{\alpha_0}w^{-\frac{\ii \mu}2},
	\end{equation}
	where the identity in the middle is a consequence of the definition of $P_\alpha$, \eqref{eq:sp.uni.equ} and the unitarity of $w^{-\frac{\ii \mu}2}$, in more detail
	\begin{equation}
		\label{P.inv.im}
		\begin{aligned}
			w^{-\frac{\ii \mu}2} P_{\alpha_0+\ii \mu} & = - \frac{1}{2 \pi \ii} \int_\Gamma w^{-\frac{\ii \mu}2} (S_{\alpha_0+\ii \mu}-\la)^{-1} \dd \la \\
			& = - \frac{1}{2 \pi \ii} \int_\Gamma  (S_{\alpha_0}-\la)^{-1} w^{-\frac{\ii \mu}2} \dd \la = P_{\alpha_0}  w^{-\frac{\ii \mu}2}.
		\end{aligned}
	\end{equation}
	In other words, also $A_{\alpha_0+\ii \mu}$ and $A_{\alpha_0}$ are unitarily equivalent and thus
	\begin{equation}
		\operatorname{tr} (A_{\alpha_0+\ii \mu}^l) = \operatorname{tr} \left(w^{-\frac{\ii \mu}2} A_{\alpha_0+\ii \mu}^l w^{\frac{\ii \mu}2}\right) = \operatorname{tr} (A_{\alpha_0}^l), \qquad l \in\N. 
	\end{equation} 
	This implies that
	\begin{equation}
		p_{\alpha_0+\ii \mu} \equiv p_{\alpha_0}, \qquad \mu \in (-\nu,\nu),
	\end{equation}
	and thus by the identity theorem for holomorphic functions, we conclude that in fact $p_\alpha \equiv p_{\alpha_0}$ for all $\alpha \in B_\nu(\alpha_0)$.	Hence, the (discrete) spectra of $S_\alpha$ and $S_{\alpha_0}$ coincide inside the contour $\Gamma$, as well as the corresponding algebraic multiplicities. This proves the local claim~\eqref{eq:ev.local}, which in turn by a compactness argument implies that for every compact $K \subset \C\setminus \Sigma$ there exists $\eta=\eta(K)>0$ such that for all $\alpha \in B_\eta([0,1])$ we have
	\begin{equation}
		\label{eq:ev.comp}
		m_{\rm a} (\la, S_\alpha) = m_{\rm a} (\la, S_0), \qquad \la \in \sigma(S_\alpha)\cap B_\eta(K) = \sigma(S_0) \cap B_\eta(K).
	\end{equation}
	
	To prove the full claim, we introduce a distribution which contains all information about the (discrete) spectrum of $S_{\alpha_0}$ outside $\Sigma$. More precisely, let
	\begin{equation}
		\Psi_{\alpha_0} := \sum_{\la \in \sigma (S_{\alpha_0}) \setminus \Sigma} m_{\rm a} (\la, S_{\alpha_0}) \delta(x-\la) \, \in \, \Dd'(\C \setminus \Sigma), \qquad \alpha_0 \in [0,1],
	\end{equation}
	where we identify $\C \simeq \R^2$ when test functions or distributions are concerned. Since there are only finitely many eigenvalues of $S_{\alpha_0}$ lying in a compact set $K \subset \C\setminus \Sigma$, it is clear that $\Psi_{\alpha_0}$ is indeed a distribution. To conclude the proof, recalling that $S_{\alpha_0}$ is unitarily equivalent to $T_{w^{\alpha_0}}$, see~\eqref{un.equiv}, it suffices to show that $\Psi_{\alpha_0}$ is constant, i.e.~that for every $\varphi \in\Dd(\C\setminus\Sigma)$ the function
	\begin{equation}
		[0,1] \ni \alpha_0 \mapsto f_\varphi (\alpha_0):=(\Psi_{\alpha_0}, \varphi )_{\Dd'(\C\setminus\Sigma) \times \Dd(\C\setminus\Sigma) }  
	\end{equation}
	is constant. This, however, is immediate from \eqref{eq:ev.comp} with $K = \supp \varphi$ and
	\begin{equation}
		f_\varphi(\alpha_0) =  \sum_{\la \in \sigma (S_{\alpha_0}) \cap \supp \varphi} m_{\rm a} (\la, S_{\alpha_0}) \varphi(\la).
	\end{equation}

	\item Invariance of eigenfunctions:
		
		We use the notation from (a) and follow the strategy of O'Connors Lemma, see \cite{OConnor-1973-32} and \cite[p.~196]{Reed4}. Let $\la_0$ be a fixed eigenvalue in $\sigma(S_{0}) \setminus \Sigma$, with algebraic multiplicity $m \in \N$. By \eqref{eq:ev.comp}, there exist $r, \eta>0$ such that 
		\begin{equation}
				\sigma (S_{\alpha}) \cap \overline{B_{r}(\la_0)} = \{\la_0\}, \qquad m_{\rm a} (\la_0, S_\alpha) = m_{\rm a} (\la_0, S_{0}), 
		\end{equation}
		for all $\alpha \in B_{\eta} ([0,1])$. Let $P_\alpha$ be the corresponding Riesz projections as in \eqref{P.alph.def} with the contour $\Gamma = \partial B_{r} (\la_0)$, which depend holomorphically on $\alpha \in B_\eta([0,1])$, see \cite[Sec.VII.1.3]{Kato-1966}.
		
		By the density of $\CcOm$ in $L^2(\Omega)$ and the boundedness of $P_0$, there exists $\{\phi_1,\dots,\phi_m\} \subset \CcOm$ such that \{$P_0 \phi_1, \dots, P_0 \phi_m\}$ is a basis of $\Ran (P_0)$ (since the vectors $P_0 \phi_j$, $j=1,\dotsc,m$, can be chosen to approximate a selected orthonormal basis of $\Ran(P_0)$ to an arbitrary  precision).
		For $\psi \in \Ran(P_0)$, there are $\beta_j \in \C$, $j=1,\dots,m$, such that
		\begin{equation}
			\psi = \sum_{j=1}^m \beta_j P_0 \phi_j.
		\end{equation}
		Since $\phi_j \in L^2_{w^{\Re \alpha}}(\Omega)$ and  $Q_\alpha:=w^{-\frac\alpha2} P_\alpha w^{\frac\alpha2}$ is the Riesz projection  of $T_{w^{\Re \alpha}}$ corresponding to $\la_0$, see~\eqref{un.equiv}, we can define a function		
		\begin{equation}
			\psi(\alpha) := \sum_{j=1}^m \beta_j w^{-\frac \alpha 2} P_\alpha w^{\frac \alpha 2} \phi_j \in \Ran (Q_\alpha), \qquad \alpha \in B_{\eta} ([0,1]).
		\end{equation}

		We show that $\psi(\alpha)$, viewed as element of $\LolocOm \subset \cD'(\Omega)$, is independent of $\alpha$. To this end, for every $\varphi \in \cD(\Omega)$, define a function
		\begin{equation}
			g_\varphi(\alpha) := ( \psi(\alpha),\varphi )_{\cD' (\Omega)\times \cD(\Omega)}, \qquad \alpha \in B_{\eta} ([0,1]).
		\end{equation}
		Note first that, for $u \in C^\infty_0(\Omega)$, one can show by standard arguments that $\alpha \to w^{\pm \frac\alpha 2} u$ is holomorphic with values in $L^2(\Omega)$. Since the $P_\alpha$ are holomorphic with values in $\cB(L^2(\Omega))$, it is then straightforward to see that the function
		\begin{equation}
			B_{\eta} ([0,1]) \ni \alpha \to g_\varphi (\alpha) = \sum_{j=1}^m \beta_j \langle P_\alpha w^\frac\alpha 2 \phi_j, \overline{w^{-\frac\alpha 2} }\varphi \rangle_{L^2(\Omega)}
		\end{equation}
		is holomorphic. Moreover, from \eqref{P.inv.im}, for any $\alpha_0 \in [0,1]$ and $\mu \in (-\eta,\eta)$ we have $\psi(\alpha_0 + \ii \mu) = \psi(\alpha_0)$,
		hence $g_\varphi (\alpha_0 + \ii \mu) = g_\varphi (\alpha_0)$. By the identity theorem for holomorphic functions, $g_\varphi$ is constant on $B_{\eta} ([0,1])$ and thus $\psi(\alpha)$ is a constant distribution in $\alpha$. In particular,  
		\begin{equation}
			 \psi = \psi(0) = \psi(1)  \in \Ran(Q_1),
		\end{equation} 
		which implies $\Ran (P_0) \subset \Ran (Q_1)$ and hence $\Ran (P_0) = \Ran (Q_1)$ as both spaces have dimension $m$. More precisely, the (finite dimensional) root subspaces of $S_0=T$ and $T_w$ coincide, and since their actions are given by the same weak differential expression, see \eqref{T.descr}, they have the same Jordan structure \eqref{gen.ef} and the proof is complete.  \qedhere
		\end{enumerate}	
		
\end{proof}

\subsection{Graph norm separation}

\begin{lemma}[Core of $T_w$]\label{lem:core}
	Let Assumption~{\rm\ref{asm:main.gn}}~\ref{item:asm.GN.reg} and~\ref{item:asm.P1.ell} be satisfied. Then $T_w$ in~\eqref{T.descr} is well-defined and
	\begin{equation}\label{cD.def}
		\cD_w :=\Big\{ f \in \Dom(T_w) \, : \,  \operatorname{ess} \supp f \text{ {\rm is bounded in }} \R^d\Big\}
	\end{equation}
	is dense in $\big(\Dom(T_w),(\|T_w \cdot\|_w^2+\|\cdot\|^2_w)^\frac12\big)$. 
	
	Moreover, if $\Omega = \Rd$, then $C_0^\infty(\Rd)$ is dense in $\big(\cD_w,(\|T_w \cdot\|_w^2+\|\cdot\|^2_w)^\frac12\big)$.
\end{lemma}
\begin{proof}
	Notice first that Assumption~{\rm\ref{asm:main.gn}}~\ref{item:asm.GN.reg},~\ref{item:asm.P1.ell} imply Assumption~\ref{asm:main.LM}~\ref{item:LM.reg},~\ref{item:asm.P} and thus $\Vv_w$ in~\eqref{def.big.cV}, $\Vv_{w,0}$ in~\eqref{cV.0.def} and $T_w$ in~\eqref{T.descr} are well-defined (with the standard distributional action), see Lemma~\ref{lem.cV.compl} and cf.~the proof of Theorem~\ref{thm:t}.
	
	The justification of the first density claim proceeds by a standard cut-off strategy, cf.~\cite[proof of Lem.~3.6]{Krejcirik-2017-221} or~\cite[proof of~Thm.\ 8.2.1, Part 1]{Davies-1995}.
	Fix $f \in \Dom (T_w)$ and let $\phi \in\Coo(\R^d)$ with $\phi =1$ on $B_1(0)$ and $\supp \phi \subset B_2(0)$. Define
	\begin{equation}
		f_n(x):= f(x) \phi\left( \frac xn \right), \qquad x \in \Omega, \qquad n \in \N;
	\end{equation}
	then $f \in L^2_w(\Omega)$, $\operatorname{ess} \supp f_n$ is bounded and clearly $f_n \to f$  in $\Ltw$ as $n \to \infty$ by dominated convergence. 	
	To see that $f_n \in \Dom(T_w)$ for fixed $n \in \N$, using that $f \in \cV_w$, $P \in L^\infty(\Omega)$  
	and 
	\begin{equation}\label{eq:nabla.fn}
		\nabla f_{n} = \phi\left(\frac xn\right) \nabla f + \frac 1n f \nabla \phi \left(\frac xn \right) \in \Loneloc(\Omega),
	\end{equation}
	it follows that $f_{n} \in \Vv_w$. Moreover, since $f \in \Vv_{w,0}$, there exists $\{f_{m}\}_m \subset \Coo(\Omega)$ with $f_m \to f$ in $\Vv_w$ as $m \to \infty$. Setting
	\begin{equation}
		f_{n,m} (x) := f_m(x) \phi \left(\frac xn\right) \in \Coo(\Omega), \qquad m \in \N,
	\end{equation}
	using $P\in L^\infty(\Omega)$ and the analogue of~\eqref{eq:nabla.fn} for $\nabla f_{n,m}$, one sees $f_{n,m} \to f_n$ in $\Vv_w$ as $m \to \infty$, which gives $f_n\in \Vv_{w,0}$. Using identity~\eqref{eq:nabla.fn}, we further derive
	\begin{equation}\label{eq:core.approx}
		\begin{aligned}
		-\nabla \cdot (P\nabla f_n) + Vf_n & = \phi \left( \frac xn \right) T_w f - \frac1n \bigg\{ \nabla \phi \left(\frac xn\right) \cdot (P \nabla f)   \\
		& \qquad \quad  + \nabla f \cdot \left(P\nabla \phi \left(\frac xn\right) \right) + f \nabla \cdot \left(P\nabla \phi\left(\frac xn\right)\right)\bigg\}.
		\end{aligned}
	\end{equation}
	By $f \in \Vv_{w}$, writing $\nabla f = P_1^{-\frac 12} P_1^\frac12 \nabla f$ and using \eqref{asm:gn.P} as well as $P_{ij} \in W^{1,\infty}(\Omega)$,  we have that the curly bracket is bounded in $\Ltw$ uniformly in $n \in \N$; in particular, it follows that $f_n \in \Dom(T_w)$. By dominated convergence we moreover have 
	\begin{equation}
		\phi \left( \frac xn\right) T_w f \to  T_w f \quad \text{in} \quad \Ltw,
	\end{equation}
	thus $T_wf_n \to T_wf$ in $\Ltw$ as $ n \to \infty$ and the first density claim is proven.
	
	For the second density claim, let $\Omega = \Rd$ and $f \in \cD_w$ be arbitrary. In particular, $\operatorname{ess} \supp f$ is compact and since $V, w^{-1} \in L^\infty_{\rm loc} (\R^d)$, writing $\nabla f = P_1^{-\frac12} P_1^\frac12 \nabla f$ and using \eqref{asm:gn.P}, one has 
	\begin{equation}\label{eq:f.dom.sep}
			Vf, \,\, -\nabla \cdot(P \nabla f), \,\, f, \, \, |\nabla f|  \, \in L^2_w(\R^d) \cap L^2(\R^d).
		\end{equation}	
		
	For $\eps>0$, let $\phi_\eps$ be a standard mollifier on $\Rd$, see~\cite[Def.\ 2.28]{Adams-2003}. Then
	\begin{equation}
		f_\eps := f * \phi_\eps \in C_0^\infty(\Rd) \subset \cV_{w,0},
	\end{equation}
	and considering $V, w \in L^\infty_{\rm loc} (\R^d)$ and $P_{ij} \in W^{1,\infty}(\Rd)$, \eqref{eq:f.dom.sep} and standard properties of mollifiers, see e.g.~\cite[Thm.~2.29]{EE}, we also have 
	\begin{equation}
		\label{eq:f.eps.dom}
		V f_\eps, \,\, -\nabla \cdot(P \nabla f_\eps) \in L^2_w(\Rd)
	\end{equation}
	as well as $f_\eps \to f$ and $V f_\eps \to V f$ in $L^2_w(\R)$ as $\eps \to 0$; in particular, $f_\eps\in \Dom (T_w)$. 
	It further follows from~\eqref{eq:f.dom.sep},~$P_{ij} \in W^{1,\infty}(\Rd)$ and~\eqref{asm:gn.P} that $f \in W^{2,2}(\Rd)$, see e.g.~\cite[Thm.~8.8]{Gilbarg-1983} and notice that its proof can be modified in a straightforward way under the present assumptions (i.e.~$P_1$ uniformly elliptic).	
	Using $P_{ij} \in W^{1,\infty}(\Rd)$ and $w \in L^{\infty}_{\rm loc} (\Rd)$, we finally arrive at 
	\begin{equation}
		\begin{aligned}
			-	\nabla \cdot(P\nabla f_\eps) &= -\sum_{i,j=1}^d 
			\left(
			P_{ij} \partial_{ij} f_\eps + (\partial_i P_{ij}) \partial_j f_\eps  
			\right)
			\\
			&\qquad \quad \to
			-\sum_{i,j=1}^d 
			\left(
			P_{ij} \partial_{ij} f + (\partial_i P_{ij}) \partial_j f 
			\right)
			=- \nabla \cdot  (P\nabla f)
		\end{aligned}
	\end{equation}
	in $L^2_w(\Rd)$, and thus $f_\eps \to f$ in $\Vv_w$ as $\eps \to 0$. 
\end{proof}

\begin{lemma}[Graph norm separation]
	\label{lem:graph.norm}
	Let Assumption~{\rm\ref{asm:main.gn}} be satisfied with $\eps:= \eps_V+ \eps_w$ as in~\eqref{eps.asm} and let $\cD_w$ be as in \eqref{cD.def}. Then for every $\delta>0$ there exists $C(\delta)> 0$ with
	\begin{equation}\label{gnorm.sep.lem}
		\norm{ T_w f}_{w}^2 \ge (1 - \mu_1 - \delta) \norm{\nabla \cdot (P \nabla f)}_{w}^2 + (1 - \mu_2 - \delta) \norm{Vf}_{w}^2 - C(\delta) \norm{f}_{w}^2
	\end{equation}
	for all $f \in \cD_w$, where $\mu_1, \mu_2 \in (0,1)$ are as in~\eqref{eq:mu} and \eqref{delta} below.
\end{lemma}

\begin{proof}[Proof of Lemma~{\rm\ref{lem:graph.norm}}]
	Fix $f \in \Dd_w$, then the boundedness of $\operatorname{ess} \supp f$ and 
	$V\in W^{1,\infty}_{\rm loc} (\overline \Omega)$ guarantee that 
	\begin{equation}\label{eq:bdd.supp}
		Vf, \, -\nabla \cdot (P\nabla f) \in \Ltw .
	\end{equation}
	
	To justify integrating by parts in what follows, observe first that, by \eqref{eq:bdd.supp}, for every $g \in \cV_{w,0}$ we have
	\begin{equation}\label{IBP}
		\langle -\nabla \cdot (P\nabla f), g \rangle_w = {\mathbf t}_w(f,v) - \langle Vf, g \rangle_w = \langle P \nabla f , \nabla (gw) \rangle.
	\end{equation}
	In~\eqref{sep.est.9}, \eqref{sep.est.2} and~\eqref{sep.est.6} below, we apply the above formula with $g=Vf$, $g=f$ and $g=|V|f$, respectively.
	
	To verify that $Vf \in \cV_{w,0}$, it is easy to check that $Vf \in \cV_w$ using $P \in L^\infty (\Omega)$ and $V \in W^{1,\infty}_{\loc} (\overline\Omega)$, see~\eqref{def.big.cV}.
	Let further $\{f_n\}_{n}\subset \Coo(\Omega)$ be a sequence approximating $f$ in $\Vv_w$. Since $f$ has bounded support, we can choose $f_n$ supported in a fixed bounded neighbourhood of $\operatorname{ess} \supp f$ (or otherwise use a standard cut-off strategy and $P \in L^\infty(\Omega)$ to modify $f_n$ accordingly). By Lemma~\ref{lem:W.1.infty} and $V \in W^{1,\infty}_{\rm loc} (\overline \Omega)$, it follows that
	\begin{equation*}
		Vf_n \in W^{1,\infty}_{\operatorname{comp}}(\Omega) \subset \cV_{w,0}, \qquad n \in \N.
	\end{equation*}
	Using $P \in L^\infty(\Omega)$, $V \in W^{1,\infty}_{\rm loc} (\overline \Omega)$ and $f_n \to f$ in $\cV_w$, one then shows by dominated convergence that $Vf_n \to Vf$ in $\cV_w$ as $n \to \infty$.
	The claim $|V|f\in \Vv_{w,0}$ is justified analogously.	

For the graph norm estimate, we start by computing 
\begin{equation}\label{sep.est.1} 
	\|(-\nabla \cdot (P \nabla) + V) f\|_w^2 
	= \| \nabla \cdot (P \nabla f)\|_w^2 + \|Vf\|_w^2   
	- 2 \re \langle \nabla \cdot (P \nabla f), Vf\rangle_w.
\end{equation}
We need to bound below the last term on the right hand side in terms of the first two positive terms. Integrating by parts leads to
\begin{equation}\label{sep.est.9}
		-  \re \langle \nabla \cdot (P \nabla f), Vf\rangle_w 
	\ge \re  \langle P \nabla f, V \nabla f \rangle_w 
	- |\langle P \nabla f, f (w \nabla V + V \nabla w ) \rangle| .
\end{equation}
The first term on the right of~\eqref{sep.est.9} can be dropped since by \eqref{asm:VP.accr}
\begin{equation}\label{sep.est.10}
	\Re \langle P \nabla f, V \nabla f\rangle_w = \Re \langle  \e^{-\ii \arg V}  P (\abs{V}^\frac{1}{2}\nabla f) ,\abs{V}^\frac{1}{2} \nabla f \rangle_w \ge 0.
\end{equation}
Notice next that, using $P \in L^\infty(\Omega)$, \eqref{asm:gn.P} and~\eqref{eq:sec.matrix}, it is elementary to derive 
\begin{equation}\label{eq:est.CP}
	\|P_1^{-\frac12} P P_1^{-\frac12}\|_{L^\infty} \le 1+C_P, \qquad C_P := \delta_P^{-1}\|P_2\|_{L^\infty}\ge 0.
\end{equation}
Employing \eqref{asm:V.sep.gn} and \eqref{asm:w.sep.gn}, one can thus estimate the second term on the right of~\eqref{sep.est.9} as follows
\begin{equation}\label{sep.est.11}
	\begin{aligned}
		& |\langle P \nabla f, f ( w \nabla V + V \nabla w ) \rangle| \\
		& \qquad \qquad \le \|P_1^{-\frac12} P P_1^{-\frac12}\|_{L^\infty} \langle |P_1^\frac{1}{2} \nabla f|, |f| ( w |P_1^{\frac{1}{2}}\nabla V| + |V| |P_1^{\frac{1}{2}} \nabla w|) \rangle 		\\
		& \qquad \qquad \le (1+C_P) \langle |P_1^\frac{1}{2} \nabla f|,  (\eps |V|^\frac{3}{2} + C_{w,\eps} |V| + C_{V,\eps} ) |f|\rangle_w.
		\end{aligned}
	\end{equation}

To bound the last term, more careful estimates with several loops are needed. First, using Cauchy--Schwarz' and Young's inequalities, we obtain
	\begin{equation}
		\label{eq:sep.main.term}
		\begin{aligned}
			& \langle |P_1^\frac{1}{2} \nabla f|, (\eps |V|^\frac{3}{2} + C_{w,\eps} |V| + C_{V,\eps} )|f| \rangle_w\\
		& \qquad \qquad\qquad \qquad\le \eps \delta_1  \| |V|^\frac12  P_1^\frac{1}{2} \nabla f \|_w^2 + \left(  \frac{C_{w,\eps}^2}{4 \delta}+ \frac{1} 2 \right) \|  P_1^\frac{1}{2} \nabla f \|_w^2 \\
		& \qquad \qquad\qquad \qquad\qquad \qquad \quad  + \left( \frac{\eps}{4 \delta_1} + \delta \right) \| Vf\|_w^2 + \frac{C_{V,\eps}^2}{2} \|f\|_w^2 ,
	\end{aligned}
\end{equation}
where $\delta_1,\delta > 0$ are arbitrary (to be chosen suitably later). We need to derive estimates for the first and second term on the right hand side of \eqref{eq:sep.main.term}.
For the second term, integration by parts gives
	\begin{equation}\label{sep.est.2}
		\begin{aligned}
			\|P_1^\frac{1}{2} \nabla f\|_w^2 
			= \langle P_1 \nabla f, \nabla f\rangle_w 
			& \le |\langle P\nabla f, \nabla f\rangle_w| \\
			& \le | \langle \nabla \cdot (P \nabla f), f\rangle_w| + |\langle P \nabla f,  f\nabla  w \rangle |.
		\end{aligned}
	\end{equation}
	Using \eqref{eq:est.CP},  \eqref{asm:w.sep.gn}, Cauchy--Schwarz' and Young's inequalities we further obtain
	\begin{equation}\label{sep.est.3}
		\begin{aligned}
			|\langle  P \nabla f, f \nabla w \rangle | & \le \|P_1^{-\frac12} P P_1^{-\frac12} \|_{L^\infty} \langle |P_1^\frac{1}{2} \nabla f|, |f| |P_1^{\frac12}\nabla w| \rangle \\ 
			& \le (1+C_P)\langle |P_1^\frac{1}{2} \nabla f|, (\eps_w |V|^\frac12 + C_{w,\eps}) |f|\rangle_w  \\
			& \le  (1+C_P) \Bigg\{ \widetilde \delta \| |V|^\frac12 P_1^\frac{1}{2} \nabla f  \|_w^2 +  \widetilde\delta\| P_1^\frac{1}{2} \nabla f \|_w^2 \\
			& \qquad \qquad \qquad \qquad \qquad \qquad \qquad + \frac {\eps_w^2 + C_{w,\eps}^2}{4 \widetilde\delta}  \|f\|_w^2 \Bigg\}
		\end{aligned}
	\end{equation}
	with $\widetilde\delta>0$ to be selected suitably. Putting together \eqref{sep.est.2} and \eqref{sep.est.3}, using Cauchy--Schwarz' and Young's inequalities and choosing $\widetilde\delta$ small enough such that $ \widetilde\delta (1+C_P)<1$ eventually leads to 
	\begin{equation}\label{sep.est.4}
		\begin{aligned}
			\|P_1^\frac{1}{2} \nabla f\|_w^2 & \le \frac 1 {1 -   \widetilde\delta(1+C_P)} \Bigg\{\widetilde \delta \|\nabla \cdot (P\nabla f )\|_w^2 + \frac{1}{4\widetilde\delta} \|f\|_w^2  +  \\
			& \qquad\qquad\qquad\quad (1+C_P) \Bigg( \widetilde\delta \||V|^\frac12  P_1^\frac{1}{2} \nabla f\|_w^2 + \frac {\eps_w^2 + C_{w,\eps}^2}{4\widetilde\delta} \|f\|_w^2   \Bigg) \Bigg\} .\\
		\end{aligned}
	\end{equation}
	Combining this with \eqref{eq:sep.main.term} we arrive at
	\begin{equation}\label{eq:sep.main.2}
		\begin{aligned}
			& \langle |P_1^\frac{1}{2} \nabla f|, (\eps |V|^\frac{3}{2} + C_{w,\eps} |V| + C_{V,\eps} ) |f| \rangle_w\\
			&  \qquad\quad \le   \frac{\widetilde \delta   }{1-\widetilde\delta (1+C_P)} \left(\frac{C_{w,\eps}^2}{4 \delta}+ \frac{1} 2  \right) \|\nabla \cdot (P\nabla f)\|_w^2  \\
			& \qquad\quad \qquad \qquad+ \left( \frac{\eps}{4 \delta_1} + \delta\right) \|Vf\|_w^2+ \eta \| |V|^\frac12  P_1^\frac{1}{2} \nabla f \|_w^2 + C_{\delta,\widetilde \delta} \|f\|_w^2 
		\end{aligned}
	\end{equation}
	with the constants $\eta, C_{\delta,\widetilde \delta}>0$ given as
	\begin{equation}
		\begin{aligned}
			\eta & :=  \left(\frac{C_{w,\eps}^2}{4  \delta}+ \frac{1} 2  \right) \frac{\widetilde\delta(1+C_P)   }{1-\widetilde\delta (1+C_P)} + \eps\delta_1 , \\
			C_{\delta,\widetilde \delta} & := \left(\frac{C_{w,\eps}^2}{4  \delta}+ \frac{1} 2  \right) \frac {1+ (1+C_P)\big(\eps_w^2 + C_{w,\eps}^2\big)}{4\widetilde\delta \big(1-\widetilde\delta(1+C_P)\big)} + \frac{C_{V,\eps}^2}{2};
		\end{aligned}
	\end{equation}
	notice that, e.g.~with the choice $\widetilde \delta = \delta^2$, we have $\eta = \eps \delta_1+ \BigO (\delta)$ for $\delta \to 0^+$.
	
	To estimate the first term in~\eqref{eq:sep.main.term} we integrate by parts, use Cauchy--Schwarz' and Young's inequalities, \eqref{eq:est.CP} as well as \eqref{asm:V.sep.gn} and \eqref{asm:w.sep.gn} to further derive
	\begin{equation}\label{sep.est.6}
		\begin{aligned}
			& \| |V|^\frac12 P_1^\frac{1}{2} \nabla f  \|_w^2   = \langle  P_1 (|V|^\frac12 \nabla f) ,  |V|^\frac12 \nabla f  \rangle_w \le |\langle  P (|V|^\frac12 \nabla f) ,  |V|^\frac12 \nabla f  \rangle_w | \\[2mm]
			& \qquad\quad\le |\langle  \nabla \cdot (P \nabla f), |V| f \rangle_w|    +|\langle P \nabla f, f ( |V| \nabla w  + w \nabla |V| ) \rangle | \\
			& \qquad\quad\le \delta_2 \|\nabla \cdot (P\nabla f )\|_w^2 + \frac1{4\delta_2} \|Vf\|_w^2 \\
			& \qquad\qquad\qquad +\|P_1^{-\frac12} P P_1^{-\frac12} \|_{L^\infty} \langle | P_1^\frac 12 \nabla f|, |f| (|V| |P_1^{\frac12}\nabla w|  + w |P_1^{\frac12} \nabla |V|| ) \rangle \\
			& \qquad\quad\le \delta_2 \|\nabla \cdot (P\nabla f )\|_w^2 + \frac1{4\delta_2} \|Vf\|_w^2 \\
			& \qquad\qquad\qquad + (1+C_P) \langle  | P_1^\frac 12 \nabla f|, (\eps |V|^\frac32 + C_{w,\eps} |V| + C_V) |f| \rangle_w
		\end{aligned}
	\end{equation} 
	with a suitable $\delta_2>0$ to be selected. Combining~\eqref{eq:sep.main.2} and \eqref{sep.est.6}, we obtain that 
	\begin{equation}\label{sep.est.12}
		\begin{aligned}
		 & \langle |P_1^\frac{1}{2} \nabla f|,  (\eps |V|^\frac{3}{2} + C_{w,\eps} |V| + C_{V,\eps} ) |f| \rangle_w\\
		 & \le \frac{1}{1-\eta (1+C_P)} \Bigg\{  \left( \eta\delta_2 + \frac{\widetilde\delta   }{1-\widetilde\delta (1+C_P)} \left(\frac{C_{w,\eps}^2}{4 \delta}+ \frac{1} 2  \right)  \right) \|\nabla \cdot (P\nabla f)\|_w^2  \\
		 &  \quad \qquad \qquad\qquad \quad + \left(  \frac{\eta}{4\delta_2}+   \frac{\eps}{4 \delta_1} +  \delta\right)   \|Vf\|_w^2 + C_{\delta,\widetilde \delta} \|f\|_w^2 \Bigg\}
		\end{aligned}
	\end{equation}
	as long as $\eta<1/(1+C_P)$. \noeqref{sep.est.9,sep.est.10} From \eqref{sep.est.1}--\eqref{sep.est.11} and \eqref{sep.est.12} it finally follows that
		\begin{equation}\label{graph.norm.sep}
			\begin{aligned}
				\| T_w f\|_w^2  \ge (1-\eta_1) \|\nabla \cdot (P \nabla f) \|_w^2 + (1-\eta_2) \|Vf\|_w^2 - \frac{2(1+C_P)C_{\delta,\widetilde\delta}}{1-\eta (1+C_P)} \|f\|_w^2
			\end{aligned}
		\end{equation}
	with the constants $\eta_1$, $\eta_2>0$ given as
		\begin{equation}
			\begin{aligned}
					\eta_1 & := \frac{2(1+C_P)}{1-\eta (1+C_P)} \left(\eta\delta_2 +  \frac{\widetilde\delta   }{1-\widetilde\delta (1+C_P)} \left(\frac{C_{w,\eps}^2}{4  \delta}+ \frac{1} 2  \right)   \right), \\
					\eta_2 & := \frac{2(1+C_P)}{1-\eta (1+C_P)} \left(\frac{\eta}{4\delta_2} + \frac{\eps}{4 \delta_1} + \delta\right).
			\end{aligned}
		\end{equation}
	
	We discuss the behaviour of $\eta$, $\eta_1$ and $\eta_2$. For fixed $\delta_1,\delta_2$, when setting $\widetilde \delta = \delta^2$ (any power $> 1$ would work), the constant $\eta = \eta(\delta)$ is continuous in $\delta$ around zero. Hence, if 
	\begin{equation}
		\eta(0) = \eps \delta_1 <\frac{1}{1+C_P}
	\end{equation}
	then $\eta < 1/(1+C_P)$ for small $\delta$ and also $\eta_1 = \eta_1 (\delta)$ and $\eta_2 = \eta_2(\delta)$ are continuous in $\delta$ around zero. Moreover,~\eqref{gnorm.sep.lem} is valid with
	\begin{equation}\label{eq:mu}
		\begin{aligned}
			\mu_1 & := \eta_1(0) = \frac{2 \eps (1+C_P) \delta_1 \delta_2}{1-\eps(1+C_P)\delta_1}, \\
			  \mu_2 & := \eta_2(0) = \frac{2 \eps(1+C_P) }{1 - \eps(1+C_P)\delta_1} \left( \frac{\delta_1}{4 \delta_2} + \frac{1}{4 \delta_1} \right).
		\end{aligned}
	\end{equation}
	It remains to discuss that one can select $\delta_1$ and $\delta_2$ such that $\eps \delta_1 < 1/(1+C_P)$ and $\mu_1,\mu_2 <1$. The latter are equivalent to
	\begin{equation}\label{eps.min}
		\eps < \frac{1}{(1+C_P)}\min \left\{\frac{1}{\delta_1(1+ 2 \delta_2 )}, \frac{2\delta_1\delta_2}{\delta_1^2 (1+2\delta_2) + \delta_2} \right\} .
	\end{equation}
	Maximising the minimum along the curve where its entries are equal leads to the choice
	\begin{equation}\label{delta}
		\delta_1 := \frac12, \qquad \delta_2:=\frac{1+\sqrt2}{2};
	\end{equation}
	(in fact, one can show by elementary considerations that the maximum of \eqref{eps.min} is indeed attained on this curve such that the above choice optimises~\eqref{eps.min}). Plugging this into \eqref{eps.min} leads to the restriction \eqref{eps.asm}. Since for such $\eps$ and $\delta_1$ as in~\eqref{delta} we have $\eps \delta_1 < 1/(1+C_P)$, the claim is proven.
\end{proof}

\begin{proof}[Proof of Theorem~{\rm\ref{thm:T.graph.sep}}]
	According to Lemma~\ref{lem:core}, $T_w$ is well-defined. Moreover, the space on the right of \eqref{eq:dom.T.gn} is clearly included in $\Dom (T_w)$. 
	
	On the other hand, for arbitrary $f \in \Dom (T_w)$ by Lemma~\ref{lem:core} there exists $\{f_n\}_n \subset \Dd_w$ such that $f_n \to f$ and $T_w f_n \to T_w f$ in $\Ltw$ as $n \to \infty$. From~\eqref{gnorm.sep.lem} with $f_n$ it follows that $\{Vf_n\}_n$ is Cauchy in $\Ltw$ and thus converges to some $g\in \Ltw$. Using $V \in L^\infty_{\loc}(\Omega)$, we have  
	\begin{equation}
		\langle g, \varphi\rangle_w = \lim_{n \to \infty} \langle Vf_n, \varphi\rangle_w = \lim_{n \to \infty}\langle f_n, \overline V \varphi\rangle_w= \langle f, \overline V\varphi\rangle_w= \langle Vf, \varphi\rangle_w
	\end{equation}  
	for any $\varphi \in \Coo(\Omega)$. By density of $ \Coo(\Omega)$ in $\Ltw$, see e.g.~\cite[Ex.~1.5.3~(c)]{BEH}, it follows that $Vf = g \in \Ltw$ and thus also
	\begin{equation}
		- \nabla \cdot (P \nabla f_n) = T_w f_n - Vf_n \to T_wf - Vf = -\nabla \cdot (P \nabla f) \in \Ltw
	\end{equation}
	converges $\Ltw$. This implies the equality in~\eqref{eq:dom.T.gn}. Moreover, from the above it follows that one can take the limit as $n \to \infty$ in \eqref{gnorm.sep.lem} with $f_n$ and arrives at the claimed graph norm estimate. The last density claim follows from Lemma~\ref{lem:core}.
\end{proof}

\section{Applications and examples}
\label{sec:appl}

\subsection{Completeness of eigensystems of Schr\"odinger operators in $L^2_{w}(\Omega)$}
\label{ssec:complete}

For one dimensional Schr\"odinger operators, we combine Corollaries~\ref{cor:Sp},~\ref{cor:res.bd} and Theorem~\ref{thm:comp.eigensys} to establish the completeness of their eigensystems in $L^2_w(\Omega)$ with suitable weights (depending on the strength of the potential).
More precisely, we consider $T_w$ with $\Omega=\R$, $P=1$ and
\begin{equation}\label{ex:compl.V}
	V \in L^2_{\rm loc}(\R), \qquad  \re V \ge 0, \qquad |V(x)|+1 \gs \langle x\rangle^\gamma, \qquad x \in  \R,
\end{equation} 
with some $\gamma>0$, and the family of weights
\begin{equation}\label{ex:compl.w}
	w(x) = \exp(\pm \langle x \rangle^\alpha), \qquad x \in \R, \qquad 0 <\alpha < 1+ \frac \gamma 2.
\end{equation}
This setting in particular covers the well-known example of the imaginary cubic oscillator $V(x)=\ii x^3$, see \cite{Lidskii-1960-9, Siegl-2012-86, Almog-2015-40} for results on the completeness of the eigensystem in $L^2(\R)$; the proposition below establishes in particular the completeness result for the imaginary cubic oscillator in the weighted spaces $L^2_{w}(\R)$ with the weights~\eqref{ex:compl.w} where $0<\alpha < 1+ 3 /2$.

\begin{proposition}
	Let $\Omega=\R$, $P=1$, and $V$, $w$ be as in~\eqref{ex:compl.V},~\eqref{ex:compl.w} with $\gamma>2$ and assume there exists $\Phi: \R \to [-1,1]$ satisfying Assumption~{\rm\ref{asm:main.LM}}~\ref{asm:LM.nabla.Phi},~\ref{asm:LM.Phi.V2}, (for sufficient conditions, see ~Section~{\rm\ref{sssec:Phi.V}}). Then $T_w$ in $\Ltw$, defined as in \eqref{T.descr}, is densely defined with compact resolvent and the linear span of its root functions (generalised eigenfunctions) is dense in $\Ltw$.
\end{proposition}

\begin{proof}
	Note that Assumption~\ref{asm:main.LM} holds, where for the weights in~\eqref{ex:compl.w} the condition~\eqref{asm:w.sep} is satisfied with arbitrarily small $\tau_w$ and $\kappa_w$. It follows from Corollary~\ref{cor:res.bd} and Theorem~\ref{thm:comp.eigensys} that $T_w$ is densely defined in $L^2_w(\R)$ and that the associated eigensystem is complete if the resolvent is of Schatten class $\cS_p$ with $p<1$. The latter, however, is a consequence of Corollary~\ref{cor:Sp} since 
	\begin{equation}
		p_{\gamma,1} = \frac{2+\gamma}{2\gamma} <1
	\end{equation}
	due to the restriction $\gamma>2$.
\end{proof}

\subsection{Matrix differential operator with non-symmetric off-diagonal}
\label{ssec:ex.1}

As an example with highly non-symmetric off-diagonal, we consider the operator matrix
\begin{equation}\label{cA.nsym}
	\cA =
	\begin{pmatrix}
		-\partial_x^2  & \ii \sinh x
		\\
		x^2    & 0
	\end{pmatrix}
\end{equation}
in $L^2(\R) \oplus L^2(\R)$. It does not satisfy any standard relative boundedness structures (diagonal/off-diagonal/upper/lower dominance, see \cite[Chap.~2]{Tretter-2008}) nor the assumptions of \cite[Thm.~2.3.3, 2.3.7]{Tretter-2008} allowing to relate the spectral properties of $\cA$ and its Schur or quadratic complements.

Our goal here is to apply the results on Schur complement dominance, see Section~\ref{ssec:Schur.dom}, to find a densely defined (Dirichlet) realisation of $\cA$ in $L^2(\R) \oplus L^2(\R)$ with non-empty resolvent set. The analysis is based on the first Schur complement
\begin{equation}\label{Sla.ex.1.def}
	S(\la) \equiv T_w (\la)  = -\partial_x^2 - \la + \frac{\ii }{\la}  x^2 \sinh x, \qquad \la \in \C \setminus\{0\},
\end{equation}
which is realised by Theorem~\ref{thm:t} in a weighted space $L^2_w(\R)$ with a weight
\begin{equation}\label{w.offd.def}
	w(x) = \frac1{1+ \sinh^2(x)}, \qquad x \in \R,
\end{equation}
(which balances the non-symmetric off-diagonal of $\cA$). An important role is played by its operator domain, which according to Theorem \ref{thm:T.graph.sep} is the ($\la$-independent) space
\begin{equation}\label{dom:T.Schur}
	\Dd_S:=\Dom(T_w) = \big\{ f \in \cV_{w,0} \, : \, -f'' \in L^2_w(\R), \, \, x^2\sinh x f \in L^2_w(\R) \big\},
\end{equation}
see~\eqref{cV.0.def}, and its graph norm is equivalent to
\begin{equation}\label{opm.1.DS}
	\|f\|_S^2 := \|f''\|_w^2 + \|x^2 \sinh x f\|^2_w + \|f\|^2_w, \qquad f \in \Dom(T_w).
\end{equation}

We remark (without further details) that since the spectra of~\eqref{Sla.ex.1.def} in $L_w^2(\R)$ and $L^2(\R)$ coincide by Theorem~\ref{thm:inv.Sp}, the spectral equivalence in Theorem~\ref{thm:Schur.dom} can be used to eventually relate the spectrum of $\cA$ to the spectrum of $T_1(\cdot)$ as in~\eqref{Sla.ex.1.def} with $w\equiv1$ realised in $L^2(\R)$. 

\begin{proposition}\label{prop:ex.non-sym.off-diag}
	Let $w$ be as in \eqref{w.offd.def} and let $(\Dd_S,\|\cdot\|_S)$ be the Hilbert space in~\eqref{dom:T.Schur},~\eqref{opm.1.DS}.
	Let $\cA$ be the operator in $L^2 (\R) \oplus L^2(\R)$ acting as in \eqref{cA.nsym} and with the domain
	\begin{equation}\label{ex.SCD.dom}
		\Dom(\cA) := \big\{ (f,g) \in \cD_S \oplus L^2(\R) \, : \, -f'' + \ii \sinh x g \in L^2(\R) \big\}.
	\end{equation} 
	Then $\cA$ is closed, has non-empty resolvent set and its domain is dense in both $\Dd_S \oplus L^2(\R) $ and $L^2(\R) \oplus L^2(\R) $. 
\end{proposition}

\begin{proof}
	The claims follow by applying Theorem~\ref{thm:Schur.dom}. In order to define $\cA$ and $S(\cdot)$ as considered therein, we need to introduce suitable spaces and distribution valued operators. In particular, we employ Theorems~\ref{thm:t} and~\ref{thm:T.graph.sep} to obtain the desired realisation of the Schur complement $T_w(\la)$ in $L^2_w(\R)$ and describe its operator domain.

	\begin{enumerate}[\upshape (a), wide] 
		\item \emph{Schur complement in weighted space.}
		For $\la < -1$, we introduce the Dirichlet realisation of~\eqref{Sla.ex.1.def} in $L^2_w(\R)$ with the weight~\eqref{w.offd.def} and establish the graph norm separation. To this end, we need to verify that the assumptions of Theorems~\ref{thm:t} and~\ref{thm:T.graph.sep} hold with $\Omega = \R$ and
		\begin{equation}
			P=1, \qquad V_1 = - \la, \qquad V_2(x) = \frac{1}\la x^2 \sinh x, \qquad x \in \R.
		\end{equation}
		Clearly, $P$, $V$ and $w$ satisfy Assumptions~\ref{asm:main.LM}~\ref{item:LM.reg}--\ref{item:asm.V.accr} and~\ref{asm:main.gn}~\ref{item:asm.GN.reg}--\ref{item:VP.accr}. Moreover, with the multiplier $\Phi$ and $R$ from Example~\ref{ex:Phi.1d}, it is elementary to see that the conditions \eqref{eq:phi.sgn.rem}, \eqref{eq:rem.rel.bdd.eps} and \eqref{asm:Phi.sep} hold with uniform constants $C_{\Phi,\eps}$ and $C_{R,\eps}$ in $\la <-1$. Moreover, 
	\begin{equation}
	|V_2'(x)| = \frac{1}{|\la|}|2x \sinh x + x^2 \cosh x| \ls |V_2(x)| + 1, \qquad x \in \R,
	\end{equation}
	hence \eqref{asm:V.sep.gn} holds with arbitrarily small $\eps_V$ and uniformly in $\la <-1$. Finally, the weight $w$ from \eqref{w.offd.def} satisfies Assumptions~\ref{asm:main.LM}~\ref{item:LM.reg} and~\ref{asm:main.gn}~\ref{item:asm.GN.reg} and it is admissible according to \eqref{asm:w.sep} and \eqref{asm:w.sep.gn} with $\kappa_w=\tau_w = \eps_w= 0$ therein; indeed, 
	\begin{equation}\label{w'.w.off.diag}
	\frac{|w'(x)|}{w(x)} \leq \frac{2|\sinh x \cosh x|}{1+ \sinh^2 x} \ls 1, \qquad x \in \R.
	\end{equation}
	It follows that \eqref{ellipse} with $\beta = 1$ and~\eqref{eps.asm} hold, such that Theorems~\ref{thm:t} and~\ref{thm:T.graph.sep} apply. In particular, $\Dom(T_w(\la))$ is given by~\eqref{dom:T.Schur}, it is dense in both $\cV_{w,0}$ and  $L^2_w(\R)$, and~\eqref{opm.1.DS} is equivalent to the graph norm of $T_w(\la)$. Since $T_w(\la)$ has non-empty resolvent set, it is closed and hence
	\begin{equation}\label{Tw.bdd.dom}
		T_w (\la) \in \Bb(\Dd_S,L^2_w(\R)).
	\end{equation}
		
		\item \emph{Spaces.} Besides $\Dd_S$, we introduce the remaining spaces needed for Assumption~\ref{asm:schur.dom}~\ref{item.spaces}. Namely,
		\begin{equation}
			\cH_1  = \cH_2 : = L^2(\R), \qquad \cD_2 =\cD_{-2}:=L^2(\R), \qquad \cD_{-S}: =L^2_w(\R).
		\end{equation}
		Notice that $\Dd_S \subset L^2(\R)$ is not immediate. However, it is straightforward to check that 	
		\begin{equation}
			\frac{x^4 \sinh^2 x}{1+\sinh^2 x} + \frac1{1+\sinh^2 x}   \approx x^4 + 1, \qquad x \in \R,
		\end{equation}
		hence
		\begin{equation}\label{S.norm.equiv}
			\|f\|^2_{S} \approx \|f''\|_w^2 + \|x^2 f\|^2 + \|f\|^2, \qquad f \in \cD_S.
		\end{equation}
		
		Considering~\eqref{dom:T.Schur} with $\Vv_{w,0} \subset \Vv_w$, see~\eqref{def.big.cV}, and~\eqref{opm.1.DS}, this implies $\|f\| \ls \|f\|_S$ for all $f \in \Dd_S \subset L^2 (\R)$ and the inclusion is dense since $\Coo(\R) \subset \Dd_S$. The continuous and dense embedding $L^2(\R) \subset L^2_w(\R)$ is clear; recall that $\Coo(\R)$ is dense in $L^2_w(\R)$, see e.g.~\cite[Ex.~1.5.3~(c)]{BEH}.

		\item \emph{Distributional operator matrix and Schur complement.} 
		We define the entries of $\wh \cA$ in Assumption~\ref{asm:schur.dom}~\ref{item.entries.1} and verify their mapping properties. In detail, 
		the needed boundedness for $\widehat A := -\partial_x^2$ follows from
		\begin{equation}
			\| \wh Af\|_{\cD_{-S}} = \|f'' \|_w \leq \|f \|_{S}, \qquad f \in  \Dd_S.
		\end{equation}
		The operator $\wh D:= 0$ is trivially bounded as claimed, and $\widehat \rho(D) = \C \setminus \{0\}$, see \eqref{eq:hat.rho}.
		Next, the multiplication $\wh B : = \ii \sinh x$ satisfies
		\begin{equation}
			\|\wh B f \|_{\cD_{-S}} = \| \ii \sinh x f \|_{w} \leq \|f\|, \qquad f \in L^2(\R),
		\end{equation}
		and from \eqref{S.norm.equiv} we obtain for $\wh C:=x^2$ that 
		\begin{equation}
			\| \wh C f\| = \| x^2 f\| \ls \|f\|_{S}, \qquad f \in \cD_{S}.
		\end{equation}
		It follows that the above defined operators satisfy Assumption~\ref{asm:schur.dom}~\ref{item.entries.1}, such that we can introduce $\widehat S(\la) \in \cB(\cD_{S}, L^2_w(\R))$ by the formula~\eqref{eq:def.hat.S} for all $\la \neq 0$. Considering $\la < -1$, one can easily check that, see~\eqref{T.descr},
		\begin{equation}
			\widehat S(\la) f = -f'' - \la  f  + \frac{\ii }{\la}  x^2 \sinh x f =  T_w(\la) f , \qquad f \in \Dd_S = \Dom (T_w).
		\end{equation}
		
		\item \emph{Bounded invertibility of $\wh S(\la)$.}		
		To apply Theorem~\ref{thm:Schur.dom}, we justify that there exists a sufficiently negative $\la <-1$ such that 
		\begin{equation}
			\wh S(\la) = T_w(\la) \in \cB(\cD_{S}, L^2_w(\R))
		\end{equation}
		is boundedly invertible. To this end, it suffices to show that $0 \in \rho(T_w(\la))$. The latter follows from Theorems~\ref{thm:lm.0}, \ref{thm:lm.1} and the coercivity estimates \eqref{t.gcoer.1} which hold with constants $\gamma_j$, $m_j$, $j=1,2$, independent of $\la < -1$ (since Assumption~\ref{asm:main.LM} is satisfied uniformly in $\la$, see above). 
		Indeed, we obtain for all $f \in \cV_{w,0}$ that
		\begin{equation}
			\Re \mathbf t_w(f,f) + |\Im \mathbf t_w(\Phi f,f)| \geq 
			m_1 \|f\|_{\cV_w}^2 -  \gamma_1\|f\|_w^2. 
		\end{equation}
		Employing \eqref{w'.w.off.diag}, we obtain
		\begin{equation}
		\begin{aligned}
		\Re \mathbf t_w(f,f) & \geq \|f'\|_w^2 - |\langle f', w' w^{-1} f \rangle_w| +|\la|\|f\|_w^2
		\\
		& 		
		\geq \frac12 \|f'\|_w^2 + \left(|\la|- 2 \|w' w^{-1}\|_{L^\infty}  \right) \|f\|_w^2 
		\\ &\geq
		\left(|\la|- 2 \|w' w^{-1}\|_{L^\infty}  \right) \|f\|_w^2. 
		\end{aligned}	
		\end{equation}
		Thus, for all $\la< -\la_0$ with $\la_0$ sufficiently large,
		\begin{equation}
		\begin{aligned}
		\Re \mathbf t_w(f,f) + |\Im \mathbf t_w(\Phi f,f)| &\geq  
		\frac12 \left(
			m_1 \|f\|_{\cV_w}^2 + \left(|\la| -  \gamma_1 - 2 \|w' w^{-1}\|_{L^\infty}\right)\|f\|_w^2
				\right)
		\\
		&\geq \frac{m_1}2 \|f\|_{\cV_w}^2;
			\end{aligned}
		\end{equation}
		the second inequality in \eqref{lm.gen.coev} is verified analogously.
		
	\end{enumerate}

	In summary, the assumptions of Theorem~\ref{thm:Schur.dom} are satisfied with $\Sigma = (-\infty,\la_0) $ for some $\la_0 <-1$, where \eqref{Schur.asm} holds with $z_\la = 0$. It follows that the operator matrix $\cA$ in \eqref{cA.nsym} and \eqref{ex.SCD.dom} has the claimed properties; in particular $(-\infty,\la_0) \subset \rho(S(\cdot))$ and hence also $(-\infty,\la_0) \subset \rho(\Aa)$. Notice that \eqref{ex.SCD.dom} indeed coincides with the maximal domain in~\eqref{eq:Schur.max.dom}.
\end{proof}

\subsection{Diagonally dominant matrix Schr\"odinger operator in a weighted space}

Similarly as in Section~\ref{ssec:ex.1}, a densely defined realisation of the matrix Schr\"odinger operator

\begin{equation}\label{diag.dom.weight}
	\cA =
	\begin{pmatrix}
		-\partial_x^2 + \ii x^3 & x
		\\
		x^5 & -\partial_x^2 + x^4
	\end{pmatrix}
\end{equation}
in $L^2(\R) \oplus L^2(\R)$ with non-empty resolvent set can be found using Schur complement dominance. However, one can also select suitable weights $w_1$ and $w_2$ such that~\eqref{diag.dom.weight}  becomes diagonally dominant in the product space $L_{w_1}^2(\R) \oplus L_{w_2}^2(\R)$. 

To be more precise, consider the Dirichlet realisations of
\begin{equation}
	A:=-\partial_x^2 + \ii x^3, \qquad D:=-\partial_x^2 + x^4
\end{equation}
in the respective spaces $L^2_{w_1}(\R)$ and  $L^2_{w_2}(\R)$ with $w_1 (x): =\langle x \rangle^\frac 52$ and $w_2 (x):= \langle x \rangle^{-\frac 52}$ for  $x \in \R$. Indeed, it is straightforward to show that Theorems~\ref{thm:t} and~\ref{thm:T.graph.sep} apply. From the graph norm estimate in the latter one derives
\begin{equation}\label{ex.2.AD.est}
	\begin{aligned}
		\|Af\|_{{w_1}} + \|f\|_{{w_1}} & \gs \|f''\|_{{w_1}} + \| \langle x \rangle ^3 f\|_{{w_1}}, \quad & f \in & ~\Dom(A),
		\\
		\|Dg\|_{{w_2}} + \|f\|_{{w_2}} & \gs \|g''\|_{{w_2}} + \| \langle x \rangle^4 g\|_{{w_2}}, \quad & g \in & ~\Dom(D).
	\end{aligned}
\end{equation}
Moreover, we consider the multiplication operators
\begin{equation}
	B:=  x :L^2_{w_2} (\R) \to  L^2_{w_1} (\R), \qquad C:= x^5 : L^2_{w_1} (\R) \to  L^2_{w_2} (\R),
\end{equation}
defined on their maximal domains (which by~\eqref{ex.2.AD.est} and~\eqref{ex.2.BC.est} below contain the domains of $D$ and~$A$, respectively). Using H\"older's and Young's inequalities, we see that for any $\eps>0$ there exists $C_{\eps}\ge 0$ such that
\begin{equation}\label{ex.2.BC.est}
	\begin{aligned}
		\|Cf\|_{{w_2}}  \ls  \|\langle x \rangle^\frac52 f\|_{{w_1}} & \ls \eps \| \langle x \rangle^3 f\|_{{w_1}} + C_{\eps} \| f\|_{{w_1}}, \quad & f \in & ~\Dom(A), 
		\\
		\|Bg\|_{{w_1}} \ls  \|\langle x \rangle^\frac72 g\|_{{w_2}} & \ls \eps \| \langle x \rangle^4 g\|_{{w_2}} + C_{\eps} \| g\|_{{w_2}}, \quad & g \in &~\Dom(D). 
	\end{aligned}
\end{equation}
Combining \eqref{ex.2.AD.est} and \eqref{ex.2.BC.est}, diagonal dominance of order zero is established.

We remark that the effect of considering weighted spaces can equivalently be explained as a transformation of $\cA$ from $L_{w_1}^2(\R) \oplus L_{w_2}^2(\R)$ to $L^2(\R) \oplus L^2(\R)$, i.e.~by employing the conjugation
\begin{equation*}
\diag (w_1, w_2)^\frac12 \cA \diag ( w_1, w_2)^{-\frac12}
	= 	\begin{pmatrix}
		w_1^\frac12(-\partial_x^2 + \ii x^3)w_1^{-\frac12} &  w_1^\frac12 x  w_2^{-\frac12}
		\\
		w_2^\frac12 x^5  w_1^{-\frac12} & w_2^\frac12 (-\partial_x^2 + x^4) w_2^{-\frac12}
	\end{pmatrix}.
\end{equation*}
Similarly to choosing suitable constants $w_1, w_2 >0$ e.g.\ in~\cite{Rasulov-2018-48}, we select weights $w_1$ and $w_2$ in order to partially balance the off-diagonal terms.

\subsection{Damped wave equation in weighted space with accretive and differential damping}
\label{ssec:dwe}

We consider the following damped wave equation
\begin{equation}\label{DWE.eq}
	u_{tt}(t,x) + \left(a_1(x) + \ii a_2(x) - \nabla_x \cdot ( a_0(x) \nabla_x)\right) u_t (x,t)= (\Delta_x - q(x)) u(t,x)
\end{equation}
for $t>0$ and $x \in \Omega$ with an open set $\Omega \subset \Rd$, imposing Dirichlet boundary conditions on $\partial \Omega$. Here we assume  that
\begin{equation}\label{dwe.reg}
	0 \le a_1 ,q \in L^1_{\rm loc}(\Omega), \quad a_2, |\nabla a_2| \in L^1_{\rm loc} (\Omega;\R), \quad 0 \le a_0 \in L^1_{\rm loc}(\Omega)^{d \times d}.
\end{equation}
By a standard procedure, \eqref{DWE.eq} can be rewritten as a system
\begin{equation}
	\partial_t \begin{pmatrix}
		u_1 \\ u_2 
	\end{pmatrix}
	= \cA 
	\begin{pmatrix}
		u_1 \\ u_2 
	\end{pmatrix}
\end{equation}
where $\cA$ is the operator matrix
\begin{equation}\label{DWE.cA.def}
	\cA: =
	\begin{pmatrix}
		0 & I
		\\
		\Delta -q & - (a_1 + \ii a_2 - \nabla \cdot ( a_0 \nabla))
	\end{pmatrix}.
\end{equation}
Our goal is to show that a suitable realisation of $-\cA$ generates a semigroup in the product space 
\begin{equation}\label{DWE.H}
	\cH_w := \cW_w \oplus L_{w}^2(\Omega),
\end{equation}
where $\cW_w$ is the Hilbert space  completion
\begin{equation}\label{DWE.cW_w.def}
	\Ww_w := \overline{\Coo(\Omega)}^{\|\cdot \|_{\Ww_w}}, \qquad  \|f\|_{\cW_w}^2 :=  \|\nabla f\|_w^2 + \| q^\frac12  f\|_w^2,
\end{equation}
with a suitably chosen weight $w$. To this end, we employ dominance of the second Schur complement; see Section~\ref{ssec:Schur.dom} with the roles of $\Hh_1$ and $\Hh_2$ exchanged and the Assumptions and results understood accordingly.

The main idea is to find a good domain for the operator matrix in terms of its second Schur complement, which formally reads as the differential expression
\begin{equation}
	S(\la) = -\frac{1}{\la} \left( -\nabla \cdot \big((I_{\C^d} + \la a_0 )\nabla\big) + q + \la (a_1 + \ii a_2) + \la^2\right),
\end{equation}
depending on the spectral parameter $\la >0$, in the weighted space $\Ltw$. Considering that $T_w = - \la S(\la)$ with the coefficients
\begin{equation}
	\label{DWE.T.coeff}
	P:= I_{\C^d}+\la a_0, \qquad V_1:= q+ \la a_1 + \la^2 , \qquad V_2 := \la a_2, \qquad \la >0,
\end{equation}
we base the analysis on our previous results about the Dirichlet realisation of $T_w$ and its properties. An important role is played by its form domain
\begin{equation}\label{DWE.cD_S.def}
	\Dd_S : = \left(\Vv_{w,0}, \|\cdot \|_S\right),
\end{equation}
i.e.~$\cV_{w,0}$ in~\eqref{cV.0.def} and~\eqref{DWE.T.coeff} equipped with the equivalent ($\la$-independent) norm
\begin{equation}\label{DWE.S.norm}
	\|f\|_{S}^2 : =  \|(I_{\C^d}+a_0)^\frac12 \nabla f\|_w^2 + \| q^\frac12  f\|_w^2  + \||a_1+\ii a_2|^{\frac 12}f\|_w^2 + \|f\|_w^2.
\end{equation}

The method of Schur dominance in Section~\ref{ssec:Schur.dom} then leads to $\Aa$ in~\eqref{DWE.cA.def} having the desired properties on
\begin{equation}
	\label{DWE.dom.Aa}
	\Dom (\Aa) := \Big\{(f,g) \in \Ww_w \oplus \Dd_S \, : \, (\Delta - q) f - (a_1 + \ii a_2 - \nabla \cdot (a_0\nabla) )g \in \Ltw \Big\},
\end{equation}
with its action understood in a weak way; for details see~\ref{item.DWE.op} in the proof of Proposition~\ref{prop:DWE} below.

\begin{proposition}\label{prop:DWE}
	Assume that $a_0$, $a_1$, $a_2$, $q$ are as in~\eqref{dwe.reg} and $w \in W^{1,\infty}_{\rm loc}(\Omega)$ is uniformly positive on compact subsets of $\Omega$. Let further $\cW_w$, $\Hh_w$, $\cD_{S}$ be as in \eqref{DWE.cW_w.def}, \eqref{DWE.H}, \eqref{DWE.cD_S.def}, and let $\cA$ be as in \eqref{DWE.cA.def}, \eqref{DWE.dom.Aa}. Suppose further that
	\begin{enumerate}[\upshape (i)]
		\item there exists $C_a \ge0$ and, for every $\delta >0$ there exists $C_\delta \ge 0$, such that
		\begin{equation}\label{asm.DWE.a2}
			\begin{aligned}
				|\nabla a_2 |& \le (1+a_2^2)^\frac32 \left(\delta q^\frac12 + C_a  |a_1+\ii a_2|^\frac12 + C_\delta\right), \\
				|a_0^\frac12 \nabla a_2 |& \le  (1+a_2^2)^\frac32 \left( \delta |a_1+\ii a_2|^\frac12 + C_\delta\right);
			\end{aligned}
		\end{equation}
		\item there exist $c, c_0>0$ and $\eps_0\in (0,2)$ such that
		\begin{equation}\label{DWE.w.asm}
			\frac{|\nabla w|}{w} \le c  (a_1 + 1)^\frac 12, \qquad 
			\frac{| a_0^\frac12  \nabla w|}{w} \leq \eps_0  (a_1 + c_0)^\frac12.
		\end{equation}	
	\end{enumerate}
	Then $\cA$ generates a $C_0$-semigroup in $\Hh_w$ and $\Dom(\Aa)$ is dense in $\cW_w \oplus \cD_{S}$ and~$\cH_w$.
\end{proposition}
\begin{proof}
	The proof is based on Theorem~\ref{thm:Schur.dom} and split into several parts justifying its assumptions and applicability.
	
	\begin{enumerate}[\upshape (a), wide] 
		
		\item \emph{Spaces.} We start by introducing the Gelfand triples and operators in Assumption~\ref{asm:schur.dom}. Let $\Hh_1 := \Ww_w$, $\Hh_2:=\Ltw$ and  $\Dd_{1} = \Dd_{-1}:= \Ww_w$, i.e.~we do not employ a non-trivial triple in the first component. For the second component, $\Dd_S$ is defined in~\eqref{DWE.cD_S.def} and we set $\Dd_{-S} := \Dd_S^*$ to be its (anti-)dual space. Then Assumption~\ref{asm:schur.dom}~\ref{item.spaces} is clearly satisfied, where the inclusions
		\begin{equation}
			\Dd_S\subset \Ltw \subset \Dd_S^*
		\end{equation}
		are understood as in~\eqref{Gelfand}; recall that $\Dd_S=\Vv_{w,0}$ with \eqref{DWE.T.coeff} up to equivalence of norms, and that~$\Vv_{w,0} \subset \Ltw$ is dense and boundedly embedded.  
		
		\item \label{item.DWE.op} \emph{Operator matrix.} To define the operators in Assumption~\ref{asm:schur.dom}~\ref{item.entries.1}, we first note that one can boundedly embed $\Dd_S$ as a (dense) subspace of $\Ww_w$; this is a consequence of~\eqref{DWE.S.norm} being a closable form in $\Ww_w$, cf.~\cite[Prop.~4.6]{Gerhat-2024-286}. We then define
		\begin{equation}
			\widehat A:= 0 \in \Bb(\cW_w), \qquad \widehat B:=I_{\Dd_S\to\Ww_w} \in \Bb(\cD_{S},\cW_w).
		\end{equation}
		It follows that $\widehat A$ satisfies an analogue of \eqref{eq:hat.rho} for $\la \in \rho(\widehat A) = \C \setminus\{0\}$. The remaining two entries are introduced weakly. In detail, 
		\begin{equation}
			\widehat C:= \Delta -q \in \Bb(\cW_w,\cD_S^*)
		\end{equation}
		is the unique bounded extension of
		\begin{equation}
			( (\Delta-q) f,g)_{\cD_{S}^* \times \cD_{S}} := \langle \nabla f, \nabla (gw) \rangle - \langle qf, g \rangle_{w}, \quad f \in C_0^\infty(\Omega), \quad g \in \cD_S.
		\end{equation}
		Indeed, this is well-defined since 
		\begin{equation}
			\begin{aligned}
				|\langle \nabla f, \nabla (gw) \rangle - \langle qf, g \rangle_{w}| & \leq |\langle \nabla f, \nabla g \rangle_{w}| + |\langle \nabla f, g \nabla w \rangle| + | \langle qf,  g \rangle_w | \\
				& \ls \|f\|_{\cW_w} \|g\|_{{S}}
			\end{aligned}
		\end{equation}
		by Cauchy--Schwarz and the first inequality in~\eqref{DWE.w.asm}.
		Finally, we define	
		\begin{equation}\label{DWE.D}
			\widehat D:=- (a_1 + \ii a_2 - \nabla \cdot (a_0 \nabla))  \in \Bb(\cD_{S},\cD_S^*)
		\end{equation}
		by 
			\begin{equation}
				( - (a_1 + \ii a_2 - \nabla \cdot (a_0 \nabla)) f,g)_{\cD_{S}^* \times \cD_{S}} := - \langle  (a_1 + \ii a_2) f, g \rangle_w - \langle a_0\nabla f, \nabla (gw) \rangle
			\end{equation}
			for $f,g \in \cD_S$. To verify the boundedness property \eqref{DWE.D}, it is sufficient to use Cauchy--Schwarz' inequality and the second inequality in~\eqref{DWE.w.asm} to estimate
		\begin{equation}
			\begin{aligned}
				&|((a_1 + \ii a_2 - \nabla \cdot( a_0 \nabla)) f, g )_{\cD_S^* \times \cD_S} | 
				\\	
				&  \qquad	\quad\qquad	\quad \leq \||a_1 + \ii a_2|^\frac12 f\|_{w} \||a_1 + \ii a_2|^\frac12 g\|_{w} 
				\\ & \qquad	\qquad \qquad \qquad \qquad 
				+ \|a_0^\frac12 \nabla f\|_{w}  \|a_0^\frac12 \nabla g\|_{w}  + \|a_0^\frac12 \nabla f\|_{w} \|g w^{-1} a_0^\frac12  \nabla w \|_w
				\\ & \qquad	\quad\qquad	\quad \ls \|f\|_{{S}} \|g\|_{{S}}.
			\end{aligned}
		\end{equation}
		Hence, Assumption~\ref{asm:schur.dom} fully holds in the above setting with the distribution valued operator matrix
			\begin{equation}\label{DWE.bdd.cA.def}
				\widehat \cA \in \Bb(\cD, \cD_-), \qquad \cD:= \cW_w \oplus \cD_S, \qquad \cD_- := \cW_w \oplus \cD_S^* = \cD^*.
			\end{equation}
			It can be shown analogously as in Theorem~\ref{thm:t} that the domain of the associated maximal operator $\cA$ in $\cH_w$ defined in~\eqref{Aa.S.max} and \eqref{eq:Schur.max.dom} is given by \eqref{DWE.dom.Aa}, and that its action is \eqref{DWE.cA.def} understood in the standard distributional way. Notice that the latter is well-defined since~\eqref{dwe.reg} guarantees that 
			\begin{equation}
				(a_1 + \ii a_2)g \in \Loneloc(\Omega), \qquad a_0 \nabla g \in \Loneloc(\Omega)^d, \qquad g \in \Dd_S \subset \Vv_w,
			\end{equation}
			see~\eqref{def.big.cV} with~\eqref{DWE.T.coeff}, while for $f \in \cW_w$, $(\Delta+q)f$ is the limit of $(\Delta+q)f_n$ in $\cD'(\Omega)$, where $\{f_n\}_n \subset C_0^\infty(\Omega)$ approximates $f$ in $\cW_w$; cf.~\cite[App.~A]{Arnal-Gerhat-Royer-Siegl-2026-arxiv} for an explicit description of the completion $\cW_w$ and the action of $\Delta-q$ on it (in the case $w=1$).

		\item \label{mu.accr} \emph{Accretivity.} We show that $-\cA + \mu$ is accretive for sufficiently large $\mu>0$.	Let $f_1,f_2 \in C_0^\infty(\Omega)$ and compute
		\begin{equation}
			\begin{aligned}
				& ( -\widehat \cA (f_1,f_2), (f_1,f_2) )_{\cD^*\times \cD} 
				\\
				& \qquad \qquad  = 
				- \langle f_2,f_1 \rangle_{\cW_w} - ( (\Delta-q) f_1 - (a_1+\ii a_2 - \nabla \cdot (a_0 \nabla)) f_2, f_2 )_{\cD_S^* \times \cD_S}
				\\
				&  \qquad \qquad = - \langle \nabla f_2, \nabla f_1 \rangle_{w} - \langle qf_2,  f_1 \rangle_{w}  
				+ \langle \nabla f_1, \nabla f_2 \rangle_{w} 
				\\
				& \qquad \qquad\qquad \qquad + \langle \nabla f_1, f_2   \nabla w \rangle + \langle qf_1, f_2 \rangle_{w} + \| a_1^\frac12 f_2 \|^2_{w} 
				\\
				&  \qquad \qquad\qquad \qquad			   
				+ \ii \langle  a_2 f_2, f_2 \rangle_{w}   + \|a_0^\frac12 \nabla f_2\|^2_{w}   + \langle a_0 \nabla f_2,  f_2 \nabla w \rangle.
			\end{aligned}
		\end{equation}	
		For $\mu>0$ and denoting $\iota_\cA:= I_{\cD \to \cH_w}$, Cauchy--Schwarz' and Young's  inequalities with $\delta>0$ then yield
		\begin{equation}\label{A.accr}
			\begin{aligned}
				&\Re ( (-\widehat \cA+\mu \iota_\cA^* \iota_\cA) (f_1,f_2), (f_1,f_2) )_{\cD^*\times \cD} 
				\\
				& \qquad \geq  - \frac{1}{4 \delta} \|\nabla f_1\|_{w}^2 
				- \delta \|  f_2 w^{-1} \nabla w\|_{w}^2
				+  \|a_1^\frac12  f_2\|^2_{w} - \frac14 \| f_2 w^{-1} a_0^\frac12\nabla w  \|_{w}^2 
				\\ 
				& \qquad  \qquad\quad
				+ \mu \left( \|\nabla f_1\|_{w}^2 + \|q^\frac12 f_1\|_{w}^2 + \|f_2\|_{w}^2 \right).
			\end{aligned}
		\end{equation}	
		Using~\eqref{DWE.w.asm}, choosing $\delta$ sufficiently small and $\mu$ sufficiently large, we see that the right hand side of \eqref{A.accr} is non-negative. By~\eqref{DWE.bdd.cA.def} and the density of $\Coo(\Omega)$ in both $\cD_S$ and $L^2_{w} (\Omega)$, this can be extended to $(f_1, f_2) \in \cD$. For $(f_1, f_2)\in \Dom (\cA)$, this gives
		\begin{equation}
			\begin{aligned}
				& \Re \langle (\cA+\mu) (f_1,f_2), (f_1,f_2) \rangle_{\cH_w} 
				\\
				& \qquad \qquad \quad = \Re ( (-\widehat \cA+\mu \iota_\cA^* \iota_\cA) (f_1,f_2), (f_1,f_2) )_{\cD^*\times \cD} 
				\ge 0.
			\end{aligned}
		\end{equation} 	
		
		\item \label{item.DWE.S.T} \emph{Schur complement.} Defined by the analogous formula to~\eqref{eq:def.hat.S}, the second Schur complement is
		\begin{equation}
			\widehat S(\la) := - (a_1 + \ii a_2 - \nabla \cdot (a_0 \nabla)) - \la + \frac{1}{\la} (\Delta-q)_{\restriction_{\cD_S}} \in \Bb(\cD_{S},\cD_{S}^*).
		\end{equation}
		While the above is defined for $\la \in \C \setminus\{0\}$, we consider only $\la >0$, since then $\|\cdot\|_S$ is equivalent to $\|\cdot \|_{\Vv_w}$ defined with the data~\eqref{DWE.T.coeff}. One can see easily from the definitions that 
		\begin{equation}
			-\la (\widehat S(\la) f,g )_{\Dd_S^*\times \Dd_S} = ( \widehat T_w f, g)_{\Vv_{w,0}^*\times \Vv_{w,0}} := \mathbf t_w (f,g), \qquad f, g \in\Coo(\Omega),
		\end{equation}
		cf.~\eqref{T.hat}. By density of $\Coo(\Omega)$ in $\Dd_S$, we obtain
		\begin{equation}\label{eq:equiv.S.T}
			- \la   \widehat S(\la) = (I_{\Dd_S \to \Vv_{w,0}} )^*	\widehat T_w I_{\Dd_S \to \Vv_{w,0}}, \qquad \la >0. 
		\end{equation}
		In particular, $\widehat S(\la)$ is boundedly invertible if and only if $\widehat T_w$ is.

		\item \emph{Bounded invertibility.} We show that, for $\la >0$ large enough, $\widehat T_w = \widehat T_w(\la)$ is boundedly invertible. To this end, we apply Theorem~\ref{thm:t} in its extended version in Remark~\ref{rem:gcoer}~\ref{asm.gen.LM} (more precisely Proposition~\ref{prop:g.coer}) with the coefficients as in~\eqref{DWE.T.coeff}.
		To verify the assumptions, we employ the multiplier
		\begin{equation}
			\Phi_a := \frac {a_2} {\sqrt {1+a_2^2}}.
		\end{equation}
		
		We determine the respective constants in Assumption~\ref{asm:main.LM} and Remark~\ref{rem:gcoer}~\ref{asm.gen.LM}. Clearly, $C_P=0$ and it is straightforward to see that the conditions \eqref{eq:phi.sgn.rem} and \eqref{eq:rem.rel.bdd} hold with $R=|\Phi V_2 - |V_2||$, $\vartheta_R=\kappa_R=\tau_R=0$ and $C_R=\la$. The weight $w$ is admissible since for all $\la> \la_0$ with $\la_0>1$ sufficiently large
		\begin{equation}
			\begin{aligned}
				\frac{|P_1^\frac12 \nabla w |}{w} & \leq c (a_1 + 1)^\frac12 + \la^\frac12 \eps_0 (a_1 + c_0)^\frac12 
				\le 
				\left(\eps_0 + \la^{-\frac 12}c \right) \left(q+ \la a_1 + \la^2 \right)^\frac12.
			\end{aligned}
		\end{equation}
		Hence, since $\eps_0<2$ by assumption, \eqref{asm:w.sep} is satisfied with $\kappa_w = \kappa_w(\la_0) <2$, $\tau_w=0$ and $C_w =  0$ for all $\la>\la_0$. Thus we can find $\beta = \beta(\la_0) \in (0,1)$ such that \eqref{ellipse.gen} holds. 
		
		The already fixed values of $C_P$, $\tau_R$, $\vartheta_R$, $\kappa_R$, $\kappa_w$, $\tau_w$ and $\beta$ determine the value of $\ec = \ec(\la_0)$ (it is important that this is independent of $\la >\la_0$), see Remark~\ref{rem:gcoer}~\ref{item.ecrit}. Fix $\eps_0 \in (0,\ec)$. To justify \eqref{asm:Phi.sep}, we find
		\begin{equation}
			|P_1^\frac12 \nabla \Phi_a| = \langle (I_{\C^d} + \la a_0) \nabla \Phi_a, \nabla \Phi_a \rangle_{\C^d}^\frac12 
			\leq (1+a_2^2)^{-\frac 32} \left(
			|\nabla a_2| + \la^\frac12 |a_0^\frac 12 \nabla a_2|\right).
		\end{equation}
		Using~\eqref{asm.DWE.a2}, one estimates
		\begin{equation}
			|P_1^\frac12 \nabla \Phi_a| 
			\leq \delta q^\frac12 + \left(\la^{-\frac12} C_a + \delta \right)
			\la^\frac12 |a_1+ \ii a_2|^\frac12 + C_\delta\left(1+\la^\frac12\right).
		\end{equation}
		Selecting $\delta$ small enough, we see that there exists $\la_1 = \la_1(\eps_0) > \la_0$ such hat~\eqref{asm:Phi.sep} holds with $C_{\Phi,\eps_0}=0$ for every $\la > \la_1$.		
		
		In summary, we proved that there exists $\la_1>0$ such that the form $\mathbf t_w (\la) $ satisfies Proposition~\ref{prop:g.coer} for all $\la>\la_1$; the constants $m_1$, $m_2$ therein are independent of $\la$ and  $\gamma_1$, $\gamma_2$ depend only through $C_R = \la$. Observing formulas~\eqref{gcoer.mu.alph},~\eqref{eq:m.gamma.1}, and that one can select $\gamma_2 = \gamma_1$, it follows that $\gamma_i = \BigO(\lambda)$ as $\la \to \infty$. Since
		\begin{equation}
			\|f \|^2 _{\cV_{w}} \geq \|V_1^\frac12 f\|_w^2 \geq \la^2 \|f\|_w^2, \qquad f \in \cV_{w,0},
		\end{equation}
		we arrive at 
		\begin{equation}
			\Re \mathbf t_w(f,f) + \Im \mathbf t_w(\beta \Phi f,f)  \ge \frac{m_1}2 \|f \|^2 _{\cV_{w}} + \left( \frac{m_1}{2} \la^2 - \gamma_1 \right) \|f \|^2_{w} \gtrsim \|f \|^2 _{\cV_{w}},
		\end{equation}
		for all sufficiently large $\la$. Analogously, for all sufficiently large $\la$ we also have 
		\begin{equation}
			\Re \mathbf t_w(f,f) + \Im \mathbf t_w ( f, \beta \Phi f)  \gtrsim \|f \|^2 _{\cV_{w}}.
		\end{equation}
		Applying Theorem~\ref{thm:lm.0}, we indeed justified that $\widehat T_w (\la)$, and thus by~\eqref{eq:equiv.S.T}  also $\widehat S(\la)$, is boundedly invertible for $\la$ sufficiently large.
		
		\item \emph{Dominance of Schur complement.}		
		In the previous steps, we verified the assumptions of Theorem~\ref{thm:Schur.dom} with $\Sigma = [\la_2, \infty)$ and a suitable $\la_2>0$, and with $z_\la = 0$ in~\eqref{Schur.asm}. This gives $\Sigma \subset \rho(\Aa)$ and the density of $\Dom (\cA)$ in both $\cW_w \oplus \cD_S$ and $\cH_w$. Since we proved in~\ref{mu.accr} that $-\cA+\mu$ is accretive in $\cH_w$, it is m-accretive. Hence, $-\cA$ indeed generates a $C_0$-semigroup, see~e.g.~\cite[Thm.~8.3.4]{Davies-2007}. \qedhere
	\end{enumerate}
\end{proof}

\appendix

\section{Used known results}
\label{app:prelim}

	\subsection{Generalised coercivity}
	\label{ssec:gen.coer}
	
	We first recall the following generalised representation theorems from \cite{Almog-2015-40}.

	\begin{theorem}[{\cite[Thm.\ 2.1]{Almog-2015-40}}] \label{thm:lm.0}
		Let $\cV$ be a Hilbert space and let $\mathbf a$ be a bounded sesquilinear form on $\cV$. Assume there exist $\Phi_1,\Phi_2 \in \mathcal{B}(\mathcal{V})$ and $m> 0$ such that for all $f \in \mathcal{V}$ we have 
		\begin{equation}\label{lm.gen.coev}
			\begin{aligned}
				\abs{\mathbf a(f,f)} + \abs{\mathbf a(\Phi_1f,f)} &\geq m \|f\|_\mathcal{V}^2 \,,
				\\
				\abs{\mathbf a(f,f)} + \abs{\mathbf a(f,\Phi_2f)} &\geq m \|f\|_\mathcal{V}^2 \,.
			\end{aligned}
		\end{equation}
		Then the operator
		\begin{equation}\label{T.hat}
			\widehat A \in \cB (\cV, \cV^*), \qquad ( \widehat A f, g )_{\cV^* \times \cV} := \mathbf a(f,g), \qquad f,g \in \cV,
		\end{equation}
		is boundedly invertible.
	\end{theorem}

	If $\cV \subset \cH$ is continuously embedded and dense in another Hilbert space $\cH$, then one can identify $\cH$ with its (anti-)dual $\cH^*$ in a standard way and consider
	\begin{equation}\label{Gelfand}
		\cH \ni f \equiv \langle f, \cdot \rangle_\cH \in \cH^*, \qquad \cV \subset \cH \equiv \cH^* \subset \cV^*.
	\end{equation}
	In the above Hilbert space triple, the operator $\widehat A$ in~\eqref{T.hat} naturally defines a maximal restriction in $\cH$, and the respective inverses are related by the formula
	\begin{equation}\label{res.A.hat}
		A^{-1} := I_{\Vv\to \Hh} \widehat A^{-1} (I_{\Vv\to \Hh})^{*}.
	\end{equation}
	Here $I_{\Vv\to \Hh}$ denotes the bounded embedding $\cV \hookrightarrow \cH$ and its adjoint is the (bounded) restriction operator $\cH \hookrightarrow \Vv^*$. Under suitable assumptions, $A$ is boundedly invertible in $\cH$.

	\begin{theorem}[{\cite[Thm.\ 2.2]{Almog-2015-40}}] \label{thm:lm.1}
		In addition to the assumptions of Theorem~{\rm\ref{thm:lm.0}}, let $\mathcal{V} \subset \cH$ be continuously embedded and dense in another Hilbert space $\cH$ and assume that $\Phi_1$, $\Phi_2$ extend to bounded operators on $\cH$. Then the operator in $\cH$ defined by 
		\begin{equation}\label{lm.op}
			\begin{aligned}
				\Dom (A) & := \big\{ f\in \cV \, : \, \exists~\eta_f \in\cH,  \,\,  \forall~g \in \cV,  \, \, \mathbf a (f, g) = \langle \eta_f, g \rangle_\cH \big\}, \\
				Af & := \eta_f,
			\end{aligned}
		\end{equation}
		is boundedly invertible and its domain is dense in $\cV$ and $\cH$.
	\end{theorem}

	To obtain the results for Schr\"odinger operators with accretive potentials $V$ in $L^2(\Omega)$ mentioned in the introduction, $\Phi_1$ and $\Phi_2$ can be selected as the multiplication operator
	\begin{equation}\label{Phi.AH}
		\Phi:= \Phi_1 = \Phi_2 =\frac{\im V}{\sqrt{1+|V|^2}}.
	\end{equation}

	\subsection{Schur complement dominant operator matrices}
	\label{ssec:Schur.dom}
	
	We recall claims from \cite{Gerhat-2024-286} which are relevant for the applications in Sections~\ref{ssec:ex.1}--\ref{ssec:dwe}. Employing suitable Gelfand triples, they allow us to introduce operator matrices with non-empty resolvent set in the product space $\cH := \cH_1 \oplus \cH_2$ of two complex Hilbert spaces $\cH_1$ and $\cH_2$.
	
	\begin{asm-sec} \label{asm:schur.dom}
		\begin{enumerate}[\upshape (i), wide]
			\item \label{item.spaces} Let $\Dd_S$, $\Dd_2$, $\Dd_{-S}$, and $\Dd_{-2}$ be complex Hilbert spaces such that, with continuous embeddings having dense ranges, 
			\begin{equation}\label{eq:triplets}
				\Dd_S \subset \Hh_1 \subset \Dd_{-S}, \qquad \Dd_2 \subset \Hh_2 \subset \Dd_{-2}.
			\end{equation}
			\item \label{item.entries.1} Suppose that
			\begin{equation}
				\begin{aligned}
					\widehat A & \in \Bb(\Dd_S, \Dd_{-S}), & \quad \widehat B & \in \Bb (\Dd_2, \Dd_{-S}), \\
					\widehat C & \in \Bb (\Dd_S, \Dd_{-2}), & \quad \widehat D  & \in \Bb (\Dd_2, \Dd_{-2}). \\
				\end{aligned}
			\end{equation}
		\end{enumerate}
	\end{asm-sec}

	Under Assumption~\ref{asm:schur.dom}, we define the operator matrix 
	\begin{equation}
		\widehat \Aa := \left(
		\begin{array}{cc}
			\widehat A & \widehat B \\
			\widehat C & \widehat D
		\end{array}
		\right) \in \Bb (\Dd, \Dd_-), \qquad
		\Dd  := \Dd_S \oplus \Dd_2, \qquad
		\Dd_-  := \Dd_{-S} \oplus \Dd_{-2}.
	\end{equation}
	Its first Schur complement
	\begin{equation}
		\label{eq:def.hat.S}
		\quad 
		\widehat S (\la) := \widehat A - \la- \widehat B (\widehat D-\la)^{-1} \widehat C \in \Bb (\Dd_S, \Dd_{-S})
	\end{equation}
	is defined for spectral parameters
	\begin{equation}\label{eq:hat.rho}
		\la \in \rho (\widehat D) := \big\{ z \in \C \, : \, (\widehat D- z)^{-1} \in \Bb (\Dd_{-2}, \Dd_{2}) \big\}.
	\end{equation}
	Finally, the corresponding (unbounded) maximal operators acting in $\Hh$ and $\Hh_1$ are 
	\begin{equation}\label{Aa.S.max}
		\Aa := \widehat \Aa\vert_{\Dom(\Aa)}, \qquad S (\la):= \widehat S(\la)\vert_{\Dom (S(\la))},
	\end{equation}
	on their respective domains
	\begin{equation}\label{eq:Schur.max.dom}
		\begin{aligned}
			\Dom (\Aa) & := \big\{ (f,g) \in \Dd \, : \, \widehat\Aa (f,g) \in \Hh \big\},\\
			\Dom (S(\la)) & := \big\{ f \in \Dd_S \, : \, \widehat S(\la)f \in \Hh_1 \big\}. 
		\end{aligned}
	\end{equation}
	The spectra of $\cA$ and $S(\cdot)$ are related by the following theorem.
	
	\begin{theorem}[{\cite[Cor.\ 3.4~(ii), Cor.\ 3.5, Cor.\ 3.6, Cor.\ 3.7]{Gerhat-2024-286}}]
		\label{thm:Schur.dom}
		Let Assumption~{\rm{\ref{asm:schur.dom}}} be satisfied and let $\widehat S(\cdot)$, $ \rho(\widehat D)$, $\Aa$ and $S(\cdot)$ be as in~\eqref{eq:def.hat.S},~\eqref{eq:hat.rho},~\eqref{Aa.S.max} and \eqref{eq:Schur.max.dom}. Let $\Sigma \subset \rho(\widehat D)$ be such that
		\begin{equation}\label{Schur.asm}
			\forall~\la \in \Sigma \quad \exists~z_\la \in \C \quad : \quad (\widehat S(\la) - z_\la)^{-1} \in \Bb (\Dd_{-S}, \Dd_S).
		\end{equation}
		Then the (point and second essential) spectra of $\cA$ and $S(\cdot)$ are equivalent on $\Sigma$. In detail,
		\begin{equation}
			\begin{aligned}
				& \la \in \sigma(\cA) \,\,  &&  \iff \quad 0 \in \sigma (S (\la)), \\
				& \la \in \spp(\cA) \,\, && \iff \quad 0 \in \spp (S (\la)), \\
				& \la \in \sigma_{\operatorname{e2}}(\cA) \,\, &&  \iff \quad 0 \in \sigma_{\operatorname{e2}} (S (\la)),
			\end{aligned}
		\end{equation}
		for all $\la \in \Sigma$; see~\cite[Chap.~IX]{EE} for the definition of $\sigma_{{\rm e}2}(\cdot)$. Moreover, if there exists $\la \in \Sigma \neq \emptyset$ with $0 \in \rho(S(\la))$, then $\Dom (\Aa)$ is dense in both $\Dd$ and $\cH$.
\end{theorem}

The following useful lemma establishes an equivalence between invertibility of a distributional operator in a Gelfand triple and the associated operator.

\begin{lemma}
	\label{lem:res.id}
	Let $\Dd \subset \Hh \subset \Dd_-$ be a triple of Hilbert spaces where the respective embeddings are continuous and have dense range. Let $\widehat T \in \Bb (\Dd, \Dd_-)$ and define
	\begin{equation}
		T := \widehat T|_{\Dom (T)}, \qquad \Dom (T) := \big\{ f \in \Dd \, : \, \widehat T f \in \Hh \big\}.
	\end{equation}
	Assume there exists $\la_0 \in \C$ such that
	\begin{equation}
		(\widehat T-\la_0 )^{-1} \in \Bb(\Dd_-,\Dd).
	\end{equation}
	Then, for all $\la \in \C$, cf.~\eqref{eq:hat.rho},
	\begin{equation}
		\la \in \rho(T) \quad \iff \quad \la \in \rho(\widehat T) \quad : \iff \quad (\widehat T-\la)^{-1} \in \Bb(\Dd_-,\Dd).
	\end{equation}
\end{lemma}

\begin{proof}
	If $\la \in \rho(T)$, then $\la \in \rho(\widehat T)$ by the analogous formula to~\eqref{res.A.hat}. To see the other implication, observe first that $\Dom(T)$ is dense in $\Dd$; this follows from
	\begin{equation}
		\Dom (T) =  (\widehat T-\la_0)^{-1} \Hh,
	\end{equation}
	the fact that $\widehat T-\la_0$ is an isomorphism between $\Dd$ and $\Dd_-$ and the density of $\Hh $ in $\Dd_-$. Next, using the resolvent identity, we have
	\begin{equation}
		(T - \lambda)^{-1} = (T - \lambda_0)^{-1} + (\lambda - \lambda_0) (T - \lambda)^{-1} (T - \lambda_0)^{-1}  \subset R,
	\end{equation}
	where the extension $R$ is constructed as
	\begin{equation}
		R := (\widehat T-\lambda_0)^{-1} + (\lambda - \lambda_0) (T - \lambda)^{-1} (\widehat T-\lambda_0)^{-1} \in \cB (\Dd_-, \Dd);
	\end{equation}
	notice hereby that 
	\begin{equation}
		(T-\la)^{-1} \in \Bb(\Hh, \Dd),
	\end{equation}
	cf.~\cite[Lem.~2.12]{Gerhat-2024-286}. Since the respective compositions
	\begin{equation}
		R(\widehat T-\lambda ) \in \cB (\Dd), \qquad (\widehat T-\lambda ) R \in \cB (\Dd_-),
	\end{equation}
	are equal to $I_{\Dd}$ and $I_{\Dd_-}$  on the dense subspaces $\Dom (T) \subset \Dd$   and $\Hh \subset \Dd_-$,  we infer, cf.~\cite[Lem.~2.13]{Gerhat-2024-286}, 
	\begin{equation*}
		R = (\widehat T-\lambda )^{-1}. \qedhere
	\end{equation*}
\end{proof}

\subsection{Schatten class and completeness of eigensystem}

\begin{theorem}[{\cite[Thm.~1.3]{Almog-2015-40}, \cite{Combes-1978-111}}]\label{thm:Sp}
	Let $\Omega \subset \Rd$ be open with $\partial \Omega \in C^{2,\alpha}$ for some $\alpha>0$. Let $ 0 \le Q \in L^2_{\rm loc}(\Omega)$
	and for $p>0$
	\begin{equation}\label{Sp.crit}
		\int_{\Omega \times \R^d} (|\xi|^2 + Q(x) + 1)^{-p} \, \dd x \, \dd \xi < \infty.
	\end{equation}
	Then for the self-adjoint Dirichlet realisation $ S:=-\Delta + Q$ in $L^2(\Omega)$ we have 
	\begin{equation}
		(S+1)^{-1} \in \cS_p(L^2(\Omega)).
	\end{equation}
\end{theorem}

\begin{theorem}[{\cite[Thm.~XIII.81]{Reed4}}] \label{thm:Sp.inf}
For $d \geq 2$, let $Q : \Rd \to \R$ be measurable and obey (a.e.~in $\Rd$)
\begin{equation}
\begin{aligned}	
c_1 (|x|^\beta -1) & \leq Q(x) \leq c_2 (|x|^\beta + 1),
\\
|Q(x)-Q(y)| & \leq c_3 \max \{ |x|, |y| \}^{\beta-1} |x-y|,
\end{aligned}	
\end{equation}
for some $\beta>1$ and constants $c_1, c_2, c_3 >0$. Let $\wt S= - \Delta + Q$ be the self-adjoint realisation in $L^2(\Rd)$ and let
\begin{equation}
	N(\la) = \# \{ \la_k \in \sigma(\wt S) \, : \, \la_k <\la   \}, \qquad \la > 0,
\end{equation}
where the eigenvalues $\la_k$, $k \in \N$, are ordered in a non-decreasing way and repeated according to their multiplicity. Then
\begin{equation}
N(\la) = \frac{|B_1(0)|}{(2 \pi)^d} \int_{\{Q(x) \leq \la \}} (\la-Q(x))^\frac d2 \, \dd x \, \big(1+o(1)\big), \qquad \la \to + \infty.
\end{equation}
\end{theorem}

\begin{theorem}[{\cite[Cor.~XI.9.31]{DS2}}]
	\label{thm:comp.eigensys}
	Let $T$ be a densely defined unbounded operator in a Hilbert space $\Hh$ and assume that there is  $p>0$ with
	\begin{equation}
		(T-\la)^{-1} \in \Ss_p (\Hh), \qquad \la \in\rho(T).
	\end{equation}
	If there exist $\varphi_1 , \dotsc , \varphi_k \in (-\pi,\pi]$ and $N \in \N_0$ such that the angle between two adjacent rays $\{ r\e^{\ii \varphi_j}\, : \, r>0\}$ is at most $\pi/p$ and
	\begin{equation}
		\| (T-r \e^{\ii \varphi_j})^{-1}\| \ls r^N , \qquad r \to \infty,
	\end{equation}
for all $j = 1, \dotsc, k$, then the eigensystem of $T$ is complete, i.e.~the space of root functions (generalised eigenfunctions) of $T$ is dense in $\Hh$.
\end{theorem}

{\footnotesize
\bibliographystyle{acm}
\bibliography{references}
}

\end{document}